\documentclass[11pt]{article}

\usepackage[a4paper,margin=1in]{geometry}
\usepackage{amsmath,amssymb,amsthm,mathtools}
\usepackage{microtype}
\usepackage{xcolor}
\usepackage[colorlinks=true,linkcolor=blue,citecolor=blue,urlcolor=blue,anchorcolor=blue]{hyperref}
\hypersetup{
  pdftitle={Single-Measurement Recovery of a Spatial Coefficient in the
  Principal Memory Term of a Wave Equation},
  pdfauthor={},
  pdfsubject={Inverse coefficient problem for a wave equation with memory},
  pdfkeywords={inverse coefficient problem, wave equation with memory,
  single lateral Cauchy measurement, Carleman estimate, conditional Holder stability}
}

\newtheorem{theorem}{Theorem}[section]
\newtheorem{proposition}[theorem]{Proposition}
\newtheorem{lemma}[theorem]{Lemma}
\theoremstyle{remark}
\newtheorem{remark}[theorem]{Remark}
\numberwithin{equation}{section}

\allowdisplaybreaks

\newcommand{\R}{\mathbb R}
\newcommand{\Sph}{\mathbb S}
\newcommand{\supp}{\operatorname{supp}}
\newcommand{\diver}{\operatorname{div}}
\newcommand{\dist}{\operatorname{dist}}
\newcommand{\dd}{\,\mathrm d}
\newcommand{\norm}[1]{\left\lVert #1\right\rVert}
\newcommand{\abs}[1]{\left\lvert #1\right\rvert}
\newcommand{\ip}[2]{\left\langle #1,#2\right\rangle}

\newcommand{\Qad}{\mathcal Q^{(5)}_{M,c_0}(D)}
\newcommand{\XT}{\mathcal X_T}
\newcommand{\Dt}{\mathcal D_T}
\newcommand{\Mdir}{\mathcal M_*}

 \usepackage{xstring} 
\newcommand{\AuthorInfo}[4]{%
	\textsc{#1}%
	\IfStrEq{#4}{true}{$^{*}$}{}\\
	#2%
	\IfStrEq{#3}{}{}{\\\textit{E-mail:} \texttt{#3}}%
	\IfStrEq{#4}{true}{\\\textit{$^{*}$Corresponding author.}}{}
	\par\vspace{0.5em}
}

\title{Single-Measurement H\"older Stability for a Spatial Coefficient in a Wave Equation with Principal Memory}

\author{}
\date{}

\begin{document}
\maketitle
	
	\begin{center}
		\footnotesize
		\AuthorInfo{Huaian Diao}{School of Mathematics and Key Laboratory of Symbolic Computation and Knowledge Engineering of Ministry of Education, Jilin University, Changchun 130012, China}{diao@jlu.edu.cn, hadiao@gmail.com}{false}

						\AuthorInfo{Umidjon Durdiev}{\textsuperscript{1}School of Mathematics,
    Jilin University, Changchun 130012, China\\
    \textsuperscript{2}Bukhara State University, Uzbekistan\\
    \textsuperscript{3}V.~I.~Romanovskiy Institute of Mathematics,
    Uzbekistan Academy of Science
}{umidjan93@list.ru}{false}
\end{center}

	\normalsize

\begin{abstract}
We study the recovery of a spatial coefficient in the principal memory flux of a three-dimensional wave equation with the prescribed exponential relaxation kernel \(\beta e^{-\beta t}\), where \(\beta>0\) is known. From a single finite-time lateral Cauchy measurement generated by one fixed, coefficient-independent initial displacement that is affine near the unknown region, with zero initial velocity, we establish global uniqueness and conditional \(L^2\)-H\"older stability. Two difficulties arise in this setting. The affine initial displacement makes the initial acceleration vanish in the unknown region, while the coefficient-difference equation contains a divergence-form source whose direct \(L^2\) estimate loses a spatial derivative. We establish an acoustic Carleman estimate that controls the vector source and its time derivative without spatially differentiating the source. The affine condition identifies the relevant initial trace with a directional derivative of the coefficient difference, and the quadratic temporal vanishing of the source permits absorption in the weighted estimate. A transport estimate then yields stability uniformly over a bounded, uniformly coercive \(W^{5,\infty}\) coefficient class, without an ordering condition or a smallness assumption near a reference coefficient.
\end{abstract}



\tableofcontents

\section{Introduction}

Wave propagation in viscoelastic and viscoacoustic media is often modeled by hyperbolic integro-differential equations, in which the present response depends on the history of the wave field through a relaxation law; see, for example, \cite{FabrizioMorro1992,CoviDeHoopSalo2026}. In the scalar model considered here, fading memory acts on the principal spatial part of the wave equation through an exponential relaxation kernel with a spatially varying amplitude. The inverse problem is to recover this amplitude from lateral boundary measurements outside its support. In particular, the unknown coefficient enters the principal memory flux rather than a lower-order term or an external source.

Let \(D\Subset\Omega\Subset\R^3\) be bounded domains with smooth boundaries, and fix \(\beta>0\). We consider the following whole-space initial-value problem for a scalar wave field \(u=u(x,t)\):
\begin{equation}\label{eq:memory}
\begin{cases}
 u_{tt}-\Delta u+
 \displaystyle\int_0^t
 \beta e^{-\beta(t-\tau)}
 \diver(q(x)\nabla u(x,\tau))\dd\tau=0,
 &(x,t)\in\R^3\times(0,T),\\[1mm]
 u(x,0)=f(x),\qquad u_t(x,0)=0,
 &x\in\R^3.
\end{cases}
\end{equation}
Here \(f\) is the prescribed initial excitation, \(\beta>0\) is the fixed, known relaxation rate, and \(q=q(x)\) is the unknown spatial coefficient of the principal memory term. We assume that \(q\) is supported in the known subdomain \(D\), so the medium is memory-free outside \(D\). The boundary \(\partial\Omega\) encloses \(D\) and serves as the measurement surface. Since wave propagation takes place in \(\R^3\), no boundary condition is imposed on \(\partial\Omega\).

The measurement is taken on the observation cylinder
\begin{equation}\label{eq:observation-cylinder}
 \Sigma_T=\partial\Omega\times(0,T).
\end{equation}
For each admissible coefficient \(q\), let \(u_q\) be the solution of
\eqref{eq:memory}. The original lateral Cauchy measurement is
\begin{equation}\label{eq:original-cauchy-data}
 \mathcal C_T(q)
 =
 \left(
 u_q\big|_{\Sigma_T},
 \,
 \partial_\nu u_q\big|_{\Sigma_T}
 \right),
\end{equation}
where \(\nu\) is the outward unit normal to \(\Omega\). The inverse
problem is to determine \(q\) by the map
\begin{equation}\label{eq:inverse-problem-map}
 q\longmapsto\mathcal C_T(q)
\end{equation}
from the
single boundary measurement \(\mathcal C_T(q)\) generated by one prescribed
initial displacement \(f\) together with the zero initial velocity in
\eqref{eq:memory}. Throughout this paper, the initial excitation \(f\) is fixed independently of \(q\) and is affine in a neighborhood of \(\overline D\), as specified in \eqref{eq:affine-f}; see Subsection~\ref{Subsec:single} for further details.  
  The admissible
coefficient class and the measurement norms are defined in
Section~\ref{sec:setting}.

\subsection*{Related work}
The recovery of spatial factors in viscoelastic memory kernels has been
investigated under different choices of initial data and boundary
observations. The identification of Lam\'e memory kernels is
considered in \cite{LorenziMessinaRomanov2007,LorenziRomanov2011}. Conditional H\"older stability from a single lateral Cauchy measurement is obtained in
\cite{RomanovYamamoto2010} for a two-dimensional scalar reduction. In that setting, the prescribed initial data satisfy \(u_t(0)=-k'(0)u(0)\), where \(k\) is the temporal factor. For the exponential kernel \(k(t)=e^{-\beta t}\), this gives \(u_t(0)=\beta u(0)\), and hence excludes the nontrivial zero-velocity excitation considered here. In three dimensions, logarithmic stability for the spatial factor in a separable viscoelastic memory coefficient from a single boundary measurement is proved in \cite{deBuhanOsses2010}.   An \(H^1\)-Lipschitz estimate for a
two-dimensional scalar model with a plane-impulse excitation is established
in \cite{Romanov2012}, under additional smallness and geometric assumptions.


A different observation setting arises when the prescribed excitation can be varied. For example, a restricted Dirichlet-to-Neumann map is considered in \cite{Janno2004}, while \cite{CoviDeHoopSalo2026} studies the viscoacoustic equation with principal memory using an exterior source-to-solution map. The studies \cite{RomanovAcoustic2026,RomanovRay2026} use families of measurements generated by point-source excitations or incident plane impulses.  These measurements contain information generated by more than one prescribed excitation. By contrast, the inverse problem \eqref{eq:inverse-problem-map} considered here uses the lateral Cauchy traces of the solution corresponding to a single initial displacement \(f\), fixed independently of \(q\), and zero initial velocity. Both traces are observed on \(\Sigma_T\), but no additional excitation is introduced.

Inverse source problems for equations with prescribed memory kernels are
studied, for example, in
\cite{CavaterraLorenziYamamoto2006,LoretiSforzaYamamoto2017,
ImanuvilovYamamoto2020}. In the inverse source problems cited above, the unknown is an external source
amplitude, whereas in \eqref{eq:memory} the coefficient to be recovered
belongs to the principal memory flux. This distinction is important: after
two coefficients are compared, the resulting source contains spatial
derivatives of their difference, leading to the derivative-loss mechanism
described above. Unique continuation for hyperbolic equations with memory is
studied in \cite{BukhgeimDyatlovUhlmann2007}.

Applying \(\partial_t+\beta\) to \eqref{eq:memory} yields the local
third-order equation \eqref{eq:Lq}, which is of
Moore--Gibson--Thompson type with a spatially varying principal coefficient.
Related direct and inverse problems for third-order acoustic equations are
studied in \cite{LiuTriggiani2013,ArancibiaEtAl2022,FuChenZhang2024}.
In particular, \cite[Theorem~2.1]{ArancibiaEtAl2022} recovers a
coefficient multiplying \(u_{tt}\) under a nonvanishing
initial-acceleration condition. Here \(u_t(0)=0\) and \(u_{tt}(0)=\Delta f=0\) in \(D\)
by \eqref{eq:affine-f} and \eqref{eq:forward-initial-traces}.
The nonvanishing initial-acceleration condition therefore fails for our
excitation. For wave equations without memory,
single-measurement coefficient recovery is treated in
\cite{ImanuvilovYamamoto2003,KlibanovYamamoto2006,
BellassouedYamamoto2008}. The divergence-form setting of
\cite{KlibanovYamamoto2006} similarly produces spatial derivatives
of a coefficient difference. The fixed-initial-field result of
\cite[Theorem~3.3]{StefanovUhlmann2013} assumes \(\Delta f\ne0\) on the
coefficient perturbation region, whereas the excitation used here is
harmonic near \(\overline D\).

\subsection*{Stability result and proof strategy} 
Theorem~\ref{thm:main} establishes global uniqueness and conditional
H\"older stability in \(L^2(D)\) from the single measurement
\eqref{eq:original-cauchy-data}. The estimate is uniform for pairs
\(q_1,q_2\in\Qad\); it imposes neither a sign condition on
\(q_1-q_2\) nor a smallness condition around a reference coefficient.
The finite observation time is subject to the geometric condition
\eqref{eq:long-time}. The filtered boundary data used in the proof are
obtained from \eqref{eq:original-cauchy-data} by the fixed temporal
operator \eqref{eq:Dbeta}, without an additional measurement.

The proof addresses two difficulties arising from the principal memory coefficient and the prescribed initial excitation. First, for the coefficient difference \(r=q_1-q_2\) defined in \eqref{eq:difference-variables}, the difference equation \eqref{eq:w-eq} contains the source \(-\beta\diver(r\nabla u_{q_2})\). A direct \(L^2\) estimate of this term involves \(\nabla r\), whereas Theorem~\ref{thm:main} estimates \(r\) in \(L^2(D)\). Second, the affine condition \eqref{eq:affine-f} and the zero initial velocity imply \(u_{q,tt}(\cdot,0)=0\) in \(D\), so the coefficient cannot be recovered through a nonvanishing initial-acceleration condition. Instead, the initial trace of the differentiated difference field satisfies \eqref{eq:h-direction}, which identifies a directional derivative of \(r\).

To control the divergence-form source without differentiating \(r\) spatially, we convert \eqref{eq:memory} into the third-order equation \eqref{eq:Lq} and establish an acoustic Carleman estimate in Proposition~\ref{prop:acoustic}. The estimate involves the vector source and its time derivative, but not their spatial derivatives. Its proof combines a scalar Carleman inequality, parameter-dependent Bessel regularization, a Volterra estimate, and a frequency decomposition. Proposition~\ref{prop:initial-trace} then controls the initial source \(h\). The affine condition \eqref{eq:affine-f} and the forward initial traces \eqref{eq:forward-initial-traces} also yield the quadratic temporal vanishing \eqref{eq:time-vanish}, leading to the coefficient-source estimate \eqref{eq:Fcoef}. Together with the transport estimate \eqref{eq:transport}, this allows the coefficient-dependent term to be absorbed by choosing the Carleman parameter sufficiently large, without requiring \(q_1-q_2\) to be small. The weight gap \eqref{eq:weight-gap} and the subsequent optimization give the conditional H\"older stability estimate \eqref{eq:main-stab}.

The paper is structured as follows. Section~\ref{sec:setting} specifies the coefficients, measurements, and
Carleman geometry. Section~\ref{sec:direct-proof} proves uniform
well-posedness of the forward problem. Section~\ref{sec:localization-reduction}
constructs the localized difference equation and bounds its cutoff terms
by the measurement discrepancy. Section~\ref{sec:acoustic} proves the
acoustic and initial-trace estimates, and Section~\ref{sec:nonlinear}
establishes the coefficient-source and transport bounds. The proof of
Theorem~\ref{thm:main} is completed in Section~\ref{sec:main-proof}.

\paragraph{Notation.}
For a multi-index
\(\alpha=(\alpha_1,\alpha_2,\alpha_3)\in\{0,1,2,\ldots\}^3\), write
\(|\alpha|=\alpha_1+\alpha_2+\alpha_3\) and
\(D^\alpha=\partial_{x_1}^{\alpha_1}
\partial_{x_2}^{\alpha_2}\partial_{x_3}^{\alpha_3}\).
For time derivatives, write \(u^{(j)}=\partial_t^j u\), with the same
convention for other fields. For a space--time field \(b=b(x,t)\), we write
\(b(t)=b(\cdot,t)\) when treating \(b\) as a function with values in a
spatial function space. Accordingly, \(b(0)\) and \(b_t(0)\) denote
spatial initial traces whenever they exist. The same convention applies
to vector-valued fields, Fourier approximations, and boundary traces,
with their respective spatial domains. Spatial Sobolev spaces and norms with no
specified domain are taken over \(\R^3\). Unless otherwise stated,
coefficients, initial data, solutions, and source fields are real-valued.
We write \(\partial_\omega=\omega\cdot\nabla_x\) for the directional
derivative and \([B,C]=BC-CB\) for the commutator of two operators.
The hat denotes the Fourier transform in the spatial variables,
\(\mathsf D=-i\nabla_x\), and \(m(\mathsf D)\) denotes the Fourier multiplier
with symbol \(m(\xi)\). The space \(C_w([0,T];H^m)\) denotes the space of
weakly continuous functions from \([0,T]\) into \(H^m\).

Throughout the paper, \(C>0\) denotes a generic constant whose value may
change from line to line. Unless otherwise stated, \(C\) depends only on the
fixed a priori parameters and is independent of the coefficients being
compared, the Carleman parameter, and the measurement discrepancy. Any
additional dependence is indicated explicitly by a subscript.

\section{Problem setting and main result}\label{sec:setting}

This section makes precise the analytical setting of the inverse problem and
states the main stability result. We first introduce the admissible
coefficient class and the uniform well-posedness properties of the forward
problem. We then specify the single boundary measurement generated by the
prescribed initial excitation. Finally, we introduce the Carleman geometry
and the observation-time condition under which the stability estimate is
proved.

Let us first introduce the admissible coefficient class used throughout the
forward and inverse analysis. The support, regularity, and coercivity
assumptions imposed below yield uniform a priori bounds of the corresponding
forward solutions and of the constants appearing in the stability argument.

Fix \(M>0\) and \(c_0\in(0,1)\), and define
\begin{equation}\label{eq:Qad}
\Qad
=\left\{
 q\in W^{5,\infty}(\R^3):
 \supp q\Subset D,\quad
 0\le q\le1-c_0,\quad
 \norm{q}_{W^{5,\infty}(\R^3)}\le M
\right\}.
\end{equation}
The support condition localizes the unknown coefficient $q$ to the interior of
the known region \(D\). The bounds \(0\le q\le1-c_0\) ensure the
nonnegativity of the memory contribution and the uniform coercivity
\(1-q\ge c_0\) of the relaxed principal spatial operator. The
\(W^{5,\infty}\)-bound provides the coefficient regularity required by the
high-order forward estimates in Proposition~\ref{prop:direct}; in
particular, fifth-order derivatives of \(q\) enter only through the
top-order spatial commutators in its proof.

\begin{remark}\label{rem:coefficient-regularity}
The regularity \(q\in W^{5,\infty}(\R^3)\) is chosen to support the
strong-solution framework and the high-order energy estimates used in the
forward problem. The highest derivatives of \(q\) arise only in the
top-order spatial commutators. Lowering
the regularity of the admissible coefficients would require a corresponding
refinement of the forward problem analysis and is not pursued in this paper.
\end{remark}

Solutions of the forward problem \eqref{eq:memory} will be considered in the
strong-solution space
\begin{equation}\label{eq:Xclass}
\XT
=L^\infty(0,T;H^6(\R^3))
\cap W^{1,\infty}(0,T;H^5(\R^3))
\cap W^{3,\infty}(0,T;H^3(\R^3))
\cap W^{5,\infty}(0,T;H^1(\R^3)),
\end{equation}
equipped with the sum of the norms of its four constituent spaces.

The following proposition gives the uniform well-posedness and regularity estimates needed for the subsequent inverse problem. In particular, the regularity estimate ensures that the filtered observation
operator introduced in \eqref{eq:observation-map} is well defined with values
in the data space \eqref{eq:data-space}, provides the temporal regularity
needed to derive the localized difference equation in
Section~\ref{sec:localization}, and supplies the forward-solution bounds used
in the nonlinear source estimates of Section~\ref{sec:nonlinear}.

\begin{proposition}\label{prop:direct}
For every \(q\in\Qad\) and \(f\in H^6(\R^3)\), problem
\eqref{eq:memory} has a unique solution \(u_q\in\XT\). Moreover,
\begin{equation}\label{eq:direct-est}
 \norm{u_q}_{\XT}
 \le C_{\rm dir}\norm{f}_{H^6(\R^3)},
 \qquad
 C_{\rm dir}=C_{\rm dir}(T,\beta,M,c_0)>0,
\end{equation}
uniformly for \(q\in\Qad\).
\end{proposition}

Proposition~\ref{prop:direct} is proved in
Section~\ref{sec:direct-proof}; we first introduce necessary setup for the inverse problem and show the main results in the next subsection. 


\subsection{The single-measurement setting}\label{Subsec:single}

The original lateral Cauchy measurement \(\mathcal C_T(q)\), defined in
\eqref{eq:original-cauchy-data}, is generated by the same initial
excitation \(f\) for every \(q\in\Qad\). Its temporally filtered
counterpart \(\Phi_T(q)\) is defined in \eqref{eq:observation-map}.
Both measurements are taken on the observation cylinder \(\Sigma_T\)
defined in \eqref{eq:observation-cylinder}. The relaxation rate \(\beta\)
in \eqref{eq:memory} is prescribed and remains fixed when comparing
coefficients \(q_1,q_2\in\Qad\).

\paragraph{Choice of the initial excitation.}
Fix a unit vector \(\omega\in\Sph^2\) and \(b_0\in\R\). Choose
\(f\in C_c^\infty(\Omega)\) such that
\begin{equation}\label{eq:affine-f}
 f(x)=\omega\cdot x+b_0
 \quad\text{in a neighborhood of }\overline D.
\end{equation}
Such a function is obtained by multiplying the affine function
\(x\mapsto\omega\cdot x+b_0\) by a smooth cutoff supported in \(\Omega\)
and equal to one near \(\overline D\). The affine identity \eqref{eq:affine-f} gives
\(\nabla f=\omega\) and \(\Delta f=0\) near \(\overline D\). In the
coefficient-to-source reduction of Section~\ref{sec:localization}, the first
identity reduces the coefficient-dependent initial source to
\[
h=-\beta\partial_\omega(q_1-q_2);
\]
see \eqref{eq:h-direction}. Proposition~\ref{prop:initial-trace} provides a
weighted estimate for \(h\), and Lemma~\ref{lem:transport} then converts the
resulting control of the directional derivative into control of the
coefficient difference itself. The second identity, \(\Delta f=0\), combines
with the zero initial velocity to produce the temporal vanishing used in
\eqref{eq:time-vanish} and the nonlinear source estimate
\eqref{eq:Fcoef}. By contrast, the compact support condition
\(f\in C_c^\infty(\Omega)\) ensures the boundary compatibility
\eqref{eq:trace-compat} required for the filtered measurement operator
\(\Phi_T\) in \eqref{eq:observation-map} and the data space \(Y_T\) in
\eqref{eq:data-space}.

\begin{remark}\label{rem:initial-excitation}
Structural conditions on a fixed initial excitation are common in
single-measurement coefficient inverse problems. For example, the
single-measurement stability result in \cite{RomanovYamamoto2010} imposes a
compatibility condition between the prescribed initial displacement and
velocity, together with a nondegeneracy condition coupling the gradient of
the initial displacement to the Carleman weight. Likewise, the fixed-initial-
field recovery result in \cite[Theorem~3.3]{StefanovUhlmann2013} assumes
\(\Delta f\neq0\) on the support of the coefficient perturbation. The
mechanism used here is different: the affine excitation satisfies
\(\Delta f=0\) near \(\overline D\), while its constant gradient
\(\nabla f=\omega\) encodes the coefficient difference through the
directional derivative \(\partial_\omega(q_1-q_2)\). Thus
\eqref{eq:affine-f} is tailored to the present single-measurement argument;
no claim is made that this choice is necessary or optimal.
\end{remark}

\paragraph{Filtered lateral data.}

For the subsequent analysis,  a temporally filtered version of the lateral Cauchy traces \eqref{eq:original-cauchy-data}
will be introduced. The filtering is determined
by the exponential relaxation law and is obtained by applying the known
temporal operator \(D_\beta\) defined by 
\begin{equation}\label{eq:Dbeta}
 D_\beta
 =\partial_t^2+\beta\partial_t
 =\partial_t(\partial_t+\beta);
\end{equation}
it therefore
amounts to a post-processing of the original measurement
\eqref{eq:original-cauchy-data}, rather than an additional boundary
measurement.   The operator \(D_\beta\) is naturally associated with the exponential
relaxation law in \eqref{eq:memory}: applying \(\partial_t+\beta\) eliminates
the memory convolution and leads to the local equation used in
Section~\ref{sec:localization}. In particular, the boundary quantities that
enter the subsequent localization and Carleman analysis are the traces of
\(D_\beta u_q\) and of its normal derivative.  Accordingly, define the filtered measurement operator
\begin{equation}\label{eq:observation-map}
 \Phi_T(q)
 =
 \left(
 D_\beta u_q\big|_{\Sigma_T},
 \,
 \partial_\nu D_\beta u_q\big|_{\Sigma_T}
 \right),
\end{equation}
where the observation cylinder \(\Sigma_T\) is introduced in
\eqref{eq:observation-cylinder}.

To specify the temporal regularity of the measured traces, let \(X\) be a
Hilbert space and \(m\ge1\) an integer, and define
\begin{equation}\label{eq:causal-Sobolev-space}
 W_+^m(\R;X)
 =\left\{
 z\in W^{m,2}(\R;X):
 \supp z\subset[0,\infty)
 \right\}.
\end{equation}
Its restriction to \((0,T)\) is denoted by
\[
 W_+^m(0,T;X)
 =\left\{
 z|_{(0,T)}:\ z\in W_+^m(\R;X)
 \right\},
\]
and is equipped with the quotient norm
\begin{equation}\label{eq:causal-restriction-norm}
 \|u\|_{W_+^m(0,T;X)}
 =
 \inf_{\substack{
 z\in W_+^m(\R;X)\\
 z|_{(0,T)}=u}}
 \|z\|_{W^{m,2}(\R;X)}.
\end{equation}
The measurement space is then
\begin{equation}\label{eq:data-space}
 Y_T
 =
 W_+^2(0,T;H^{1/2}(\partial\Omega))
 \times
 W_+^1(0,T;H^{-1/2}(\partial\Omega)).
\end{equation}

The regularity furnished by Proposition~\ref{prop:direct} implies that the
filtered measurement operator in \eqref{eq:observation-map} is well defined
on the admissible coefficient class. The precise statement is as follows.

\begin{proposition}
\label{prop:measurement}
For every \(q\in\Qad\), the filtered measurement
\(\Phi_T(q)\) defined in \eqref{eq:observation-map} belongs to the data space
\(Y_T\) defined in \eqref{eq:data-space}. Moreover, there exists a constant
\(C>0\), depending only on the fixed a priori parameters, such that
\begin{equation}\label{eq:measurement-bound}
 \|\Phi_T(q)\|_{Y_T}
 \le
 C\|u_q\|_{\XT}
 \le
 C C_{\rm dir}\|f\|_{H^6(\R^3)},
 \qquad q\in\Qad.
\end{equation}
In particular,
\[
 \Phi_T:\Qad\longrightarrow Y_T
\]
is well defined and uniformly bounded on \(\Qad\).
\end{proposition}

\begin{proof}
By Proposition~\ref{prop:direct}, we have
\begin{equation}\label{eq:filtered-regularity}
 D_\beta u_q
 \in
 W^{2,\infty}(0,T;H^1(\R^3))
 \cap
 W^{1,\infty}(0,T;H^3(\R^3)).
\end{equation}
The standard trace estimates imply
\begin{equation}\label{eq:trace-bounds}
 \|v|_{\partial\Omega}\|_{H^{1/2}(\partial\Omega)}
 \le C_\Omega\|v\|_{H^1(\R^3)},
 \qquad
 \|\partial_\nu v\|_{H^{3/2}(\partial\Omega)}
 \le C_\Omega\|v\|_{H^3(\R^3)}.
\end{equation}
For the two components of the filtered measurement
\eqref{eq:observation-map}, write
\[
 g_q=D_\beta u_q\big|_{\partial\Omega},
 \qquad
 k_q=\partial_\nu D_\beta u_q\big|_{\partial\Omega}.
\]
The trace bounds \eqref{eq:trace-bounds} and the regularity
\eqref{eq:filtered-regularity} imply
\[
 \begin{aligned}
 g_q&\in W^{2,\infty}(0,T;H^{1/2}(\partial\Omega)),\\
 k_q&\in W^{1,\infty}(0,T;H^{3/2}(\partial\Omega))
 \hookrightarrow W^{1,\infty}(0,T;H^{-1/2}(\partial\Omega)).
 \end{aligned}
\]

By Proposition~\ref{prop:direct}, the time regularity encoded in
\(\XT\) yields continuous representatives of \(u_{q,tt}\) in
\(H^3(\R^3)\) and of \(u_{q,ttt}\) in \(H^1(\R^3)\) at \(t=0\).
Hence \eqref{eq:memory} and its first time derivative may be evaluated at
\(t=0\). Using the initial conditions in \eqref{eq:memory}, we obtain
\begin{equation}\label{eq:forward-initial-traces}
 u_{q,tt}(\cdot,0)=\Delta f,
 \qquad
 u_{q,ttt}(\cdot,0)=-\beta\diver(q\nabla f).
\end{equation}
Consequently,
\[
 D_\beta u_q(\cdot,0)=\Delta f,
 \qquad
 \partial_tD_\beta u_q(\cdot,0)
 =\beta\bigl(\Delta f-\diver(q\nabla f)\bigr).
\]
Since \(f\in C_c^\infty(\Omega)\) and
\(\supp q\Subset D\Subset\Omega\), both \(\Delta f\) and
\(\diver(q\nabla f)\) vanish in a neighborhood of
\(\partial\Omega\). Taking the traces in the preceding identities gives
\begin{equation}\label{eq:trace-compat}
 g_q(\cdot,0)=\partial_tg_q(\cdot,0)=0,
 \qquad
 k_q(\cdot,0)=0,
\end{equation}
in \(H^{1/2}(\partial\Omega)\) and
\(H^{-1/2}(\partial\Omega)\), respectively.

For completeness, if \(X\) is a Hilbert space, \(m\in\{1,2\}\), and
\(z\in W^{m,2}(0,T;X)\) has
\(\partial_t^jz(0)=0\) for \(0\le j<m\), then its zero extension to
\(t<0\), followed by a bounded Sobolev extension across \(t=T\) and a
smooth cutoff for large positive time, belongs to \(W_+^m(\R;X)\).
The quotient norm in \eqref{eq:causal-restriction-norm} consequently obeys
\[
 \norm{z}_{W_+^m(0,T;X)}
 \le C_{m,T}\norm{z}_{W^{m,2}(0,T;X)}.
\]
Applying this extension estimate with \(m=2\) to \(g_q\) and with
\(m=1\) to \(k_q\), using \eqref{eq:trace-compat}, yields
\[
 \begin{aligned}
 &\norm{g_q}_{W_+^2(0,T;H^{1/2}(\partial\Omega))}
 +\norm{k_q}_{W_+^1(0,T;H^{-1/2}(\partial\Omega))}\\
 &\qquad\le C_{T,\Omega,\beta}\norm{u_q}_{\XT}.
 \end{aligned}
\]
This establishes \(\Phi_T(q)\in Y_T\). Proposition~\ref{prop:direct}
then gives the uniform bound \eqref{eq:measurement-bound}.
\end{proof}

\begin{remark}
For \(q_1,q_2\in\Qad\), the discrepancy between the corresponding filtered
measurements is defined by
\begin{equation}\label{eq:data-discrepancy}
 \Dt
 =
 \|\Phi_T(q_1)-\Phi_T(q_2)\|_{Y_T}.
\end{equation}

To compare this quantity with the unfiltered lateral Cauchy data, set
\(w=u_{q_1}-u_{q_2}\) and define
\begin{equation}\label{eq:unfiltered-data-discrepancy}
\begin{aligned}
 \mathcal E_T
 ={}&
 \|w|_{\Sigma_T}\|_
 {W^{4,2}(0,T;H^{1/2}(\partial\Omega))}
 \\
 &+
 \|\partial_\nu w|_{\Sigma_T}\|_
 {W^{3,2}(0,T;H^{-1/2}(\partial\Omega))}.
\end{aligned}
\end{equation}
Proposition~\ref{prop:direct} and the trace estimates imply
\(\mathcal E_T<\infty\). The temporal operator \(D_\beta\) defined
in \eqref{eq:Dbeta} maps the Dirichlet trace space in
\eqref{eq:unfiltered-data-discrepancy} continuously into
\(W^{2,2}(0,T;H^{1/2}(\partial\Omega))\), and maps the Neumann
trace space into \(W^{1,2}(0,T;H^{-1/2}(\partial\Omega))\).
The zero initial traces in \eqref{eq:trace-compat} allow the causal
extensions constructed in the proof of Proposition~\ref{prop:measurement},
and therefore
\begin{equation}\label{eq:filtered-by-unfiltered}
 \Dt\le C_{\beta,T}\mathcal E_T,
\end{equation}
where \(C_{\beta,T}>0\) depends only on \(\beta\) and \(T\).
The stability estimate \eqref{eq:main-stab-unfiltered}, expressed in terms of
\(\mathcal E_T\), follows from Theorem~\ref{thm:main} and
\eqref{eq:filtered-by-unfiltered}.
\end{remark}


\subsection{Carleman geometry and the main theorem}

The direction \(\omega\) fixed in the choice of the initial excitation
\eqref{eq:affine-f} also determines the geometry of the Carleman weight.
Fix a smooth bounded domain \(\Omega_2\) such that
\begin{equation}\label{eq:Omega2}
D\Subset\Omega\Subset\Omega_2,
\end{equation}
with \(\Omega_2\) large enough to contain the supports of the spatial
cutoffs used later in Section~\ref{sec:localization-reduction}. Since
\(\Omega_2\) is bounded, one may choose
\(x_0\notin\overline{\Omega_2}\) sufficiently far in the direction
\(-\omega\) so that the function
\begin{equation}\label{eq:psi}
 \psi(x)=|x-x_0|^2
\end{equation}
is strictly increasing in the direction \(\omega\) throughout \(D\).
More precisely, there exists \(c_*>0\) such that
\begin{equation}\label{eq:direction}
 \partial_\omega\psi(x)
 =\omega\cdot\nabla\psi(x)
 \ge c_*,
 \qquad x\in D.
\end{equation}
The directional condition \eqref{eq:direction} ensures that the Carleman
weight is strictly increasing in the direction \(\omega\) on \(D\), a
property used in Lemma~\ref{lem:transport} to recover \(r\) from
\(\partial_\omega r\).

For \(0<\gamma<1\), define the space--time Carleman weight
\begin{equation}\label{eq:phi}
 \phi(x,t)=\psi(x)-\gamma t^2.
\end{equation}
Choose \(\delta\in(0,T/4)\) and introduce the weight levels
\begin{equation}\label{eq:weight-levels}
 m_D=\min_{\overline D}\psi,
 \qquad
 M_\psi=\max_{\overline{\Omega_2}}\psi,
 \qquad
 d_*=m_D-M_\psi+\gamma(T-2\delta)^2.
\end{equation}
The observation time $T$ is assumed to satisfy
\begin{equation}\label{eq:long-time}
 d_*>0,
\end{equation}
or equivalently,
\[
 T>
 2\delta+
 \sqrt{\frac{M_\psi-m_D}{\gamma}}.
\]
Indeed, by \eqref{eq:phi} and \eqref{eq:weight-levels},
\begin{equation}\label{eq:weight-gap}
 \min_{x\in\overline D}\phi(x,0)
 -
 \max_{\substack{x\in\overline{\Omega_2}\\ T-2\delta\le t\le T}}
 \phi(x,t)
 \ge d_*>0.
\end{equation}
Thus \eqref{eq:long-time} creates a strictly positive gap between the
initial Carleman weight on the unknown region \(D\) and the weight on the
terminal localization region. This gap produces the exponentially decaying
terminal remainder used in the proof of Theorem~\ref{thm:main}.

To state the stability estimate, it is convenient to collect the fixed
a priori bounds into the quantity
\begin{equation}\label{eq:Mstar}
 \Mdir
 =
 1+C_{\rm dir}\norm{f}_{H^6(\R^3)}
 +2M|D|^{1/2},
\end{equation}
where $C_{\rm dir}$ is a positive constant given by \eqref{eq:direct-est}, $M$ is the admissible parameter appearing in \eqref{eq:Qad}, and $|D|$ is the measure of $D$.  By Proposition~\ref{prop:direct} and the definition of the admissible class
\(\Qad\) in \eqref{eq:Qad}, we have
\begin{equation}\label{eq:Mstar-bounds}
 \norm{u_q}_{\XT}
 \le \Mdir,
 \qquad q\in\Qad,
\end{equation}
and, for any \(q_1,q_2\in\Qad\),
\begin{equation}\label{eq:r-apriori-Mdir}
 \norm{q_1-q_2}_{L^2(D)}
 \le
 |D|^{1/2}\norm{q_1-q_2}_{L^\infty(D)}
 \le
 2M|D|^{1/2}
 \le\Mdir.
\end{equation}
In particular, \(\Mdir\) depends only on the fixed a priori data and not on
the particular admissible coefficients being compared.

The main result can now be stated in terms of the measurement discrepancy
\(\Dt\) defined in \eqref{eq:data-discrepancy}.

\begin{theorem}\label{thm:main}
Assume \eqref{eq:affine-f}, \eqref{eq:direction}, and
\eqref{eq:long-time}. Then there exist \(C>0\) and
\(\kappa\in(0,1)\), depending only on the fixed a priori parameters,
such that, for all \(q_1,q_2\in\Qad\),
\begin{equation}\label{eq:main-stab}
 \norm{q_1-q_2}_{L^2(D)}
 \le
 C\Mdir^{1-\kappa}\Dt^\kappa.
\end{equation}
Since the filtered and unfiltered discrepancies satisfy
\eqref{eq:filtered-by-unfiltered}, it follows that
\begin{equation}\label{eq:main-stab-unfiltered}
 \norm{q_1-q_2}_{L^2(D)}
 \le
 C\Mdir^{1-\kappa}\mathcal E_T^\kappa,
\end{equation}
where \(\mathcal E_T\) is defined in \eqref{eq:unfiltered-data-discrepancy}. 
In particular, equality of the corresponding boundary measurements implies
\(q_1=q_2\) in \(D\).
\end{theorem}

The proof of Theorem~\ref{thm:main} occupies
Sections~\ref{sec:localization-reduction}--\ref{sec:main-proof}.

\begin{remark}\label{rem:stability-scope}
The estimates in Theorem~\ref{thm:main} hold uniformly for every pair
\(q_1,q_2\in\Qad\). In particular, neither coefficient is required to lie
in a sufficiently small neighborhood of a fixed reference coefficient, and
no ordering condition such as \(q_1\le q_2\) or \(q_2\le q_1\) is imposed.
Thus the coefficient difference \(q_1-q_2\) may change sign in \(D\).

The H\"older exponent \(\kappa\in(0,1)\) depends only on the fixed a priori
parameters and the Carleman geometry. It is determined by the balance between
the exponentially increasing data term and the exponentially decaying
terminal remainder in the final stability estimate; its explicit value is
given in \eqref{eq:optimization-parameters}. The argument does not address
the optimality of this exponent.

The estimate \eqref{eq:main-stab} is formulated in terms of the filtered
measurement discrepancy \(\Dt\) defined in
\eqref{eq:data-discrepancy}. This formulation is natural for the proof because
the temporal operator \(D_\beta\) introduced in \eqref{eq:Dbeta} produces the
boundary field that enters the localization and Carleman arguments. Since the
filtered data are obtained from the original lateral Cauchy measurement
\(\mathcal C_T(q)\) defined in \eqref{eq:original-cauchy-data} by this fixed
temporal post-processing, the comparison
\eqref{eq:filtered-by-unfiltered} also yields the stability estimate
\eqref{eq:main-stab-unfiltered} in terms of the original Cauchy traces.

The discrepancy \(\mathcal E_T\) in
\eqref{eq:unfiltered-data-discrepancy} measures these original traces in the
specified temporal Sobolev norms. Thus the regularity of the measured
boundary data is part of the stability framework.
\end{remark}


\section{The forward problem}\label{sec:direct-proof}

We prove Proposition~\ref{prop:direct} for \(q\in\Qad\) in the
strong-solution space \(\XT\) defined in \eqref{eq:Xclass}. The uniform
estimate \eqref{eq:direct-est} justifies the initial traces in
\eqref{eq:forward-initial-traces} and the source estimates in
Section~\ref{sec:nonlinear}.

At the highest spatial order, a direct energy estimate would require an
additional derivative of \(u_t\), beyond the regularity in \(\XT\).
We instead integrate the top-order commutator pairing by parts in time
and include the resulting mixed term in a modified energy. The divergence
form in \eqref{eq:memory-system} ensures that \(C_\alpha\) in
\eqref{eq:Calpha} involves derivatives of \(q\) of order at most five,
as permitted by \eqref{eq:Qad}. We then construct the solution by
Fourier spectral approximation and passage to the limit.

\begin{proof}[Proof of Proposition~\ref{prop:direct}]

To express the memory term in \eqref{eq:memory} through a local system, introduce
\begin{equation}\label{eq:relaxed-variables}
\begin{aligned}
 z(x,t)&=\int_0^t\beta e^{-\beta(t-\tau)}u(x,\tau)\dd\tau,\\
 \xi(x,t)&=u(x,t)-z(x,t).
\end{aligned}
\end{equation}
Differentiating the convolution in \eqref{eq:relaxed-variables} yields
\begin{equation}\label{eq:z-ode}
 z_t+\beta z=\beta u,
 \qquad z(\cdot,0)=0.
\end{equation}
Consequently, \(\xi=u-z\) satisfies
\begin{equation}\label{eq:xi-ode}
 \xi_t+\beta\xi=u_t,
 \qquad \xi(\cdot,0)=f.
\end{equation}
Substituting \(z=u-\xi\) into \eqref{eq:memory} gives
\begin{equation}\label{eq:memory-system}
 u_{tt}
 -\diver((1-q)\nabla u)
 -\diver(q\nabla\xi)=0.
\end{equation}
We derive the estimates \eqref{eq:spatial-direct} and
\eqref{eq:time-direct} for the coupled equations \eqref{eq:xi-ode}
and \eqref{eq:memory-system}. In Step~4, we construct a solution
\((u,\xi)\) of these equations and prove that \(u\) satisfies the
original memory equation \eqref{eq:memory}.

The spatial estimate \eqref{eq:spatial-direct} and the time-derivative
estimate \eqref{eq:time-direct} are first derived for smooth solutions.
Step~4 establishes these estimates for the solution constructed in the
class \(\XT\) defined in \eqref{eq:Xclass}.

\medskip
\noindent\emph{Step 1: High-order energy identities.}
For \(|\alpha|\le5\), set
\[
 U_\alpha=D^\alpha u,
 \qquad
 X_\alpha=D^\alpha\xi,
\]
and define the commutator field
\begin{equation}\label{eq:Calpha}
 C_\alpha
 =
 [D^\alpha,1-q]\nabla u
 +
 [D^\alpha,q]\nabla\xi.
\end{equation}
Differentiating the coupled equation \eqref{eq:memory-system} gives
\begin{equation}\label{eq:differentiated-memory}
 (U_\alpha)_{tt}
 -\diver((1-q)\nabla U_\alpha)
 -\diver(q\nabla X_\alpha)
 =\diver C_\alpha.
\end{equation}
Since
\([D^\alpha,1-q]=-[D^\alpha,q]\), and since
\eqref{eq:xi-ode} gives \(\xi_t-u_t=-\beta\xi\),
\[
 C_\alpha=[D^\alpha,q]\nabla(\xi-u),
 \qquad
 (C_\alpha)_t=-\beta[D^\alpha,q]\nabla\xi.
\]
In each commutator in \eqref{eq:Calpha}, at least one spatial derivative
falls on \(q\). The remaining derivatives on \(u\) or \(\xi\), including
the displayed gradient, therefore have total order at most five. The Leibniz
rule and the bound on \(q\) in \eqref{eq:Qad} yield
\begin{equation}\label{eq:Cbound}
\begin{aligned}
 \norm{C_\alpha}_{L^2}
 &\le
 CM\bigl(\norm u_{H^5}+\norm\xi_{H^5}\bigr),
 &&|\alpha|\le5,\\
 \norm{(C_\alpha)_t}_{L^2}
 &\le
 C\beta M\norm\xi_{H^5},
 &&|\alpha|=5.
\end{aligned}
\end{equation}
Thus the commutator estimates require derivatives of \(q\) of order at most
five.

For each \(|\alpha|\le5\), introduce the energy
\begin{equation}\label{eq:component-energy}
 E_\alpha(t)
 =
 \frac12\norm{(U_\alpha)_t}_{L^2}^2
 +\frac12\int_{\R^3}(1-q)\abs{\nabla U_\alpha}^2\dd x
 +\frac12\int_{\R^3}q\abs{\nabla X_\alpha}^2\dd x.
\end{equation}
Testing \eqref{eq:differentiated-memory} by \((U_\alpha)_t\) and using
\[
 \nabla(U_\alpha)_t
 =
 \nabla(X_\alpha)_t+\beta\nabla X_\alpha,
\]
which follows from \eqref{eq:xi-ode}, we can deduce that
\begin{equation}\label{eq:Ealpha-pre}
 \frac{\mathrm d}{\mathrm dt}E_\alpha(t)
 +\beta\int_{\R^3}q\abs{\nabla X_\alpha}^2\dd x
 =
 -\int_{\R^3}
 C_\alpha\cdot\nabla(U_\alpha)_t\dd x.
\end{equation}

For \(|\alpha|\le4\), the commutator estimate \eqref{eq:Cbound} controls
the right-hand side of \eqref{eq:Ealpha-pre} directly:
\begin{equation}\label{eq:low-commutator-pairing}
 \left|
 \int_{\R^3}C_\alpha\cdot\nabla(U_\alpha)_t\dd x
 \right|
 \le
 CM\bigl(\norm u_{H^5}+\norm\xi_{H^5}\bigr)
 \norm{u_t}_{H^5}.
\end{equation}
At the top order \(|\alpha|=5\), however, a direct estimate would require
\(\nabla(U_\alpha)_t\), and hence an \(H^6\)-bound for \(u_t\), which is not
available in the solution class \(\XT\). To avoid this loss of one spatial
derivative, integrate the top-order term by parts in time:
\begin{equation}\label{eq:time-ibp-top}
 -\int_{\R^3}C_\alpha\cdot\nabla(U_\alpha)_t\dd x
 =
 -\frac{\mathrm d}{\mathrm dt}
 \int_{\R^3}C_\alpha\cdot\nabla U_\alpha\dd x
 +
 \int_{\R^3}(C_\alpha)_t\cdot\nabla U_\alpha\dd x.
\end{equation}
By \eqref{eq:Cbound} and the Cauchy--Schwarz inequality, we obtain
\begin{equation}\label{eq:top-commutator-pairing}
 \left|
 \int_{\R^3}(C_\alpha)_t\cdot\nabla U_\alpha\dd x
 \right|
 \le
 C\beta M\norm\xi_{H^5}\norm u_{H^6}.
\end{equation}
Moreover, the identity
\(C_\alpha=[D^\alpha,q]\nabla(\xi-u)\), obtained from
\eqref{eq:Calpha}, and the initial conditions \(u(0)=\xi(0)=f\)
give
\begin{equation}\label{eq:commutator-initial-zero}
 C_\alpha(\cdot,0)=0.
\end{equation}

\medskip
\noindent\emph{Step 2: Modified energy and spatial estimates.}
The time integration by parts in \eqref{eq:time-ibp-top} produces a mixed
commutator term. Incorporating this term into the energy, we remove the
top-order derivative loss identified in Step~1. Define
\begin{equation}\label{eq:modified-energy}
 \mathcal E_5(t)
 =
 \sum_{|\alpha|\le5}E_\alpha(t)
 +
 \sum_{|\alpha|=5}
 \int_{\R^3}C_\alpha\cdot\nabla U_\alpha\dd x
 +
 K_{\rm en}
 \bigl(\norm u_{H^5}^2+\norm\xi_{H^5}^2\bigr).
\end{equation}
The mixed term in \eqref{eq:modified-energy} satisfies
\begin{equation}\label{eq:mixed-energy-bound}
 \left|
 \sum_{|\alpha|=5}
 \int_{\R^3}C_\alpha\cdot\nabla U_\alpha\dd x
 \right|
 \le
 \varepsilon\norm u_{H^6}^2
 +
 C_\varepsilon M^2
 \bigl(\norm u_{H^5}^2+\norm\xi_{H^5}^2\bigr).
\end{equation}
By the bounds \(q\ge0\) and \(1-q\ge c_0\) in \eqref{eq:Qad},
\eqref{eq:component-energy} yields
\[
\begin{aligned}
 \sum_{|\alpha|\le5}E_\alpha(t)
 \ge{}&\frac12\sum_{|\alpha|\le5}\norm{D^\alpha u_t}_{L^2}^2
 +\frac{c_0}{2}\sum_{|\alpha|\le5}\norm{\nabla D^\alpha u}_{L^2}^2\\
 &+\frac12\sum_{|\alpha|\le5}
 \norm{\sqrt q\,\nabla D^\alpha\xi}_{L^2}^2.
\end{aligned}
\]
Since
\(\norm u_{H^5}^2+
\sum_{|\alpha|\le5}\norm{\nabla D^\alpha u}_{L^2}^2
\asymp\norm u_{H^6}^2\), choose
\(\varepsilon>0\) in \eqref{eq:mixed-energy-bound} small enough to
absorb the mixed term into the gradient energy. Then choose
\(K_{\rm en}=K_{\rm en}(M,c_0)\) large enough to absorb the remaining
lower-order contribution. The bounds \eqref{eq:Cbound} and \eqref{eq:mixed-energy-bound}
also give the upper bound for \(\mathcal E_5(t)\), so
\begin{equation}\label{eq:energy-equivalence}
 \mathcal E_5(t)\asymp
 \norm{u_t}_{H^5}^2
 +\norm u_{H^6}^2
 +\norm\xi_{H^5}^2
 +
 \sum_{|\alpha|\le5}
 \norm{\sqrt q\,\nabla D^\alpha\xi}_{L^2}^2,
\end{equation}
where the equivalence constants depend only on \(M\) and \(c_0\).
The lower-order term in \eqref{eq:modified-energy} supplies the \(L^2\)
component of \(\norm u_{H^6}\). The weighted \(\xi\)-term in
\eqref{eq:energy-equivalence} is used only through its nonnegativity, so no
derivatives of \(\sqrt q\) are required.

The auxiliary equation \eqref{eq:xi-ode} and Young's inequality give
\begin{equation}\label{eq:lower-order-energy-derivative}
 \frac{\mathrm d}{\mathrm dt}
 \bigl(
 \norm u_{H^5}^2+\norm\xi_{H^5}^2
 \bigr)
 \le
 C(\beta)
 \bigl(
 \norm{u_t}_{H^5}^2
 +\norm u_{H^5}^2
 +\norm\xi_{H^5}^2
 \bigr).
\end{equation}
Combining
\eqref{eq:Ealpha-pre},
\eqref{eq:low-commutator-pairing},
\eqref{eq:time-ibp-top},
\eqref{eq:top-commutator-pairing},
\eqref{eq:Cbound}, and
\eqref{eq:lower-order-energy-derivative} yields
\begin{equation}\label{eq:energy-gronwall}
 \mathcal E_5'(t)
 +
 \beta\sum_{|\alpha|\le5}
 \norm{\sqrt q\,\nabla D^\alpha\xi}_{L^2}^2
 \le
 C(M,c_0,\beta)\mathcal E_5(t).
\end{equation}
Equation \eqref{eq:commutator-initial-zero} gives
\[
 \mathcal E_5(0)
 \le
 C(M,c_0)\norm f_{H^6}^2.
\]
Applying Gronwall's inequality to \eqref{eq:energy-gronwall} and using the
energy equivalence \eqref{eq:energy-equivalence}, we get that
\begin{equation}\label{eq:spatial-direct}
 \sup_{0\le t\le T}
 \left(
 \norm{u(t)}_{H^6}
 +\norm{u_t(t)}_{H^5}
 +\norm{\xi(t)}_{H^5}
 \right)
 \le
 C(T,\beta,M,c_0)\norm f_{H^6}.
\end{equation}
The constant in \eqref{eq:spatial-direct} depends on the a priori bound
\(M\), but no smallness assumption on \(M\) is used.

Finally, according to the convolution representation of \(z\) in
\eqref{eq:relaxed-variables}, together with
\eqref{eq:spatial-direct}, we know that
\begin{equation}\label{eq:xi-H6-bound}
 \norm{\xi}_{L^\infty(0,T;H^6)}
 \le
 C(T,\beta,M,c_0)\norm f_{H^6}.
\end{equation}

\medskip
\noindent\emph{Step 3: Time-derivative estimates.}
The spatial estimate \eqref{eq:spatial-direct} provides the regularity needed
to bootstrap the time derivatives in the strong-solution class
\(\XT\). For \(2\le m\le6\), the Leibniz rule gives the product estimate
\begin{equation}\label{eq:divergence-product-estimate}
 \norm{\diver(b\nabla v)}_{H^{m-2}}
 \le
 C\norm b_{W^{m-1,\infty}}\norm v_{H^m},
\end{equation}
for \(b\in W^{m-1,\infty}(\R^3)\) and
\(v\in H^m(\R^3)\).

Since the coefficient \(q\) defined in \eqref{eq:Qad} is independent of
time, differentiating the coupled equation \eqref{eq:memory-system},  we can derive that
\begin{equation}\label{eq:time-derivative-recursion}
 \partial_t^{j+2}u
 =
 \diver\bigl((1-q)\nabla\partial_t^ju\bigr)
 +
 \diver\bigl(q\nabla\partial_t^j\xi\bigr),
 \qquad 0\le j\le3.
\end{equation}
The recursion \eqref{eq:time-derivative-recursion} is applied in increasing
order of \(j\). The bounds
\eqref{eq:spatial-direct} and \eqref{eq:xi-ode} first control
\(u,u_t,\xi,\xi_t\) in \(L^\infty(0,T;H^5)\).
Applying \eqref{eq:divergence-product-estimate} with \(m=5\) to the cases
\(j=0,1\) in \eqref{eq:time-derivative-recursion} yields
\[
 u_{tt},u_{ttt}
 \in L^\infty(0,T;H^3).
\]
From the identities
\begin{equation}\label{eq:xi-time-recursion}
 \xi_{tt}=u_{tt}-\beta\xi_t,
 \qquad
 \xi_{ttt}=u_{ttt}-\beta\xi_{tt},
\end{equation}
which follow from \eqref{eq:xi-ode}, we obtain
\[
 \xi_{tt},\xi_{ttt}\in L^\infty(0,T;H^3(\R^3)).
\]
Finally, using the cases \(j=2,3\) in
\eqref{eq:time-derivative-recursion} together with
\eqref{eq:divergence-product-estimate} for \(m=3\), we obtain that
\[
 u^{(4)},u^{(5)}
 \in L^\infty(0,T;H^1).
\]
Consequently,
\begin{equation}\label{eq:time-direct}
\begin{aligned}
 &\operatorname*{ess\,sup}_{0<t<T}
 \bigl(
 \norm{u_{tt}(t)}_{H^3}
 +\norm{u_{ttt}(t)}_{H^3}
 +\norm{u^{(4)}(t)}_{H^1}
 +\norm{u^{(5)}(t)}_{H^1}
 \bigr)
 \\
 &\qquad\le
 C(T,\beta,M,c_0)\norm f_{H^6}.
\end{aligned}
\end{equation}
The estimates \eqref{eq:spatial-direct} and \eqref{eq:time-direct} give
\eqref{eq:direct-est} for smooth solutions.

The full \(W^{5,\infty}(\R^3)\)-regularity imposed on \(q\) in
\eqref{eq:Qad} is needed only for the top-order spatial commutator estimates
in Step~1. Indeed, when \(|\alpha|=5\), the commutator
\(C_\alpha\) defined in \eqref{eq:Calpha} contains derivatives of \(q\) of
order up to five. By contrast, the time-derivative bootstrap in the present
step uses \eqref{eq:divergence-product-estimate} with \(m=5\) for
\(j=0,1\), and hence requires at most \(q\in W^{4,\infty}(\R^3)\); the
cases \(j=2,3\), where \(m=3\), require only
\(q\in W^{2,\infty}(\R^3)\). Thus the \(W^{5,\infty}\)-assumption is
dictated by the highest-order spatial energy estimate rather than by the
time-derivative estimates.

\medskip
\noindent\emph{Step 4: Fourier cutoff approximation and passage to the
limit.}
We construct solutions of the coupled system
\eqref{eq:xi-ode}--\eqref{eq:memory-system} and pass the estimates
\eqref{eq:spatial-direct} and \eqref{eq:time-direct} to the limit.
For each integer \(N\ge1\), let \(\mathcal P_N\) be the orthogonal
Fourier projection onto \(\{|\zeta|\le N\}\), and set
\[
 H_N=\mathcal P_NL^2(\R^3),\qquad
 \mathcal H_N=H_N\times H_N\times H_N.
\]
The projection \(\mathcal P_N\) is self-adjoint, commutes with spatial
derivatives, and is a contraction on every \(H^m(\R^3)\). Moreover,
\begin{align}\label{eq:hNbound-step4}
	 \mathcal P_Nf\longrightarrow f
 \quad\text{in }H^6(\R^3),
 \qquad
 \norm{\varphi_N}_{H^m}
 \le(1+N^2)^{m/2}\norm{\varphi_N}_{L^2}
 \quad(\varphi_N\in H_N).
\end{align}
Thus \(H_N\) consists of band-limited functions belonging to every
spatial Sobolev space.

For \(U=(u,\upsilon,\xi)\in\mathcal H_N\), define
\[
 \mathcal A_NU
 =
 \begin{pmatrix}
  \upsilon\\[1mm]
  \mathcal P_N\diver((1-q)\nabla u)
  +\mathcal P_N\diver(q\nabla\xi)\\[1mm]
  \upsilon-\beta\xi
 \end{pmatrix}.
\]
The Fourier support condition gives
\(\norm{\nabla\varphi_N}_{L^2}\le N\norm{\varphi_N}_{L^2}\) and
\(\norm{\Delta\varphi_N}_{L^2}\le N^2\norm{\varphi_N}_{L^2}\)
for \(\varphi_N\in H_N\).
Consequently, the coefficient bounds in \eqref{eq:Qad} imply
\[
 \norm{\mathcal P_N\diver(q\nabla\varphi_N)}_{L^2}
 \le
 \bigl(N^2\norm q_{L^\infty}
       +N\norm{\nabla q}_{L^\infty}\bigr)\norm{\varphi_N}_{L^2},
 \qquad \varphi_N\in H_N.
\]
Replacing \(q\) by \(1-q\) in the preceding inequality gives
the same type of bound for \(\mathcal P_N\diver((1-q)\nabla\varphi_N)\).
Thus \(\mathcal A_N\in\mathcal L(\mathcal H_N)\) for each fixed \(N\).
The bounded-operator semigroup theorem
\cite[Section~1.1, Theorem~1.2]{Pazy1983} therefore gives a unique
global solution of
\[
 U_N'(t)=\mathcal A_NU_N(t),
 \qquad
 U_N(0)=(\mathcal P_Nf,0,\mathcal P_Nf).
\]
Writing \(U_N=(u_N,(u_N)_t,\xi_N)\), we obtain
\begin{equation}\label{eq:spectral-system}
\begin{cases}
 (u_N)_{tt}
 -\mathcal P_N\diver((1-q)\nabla u_N)
 -\mathcal P_N\diver(q\nabla\xi_N)=0,\\
 (\xi_N)_t+\beta\xi_N=(u_N)_t,\\
 u_N(\cdot,0)=\xi_N(\cdot,0)=\mathcal P_Nf,
 \qquad (u_N)_t(\cdot,0)=0.
\end{cases}
\end{equation}
The operator norm of \(\mathcal A_N\) may depend on \(N\); this does
not affect the uniform estimates derived next. For each fixed \(N\),
\eqref{eq:spectral-system} is smooth in time with band-limited spatial
components, so the energy calculations in Steps~1--3 are justified.

We verify that the spatial and time-derivative estimates \eqref{eq:spatial-direct} and \eqref{eq:time-direct} hold for the approximate solutions \((u_N,\xi_N)\), with constants independent of \(N\). For each multi-index \(\alpha\), set
\[
U_{N,\alpha}=D^\alpha u_N,\qquad
X_{N,\alpha}=D^\alpha\xi_N,
\]
and define
\[
C_{N,\alpha}
=[D^\alpha,1-q]\nabla u_N
+[D^\alpha,q]\nabla\xi_N.
\]
Since \(D^\alpha\) commutes with \(\mathcal P_N\), differentiating
\eqref{eq:spectral-system} gives
\begin{equation}\label{eq:projected-differentiated-system}
 (U_{N,\alpha})_{tt}
 -\mathcal P_N\diver((1-q)\nabla U_{N,\alpha})
 -\mathcal P_N\diver(q\nabla X_{N,\alpha})
 =\mathcal P_N\diver C_{N,\alpha}.
\end{equation}
For \(a\in H^{-1}(\R^3)\) and \(b\in H_N\cap H^1(\R^3)\),
self-adjointness and \(\mathcal P_Nb=b\) give
\begin{equation}\label{eq:projection-duality}
 \langle\mathcal P_Na,b\rangle_{H^{-1},H^1}
 =\langle a,b\rangle_{H^{-1},H^1}.
\end{equation}
In particular, all projections in
\eqref{eq:projected-differentiated-system} disappear when it is tested
against \((U_{N,\alpha})_t\in H_N\). The commutator identities
\eqref{eq:Calpha}--\eqref{eq:Cbound} therefore apply with
\((u,\xi,C_\alpha)\) replaced by
\((u_N,\xi_N,C_{N,\alpha})\). Since
\(u_N(0)=\xi_N(0)\), the initial commutator vanishes as in
\eqref{eq:commutator-initial-zero}. The top-order time integration by
parts \eqref{eq:time-ibp-top}, the modified-energy estimate
\eqref{eq:energy-gronwall}, and the time-derivative recursion
\eqref{eq:time-derivative-recursion} consequently have the same
constants as in Steps~1--3. For the time-derivative recursion
\eqref{eq:time-derivative-recursion}, we also use that \(\mathcal P_N\)
is a contraction on Sobolev spaces. In
particular,
\begin{equation}\label{eq:uniform-spectral-direct}
 \norm{u_N}_{\XT}
 \le C_{\rm dir}\norm{\mathcal P_Nf}_{H^6}
 \le C_{\rm dir}\norm f_{H^6},
 \qquad N\ge1.
\end{equation}

For completeness, set \(z_N=u_N-\xi_N\). The second equation in
\eqref{eq:spectral-system} implies
\( (z_N)_t+\beta z_N=\beta u_N\), with \(z_N(\cdot,0)=0\). Hence
\[
 z_N(x,t)=\int_0^t\beta e^{-\beta(t-\tau)}u_N(x,\tau)\dd\tau.
\]
The spatial bound in \eqref{eq:spatial-direct} therefore also yields
\begin{equation}\label{eq:xiN-H6}
 \sup_{N\ge1}
 \norm{\xi_N}_{L^\infty(0,T;H^6(\R^3))}
 \le C(T,\beta,M,c_0)\norm f_{H^6(\R^3)}.
\end{equation}
The estimates \eqref{eq:uniform-spectral-direct} and \eqref{eq:xiN-H6} hold with constants independent of \(N\). By contrast, the Fourier support bound in \eqref{eq:hNbound-step4} is used only to establish the boundedness of \(\mathcal A_N\) for each fixed \(N\) and to justify the differentiated energy identities. Its \(N\)-dependent constant does not enter the uniform bounds \eqref{eq:uniform-spectral-direct} and \eqref{eq:xiN-H6}.


By weak-\(*\) compactness and the separability of the relevant preduals
\cite[Theorem~3.28 and Corollary~3.30]{Brezis2011}, a subsequence,
not relabeled, and fields \(u,\xi\) can be chosen so that
\[
\begin{aligned}
 u_N&\rightharpoonup^\ast u
 &&\text{in }L^\infty(0,T;H^6(\R^3)),\\
 (u_N)_t&\rightharpoonup^\ast u_t
 &&\text{in }L^\infty(0,T;H^5(\R^3)),\\
 (u_N)_{tt},(u_N)_{ttt}&\rightharpoonup^\ast
 u_{tt},u_{ttt}
 &&\text{in }L^\infty(0,T;H^3(\R^3)),\\
 (u_N)^{(4)},(u_N)^{(5)}&\rightharpoonup^\ast
 u^{(4)},u^{(5)}
 &&\text{in }L^\infty(0,T;H^1(\R^3)),\\
 \xi_N&\rightharpoonup^\ast\xi
 &&\text{in }L^\infty(0,T;H^6(\R^3)).
\end{aligned}
\]
To identify the weak-\(*\) limits of the time derivatives, we observe that
\[
\partial_t (u_N)^{(j)}=(u_N)^{(j+1)},
\qquad 0\le j\le4.
\]
Passing to the limit in these identities in
\(\mathcal D'(\R^3\times(0,T))\) shows that the limiting fields
coincide with \(\partial_t^j u\) for \(1\le j\le5\).
The uniform bounds in \eqref{eq:uniform-spectral-direct} therefore
imply \(u\in\XT\).

The equations in \eqref{eq:spectral-system} also pass to the limit in
distributions. Indeed, for
\(\varphi\in C_c^\infty(\R^3\times(0,T))\),
self-adjointness gives, for example,
\[
 \int_0^T
 \langle\mathcal P_N\diver(q\nabla\xi_N),\varphi\rangle\dd t
 =-\int_0^T\int_{\R^3}
 q\nabla\xi_N\cdot\nabla(\mathcal P_N\varphi)\dd x\dd t.
\]
Since \(\mathcal P_N\varphi\to\varphi\) strongly in
\(L^1(0,T;H^1(\R^3))\), and \(\nabla\xi_N\) is uniformly bounded
and weakly convergent in the corresponding Sobolev spaces, the
right-hand integral in the preceding identity converges to
\(-\int_0^T\!\int_{\R^3}q\nabla\xi\cdot\nabla\varphi\dd x\dd t\).
The term involving \(1-q\) is treated identically. Hence
\begin{equation}\label{eq:limit-memory-system}
 u_{tt}
 -\diver((1-q)\nabla u)
 -\diver(q\nabla\xi)=0
\end{equation}
in \(\mathcal D'(\R^3\times(0,T))\), together with
\begin{equation}\label{eq:limit-xi-ode}
 \xi_t+\beta\xi=u_t.
\end{equation}
Equation \eqref{eq:limit-xi-ode}, together with the weak-\(*\) bounds
for \(u_t\) and \(\xi\), yields
\[
\xi_t=u_t-\beta\xi\in L^\infty(0,T;H^5(\R^3)).
\]
The weak-\(*\) limits obtained above satisfy
\[
u,\xi\in L^\infty(0,T;H^6(\R^3)),
\qquad
u_{tt}\in L^\infty(0,T;H^3(\R^3)).
\]
Since \(q\in W^{5,\infty}(\R^3)\) by \eqref{eq:Qad}, we also have
\[
\diver((1-q)\nabla u),\;
\diver(q\nabla\xi)
\in L^\infty(0,T;H^4(\R^3)).
\]
Consequently, the distributional identity \eqref{eq:limit-memory-system} holds in \(H^3(\R^3)\) for almost every \(t\in(0,T)\).

To recover the initial conditions, we integrate the equations in \eqref{eq:spectral-system} in time. For every \(t\in[0,T]\),
\begin{equation}\label{eq:spectral-integral-identities}
\begin{aligned}
u_N(t)
&=\mathcal P_Nf+\int_0^t (u_N)_t(\tau)\dd\tau
&&\text{in }H^5(\R^3),\\
(u_N)_t(t)
&=\int_0^t (u_N)_{tt}(\tau)\dd\tau
&&\text{in }H^3(\R^3),\\
\xi_N(t)
&=\mathcal P_Nf+
\int_0^t\bigl((u_N)_t(\tau)-\beta\xi_N(\tau)\bigr)\dd\tau
&&\text{in }H^5(\R^3).
\end{aligned}
\end{equation}
Using the weak-\(*\) convergences established above and \(\mathcal P_Nf\to f\) in \(H^6(\R^3)\), we pass to the limit in \eqref{eq:spectral-integral-identities} in the sense of distributions in time and obtain
\begin{equation}\label{eq:limit-integral-identities}
\begin{aligned}
 u(t)
 &=f+\int_0^t u_t(\tau)\dd\tau
 &&\text{in }H^5(\R^3),\\
 u_t(t)
 &=\int_0^t u_{tt}(\tau)\dd\tau
 &&\text{in }H^3(\R^3),\\
 \xi(t)
 &=f+\int_0^t
 \bigl(u_t(\tau)-\beta\xi(\tau)\bigr)\dd\tau
 &&\text{in }H^5(\R^3).
\end{aligned}
\end{equation}
The identities in \eqref{eq:limit-integral-identities} determine
continuous representatives on \([0,T]\) and give
\[
 u(\cdot,0)=\xi(\cdot,0)=f,\qquad u_t(\cdot,0)=0.
\]
Moreover,
\[
\begin{aligned}
 u&\in C([0,T];H^5(\R^3))\cap C_w([0,T];H^6(\R^3)),\\
 u_t&\in C([0,T];H^3(\R^3))\cap C_w([0,T];H^5(\R^3)),\\
 \xi&\in C([0,T];H^5(\R^3))\cap C_w([0,T];H^6(\R^3)).
\end{aligned}
\]
To justify the weak continuity in the higher-order Sobolev spaces, let \(t_n\to t\) in \([0,T]\). The uniform bounds in \(H^6(\R^3)\) for \(u,\xi\) and in \(H^5(\R^3)\) for \(u_t\), together with their strong continuity in \(H^5(\R^3)\) and \(H^3(\R^3)\), respectively, imply
\[
\begin{aligned}
u(t_n)&\rightharpoonup u(t),
\qquad \xi(t_n)\rightharpoonup\xi(t)
&&\text{in }H^6(\R^3),\\
u_t(t_n)&\rightharpoonup u_t(t)
&&\text{in }H^5(\R^3).
\end{aligned}
\]
Weak compactness and the bounds \eqref{eq:uniform-spectral-direct} and
\eqref{eq:xiN-H6} also show that the continuous representatives
\(u(t),\xi(t)\) belong to \(H^6(\R^3)\) and that \(u_t(t)\) belongs
to \(H^5(\R^3)\) for every \(t\in[0,T]\). Evaluating \eqref{eq:limit-integral-identities} at \(t=0\) and using the injective embeddings
\[
H^6(\R^3)\hookrightarrow H^5(\R^3),
\qquad
H^5(\R^3)\hookrightarrow H^3(\R^3),
\]
we conclude that
\[
u(\cdot,0)=\xi(\cdot,0)=f
\quad\text{in }H^6(\R^3),
\qquad
u_t(\cdot,0)=0
\quad\text{in }H^5(\R^3).
\]

Finally, \eqref{eq:time-direct} gives
\[
 u_{tt}\in W^{1,\infty}(0,T;H^3(\R^3)),\qquad
 u_{ttt}\in W^{1,\infty}(0,T;H^1(\R^3)).
\]
Thus \(u_{tt}\) and \(u_{ttt}\) have continuous representatives in
\(H^3\) and \(H^1\), respectively. Also,
\(\xi_t=u_t-\beta\xi\in C([0,T];H^3(\R^3))\).
Consequently, \eqref{eq:limit-memory-system} and its first time
derivative can be evaluated at \(t=0\). Using
\(\xi_t(0)=-\beta f\) and the initial conditions just established
yields the identities \eqref{eq:forward-initial-traces}.

Weak-\(*\) lower semicontinuity now gives
\[
 \norm u_{\XT}
 \le\liminf_{N\to\infty}\norm{u_N}_{\XT}
 \le C_{\rm dir}\norm f_{H^6(\R^3)},
\]
which proves \eqref{eq:direct-est}.   To recover the original memory equation \eqref{eq:memory}, set \(z=u-\xi\). Equation \eqref{eq:limit-xi-ode} and the initial conditions \(u(0)=\xi(0)=f\) imply that \(z\) satisfies \eqref{eq:z-ode}. The variation-of-constants formula therefore gives the convolution representation in \eqref{eq:relaxed-variables}. Substituting \(\xi=u-z\) into \eqref{eq:limit-memory-system} recovers \eqref{eq:memory}. Hence \(u\in\XT\) solves the original forward problem.

To prove uniqueness, let \(u^{(1)},u^{(2)}\in\XT\) solve
\eqref{eq:memory} for the same coefficient \(q\in\Qad\) and the same
initial data, and put
\[
 \widetilde u=u^{(1)}-u^{(2)},\qquad
 \widetilde\xi=\xi^{(1)}-\xi^{(2)}.
\]
The difference satisfies \eqref{eq:xi-ode}--\eqref{eq:memory-system}
with zero initial data. For \(\alpha=0\), the commutator
\(C_\alpha\) defined in \eqref{eq:Calpha} vanishes, and the identity
\eqref{eq:Ealpha-pre} gives
\[
 E_0(t)
 +\beta\int_0^t\!\int_{\R^3}
 q\abs{\nabla\widetilde\xi}^2\dd x\dd\tau
 =E_0(0)=0,
\]
where \(E_0\) is the energy \eqref{eq:component-energy} evaluated
at \((\widetilde u,\widetilde\xi)\). Since
\(q\ge0\) and \(1-q\ge c_0>0\) by \eqref{eq:Qad}, both terms on
the left are nonnegative. In particular,
\(\widetilde u_t=0\). Together with \(\widetilde u(0)=0\), this
implies \(\widetilde u=0\), and \eqref{eq:xi-ode} then yields
\(\widetilde\xi=0\). Thus the solution is unique.
\end{proof}


\section{Localization and boundary reduction}
\label{sec:localization-reduction}

This section converts the comparison of two admissible coefficients into a
compactly supported source problem to which the Carleman estimates of
Section~\ref{sec:acoustic} can be applied. The first step uses the exponential
relaxation law to eliminate the time convolution and identifies the
coefficient difference through an initial source. The second step localizes
the resulting equation in space and estimates the cutoff terms by the
filtered measurement discrepancy \(\Dt\) defined in
\eqref{eq:data-discrepancy}.


\subsection{Reduction to a local difference equation}
\label{sec:localization}

The exponential kernel in \eqref{eq:memory} permits the memory equation to be
replaced by a local third-order equation. The regularity furnished by
Proposition~\ref{prop:direct} justifies the following differentiation in the
sense of distributions. For any \(q\in\Qad\), applying
\(\partial_t+\beta\) to \eqref{eq:memory} for the corresponding
solution \(u_q\) and using the exponential relaxation law gives
\begin{equation}\label{eq:Lq}
 L_qu_q=0,
 \qquad
 L_q
 =
 \partial_t^3+\beta\partial_t^2-\Delta\partial_t
 -\beta\diver((1-q)\nabla\,\cdot\,).
\end{equation}

Let \(q_1,q_2\in\Qad\), with the admissible class \(\Qad\) defined in
\eqref{eq:Qad}, and let \(u_{q_1},u_{q_2}\) denote the corresponding
solutions of \eqref{eq:memory}. Define
\begin{equation}\label{eq:difference-variables}
 w=u_{q_1}-u_{q_2},
 \qquad
 r=q_1-q_2.
\end{equation}
Subtracting \(L_{q_1}u_{q_1}=0\) and \(L_{q_2}u_{q_2}=0\),
which follow from \eqref{eq:Lq}, yields
\begin{equation}\label{eq:w-eq}
 L_{q_1}w
 =
 -\beta\diver(r\nabla u_{q_2}),
 \qquad
 w(\cdot,0)=w_t(\cdot,0)=w_{tt}(\cdot,0)=0.
\end{equation}
The forcing term \(-\beta\diver(r\nabla u_{q_2})\) in \eqref{eq:w-eq} contains the spatial derivative \(\nabla r\). Consequently, estimating this term directly in \(L^2\) would require control of \(\nabla r\), whereas the desired stability estimate concerns \(\|r\|_{L^2(D)}\).

The coefficient information required by the inverse argument is exposed more
directly after one time differentiation. Set 
\begin{equation}\label{eq:vwt}
 v=w_t.
\end{equation}
Since \(q_1,q_2\), and hence \(r\), are independent of time,
\eqref{eq:w-eq} gives
\begin{equation}\label{eq:v-source}
 L_{q_1}v=\diver F^{\rm coef},
 \qquad
 F^{\rm coef}=-\beta r\nabla (\partial_tu_{q_2}),
\end{equation}
The coefficient-dependent source \(F^{\rm coef}\) defined in \eqref{eq:v-source} satisfies
\[
 \supp F^{\rm coef}(\cdot,t)
 \subset\supp r
 \Subset D,
 \qquad
 F^{\rm coef}(\cdot,0)=0,
\]
where the second identity follows from the zero initial velocity in
\eqref{eq:memory}.

The forward initial-trace identities
\eqref{eq:forward-initial-traces} imply
\begin{equation}\label{eq:v-initial}
 v(\cdot,0)=v_t(\cdot,0)=0,
 \qquad
 v_{tt}(\cdot,0)=h=-\beta\diver(r\nabla f).
\end{equation}
The initial excitation \(f\) satisfies the affine condition
\eqref{eq:affine-f}; in particular,
\(\nabla f=\omega\) near \(\overline D\). Since
\(\supp r\Subset D\), the initial source in \eqref{eq:v-initial} reduces to
\begin{equation}\label{eq:h-direction}
 h=-\beta\partial_\omega r,
 \qquad
 \partial_\omega r=\omega\cdot\nabla r.
\end{equation}
Thus the coefficient-to-source reduction converts the unknown coefficient
difference \(r\) into its directional derivative. Proposition~
\ref{prop:initial-trace} will estimate the source \(h\), while
Lemma~\ref{lem:transport} will subsequently convert the resulting control of
\(\partial_\omega r\) into control of \(r\) itself.

\subsection{Compact localization and boundary reduction}
\label{sec:compact-localization}

The Carleman estimate in Section~\ref{sec:acoustic} requires compact spatial support. We therefore localize the solution \(v\) of \eqref{eq:v-source} in the auxiliary domain \(\Omega_2\) fixed in Section~\ref{sec:setting}. The commutator terms arising from derivatives of the cutoff \(\theta\) are supported outside \(\overline\Omega\), where \(q_1=0\) and \(F^{\rm coef}=0\). In this region, the filtered difference field \(W_\beta\), defined in \eqref{eq:filtered-difference-field}, satisfies the free wave equation and can be estimated in terms of the filtered lateral Cauchy measurements. The auxiliary field \(V\), defined in \eqref{eq:H-V}, enters the localized equation through the compactly supported terms \(\rho V\) and \(-\nabla\rho\cdot V\), where the cutoff \(\rho\) is specified below.

The localization is applied to the difference equation \eqref{eq:v-source},
whose principal operator is \(L_{q_1}\) as defined in \eqref{eq:Lq}.
The coefficient-dependent vector source \(F^{\rm coef}\) retains its
definition in \eqref{eq:v-source}. Choose
\(\theta,\rho\in C_c^\infty(\Omega_2)\) such that
\begin{equation}\label{eq:spatial-cutoffs}
\begin{aligned}
&\theta=1\quad\text{in a neighborhood of }\overline\Omega,\qquad
\supp\nabla\theta\Subset\Omega_2\setminus\overline\Omega,\\
&\rho=1\quad\text{in a neighborhood of }\supp\nabla\theta.
\end{aligned}
\end{equation}

Introduce the filtered difference field
\begin{equation}\label{eq:filtered-difference-field}
 W_\beta
 =v_t+\beta v
 =D_\beta w,
\end{equation}
where \(D_\beta\) is the temporal operator defined in \eqref{eq:Dbeta}.
Consequently, the lateral Cauchy traces of \(W_\beta\) on \(\Sigma_T\) are
precisely the difference of the filtered measurements
\(\Phi_T(q_1)-\Phi_T(q_2)\), where \(\Phi_T\) is defined in
\eqref{eq:observation-map}.

Commuting the cutoff \(\theta\) through \(L_{q_1}\) produces, among other terms,
the zeroth-order expression \((\Delta\theta)W_\beta\). To preserve the divergence
structure of the principal cutoff forcing, set
\begin{equation}\label{eq:H-V}
 H_\theta=(\Delta\theta)W_\beta,
 \qquad
 V=-\nabla(-\Delta)^{-1}H_\theta.
\end{equation}
Then
\[
 \diver V=H_\theta.
\]
With the cutoffs \(\theta,\rho\) specified in
\eqref{eq:spatial-cutoffs}, define
\begin{equation}\label{eq:localized-fields}
 \widetilde v=\theta v,
 \qquad
 F^{\rm loc}=F^{\rm coef}-2W_\beta\nabla\theta+\rho V,
 \qquad
 G^{\rm loc}=-\nabla\rho\cdot V.
\end{equation}
Lemma~\ref{lem:compact-localization} establishes \eqref{eq:localized-equation} for the localized fields in
\eqref{eq:localized-fields}.

\begin{lemma}
\label{lem:compact-localization}
Assume that \(v\) and \(F^{\rm coef}\) satisfy \eqref{eq:v-source}, with
\(\supp F^{\rm coef}(\cdot,t)\Subset D\), and let the fields
\(W_\beta,H_\theta,V,F^{\rm loc},G^{\rm loc}\) be defined by
\eqref{eq:filtered-difference-field}--\eqref{eq:localized-fields}. Then the localized fields defined in \eqref{eq:localized-fields} have compact
support in \(\Omega_2\) and satisfy
\begin{equation}\label{eq:localized-equation}
 L_{q_1}\widetilde v=\diver F^{\rm loc}+G^{\rm loc},
\end{equation}

If
\[
 v(\cdot,0)=v_t(\cdot,0)=0,
 \qquad
 F^{\rm coef}(\cdot,0)=0,
\]
then
\begin{equation}\label{eq:localized-source-initial}
 F^{\rm loc}(\cdot,0)=G^{\rm loc}(\cdot,0)=0.
\end{equation}
\end{lemma}

\begin{proof}
Because \(\supp F^{\rm coef}\Subset D\) and \(\theta=1\) near
\(\overline D\),
\[
 \theta\diver F^{\rm coef}=\diver F^{\rm coef}
\]
in the sense of distributions. Moreover,
\(\supp\nabla\theta\Subset\R^3\setminus\overline\Omega\), while
\(\supp q_1\Subset D\Subset\Omega\), so \(q_1=0\) in a neighborhood of
\(\supp\nabla\theta\). A direct commutator calculation using
\eqref{eq:filtered-difference-field} gives
\[
 L_{q_1}(\theta v)
 =
 \diver F^{\rm coef}
 -2\nabla\theta\cdot\nabla W_\beta
 -(\Delta\theta)W_\beta.
\]
On the other hand, \(\rho H_\theta=H_\theta\), and the definitions
\eqref{eq:H-V}--\eqref{eq:localized-fields} give
\[
\begin{aligned}
 \diver F^{\rm loc}+G^{\rm loc}
 &=
 \diver F^{\rm coef}
 -2\diver(W_\beta\nabla\theta)
 +\rho\diver V=
 \diver F^{\rm coef}
 -2\nabla\theta\cdot\nabla W_\beta
 -(\Delta\theta)W_\beta.
\end{aligned}
\]
This proves \eqref{eq:localized-equation}. Compact support follows directly from the
cutoffs \(\theta\) and \(\rho\).

At \(t=0\), the identities \(v(0)=v_t(0)=0\) imply
\(W_\beta(0)=0\) by \eqref{eq:filtered-difference-field}. Hence
\(H_\theta(0)=V(0)=0\) by \eqref{eq:H-V}. Together with \(F^{\rm coef}(0)=0\),
the definitions in \eqref{eq:localized-fields} yield
\eqref{eq:localized-source-initial}.
\end{proof}

Lemma~\ref{lem:compact-localization} reduces the cutoff error to the filtered field \(W_\beta\)
near \(\supp\nabla\theta\). The next step is therefore to estimate \(W_\beta\) in
a neighborhood of \(\supp\nabla\theta\) lying outside \(\overline\Omega\)
from its lateral Cauchy traces. Choose an open set \(U_\theta\) such that
\begin{equation}\label{eq:exterior-neighborhood}
 \supp\nabla\theta
 \Subset U_\theta
 \Subset\R^3\setminus\overline\Omega.
\end{equation}
For each fixed \(t\), the function \(H_\theta(\cdot,t)\) in
\eqref{eq:H-V} is supported in the fixed compact set
\(\supp\nabla\theta\). The Fourier multiplier defining \(V\) gives
\[
 \norm{\nabla V(t)}_{L^2(\R^3)}
 \le
 C\norm{H_\theta(t)}_{L^2(\R^3)}.
\]
The homogeneous Sobolev inequality then yields
\[
 \norm{V(t)}_{L^2(\Omega_2)}
 \le
 C_{\Omega_2}\norm{V(t)}_{L^6(\R^3)}
 \le
 C\norm{\nabla V(t)}_{L^2(\R^3)}
 \le
 C\norm{H_\theta(t)}_{L^2(\R^3)}.
\]
Applying the Fourier multiplier and Sobolev estimates above to
\(V_t=-\nabla(-\Delta)^{-1}(\Delta\theta)(W_\beta)_t\), and using
\eqref{eq:H-V}--\eqref{eq:localized-fields}, we obtain
\begin{equation}\label{eq:localization-error}
 \norm{F^{\rm loc}-F^{\rm coef}}_{H^1(0,T;L^2(\Omega_2))}
 +
 \norm{G^{\rm loc}}_{H^1(0,T;L^2(\Omega_2))}
 \le
 C
 \norm{W_\beta}_{H^1(0,T;L^2(U_\theta))}.
\end{equation}

It remains to control the right-hand side of \eqref{eq:localization-error}. Outside
\(\Omega\), both \(q_1\) and \(F^{\rm coef}\) vanish. Hence
\eqref{eq:v-source} reduces to
\[
 (\partial_t^2-\Delta)(v_t+\beta v)=0.
\]
By \eqref{eq:filtered-difference-field}, the field \(W_\beta\) therefore satisfies the free
wave equation in
\((\R^3\setminus\overline\Omega)\times(0,T)\).
The following estimate controls such an exterior field from its causal
lateral Cauchy traces.

\begin{lemma}
\label{lem:exterior}
Let \(\mathcal O\Subset\R^3\setminus\overline\Omega\) be open, and let
\(Z\) be a causal finite-energy solution of
\[
 Z_{tt}-\Delta Z=0
 \quad\text{in }
 (\R^3\setminus\overline\Omega)\times(0,T),
\]
with zero initial data in the exterior domain. Suppose that
\[
 g=Z|_{\Sigma_T}
 \in W_+^2(0,T;H^{1/2}(\partial\Omega)),
\]
and
\[
 k=\partial_\nu Z|_{\Sigma_T}
 \in W_+^1(0,T;H^{-1/2}(\partial\Omega)).
\]
Then
\begin{equation}\label{eq:exterior-est}
 \norm{Z}_{L^2(0,T;H^1(\mathcal O))}
 +
 \norm{Z_t}_{L^2(\mathcal O\times(0,T))}
 \le
 C_{\Omega,\mathcal O,T}
 \left(
 \norm{g}_{W_+^2(0,T;H^{1/2}(\partial\Omega))}
 +
 \norm{k}_{W_+^1(0,T;H^{-1/2}(\partial\Omega))}
 \right),
\end{equation}
where \(C_{\Omega,\mathcal O,T}>0\) depends only on
\(\Omega\), \(\mathcal O\), and \(T\), and is independent of
\(Z\), \(g\), and \(k\). 
\end{lemma}

\begin{proof}
Let \(\Omega_{\rm ext}=\R^3\setminus\overline\Omega\). The smoothness
of \(\partial\Omega\) ensures the existence of a bounded linear Dirichlet
lifting
\[
 \mathfrak E:H^{1/2}(\partial\Omega)\longrightarrow
 H^1(\Omega_{\rm ext})
\]
such that, for every \(\varphi\in H^{1/2}(\partial\Omega)\),
\begin{equation}\label{eq:exterior-elliptic-lifting}
 \begin{cases}
 (1-\Delta)\mathfrak E\varphi=0
 &\text{in }\Omega_{\rm ext},\\
 (\mathfrak E\varphi)|_{\partial\Omega}=\varphi,
 \end{cases}
 \qquad
 \norm{\mathfrak E\varphi}_{H^1(\Omega_{\rm ext})}
 \le C_\Omega\norm\varphi_{H^{1/2}(\partial\Omega)}.
\end{equation}
Indeed, take any bounded \(H^1\)-extension of \(\varphi\) and solve the
resulting homogeneous-trace problem for \(1-\Delta\) by the Lax--Milgram
theorem.

Since \(g\in W_+^2(0,T;H^{1/2}(\partial\Omega))\), its causal
extensions have vanishing initial traces of orders zero and one; hence
\(g(0)=g_t(0)=0\). Set
\[
 E_g(t)=\mathfrak E g(t).
\]
The boundedness in \eqref{eq:exterior-elliptic-lifting} implies
\begin{equation}\label{eq:exterior-lift-time}
 \norm{E_g}_{W^{2,2}(0,T;H^1(\Omega_{\rm ext}))}
 \le C_\Omega
 \norm g_{W^{2,2}(0,T;H^{1/2}(\partial\Omega))},
 \qquad E_g(\cdot,0)=(E_g)_t(\cdot,0)=0.
\end{equation}
Moreover, \(\Delta E_g=E_g\) in distributions on
\(\Omega_{\rm ext}\), by \eqref{eq:exterior-elliptic-lifting}.

The field \(Z_0=Z-E_g\) has zero Dirichlet and initial data, and
satisfies
\begin{equation}\label{eq:exterior-homogeneous-boundary}
 \begin{cases}
 (Z_0)_{tt}-\Delta Z_0=E_g-(E_g)_{tt}
 &\text{in }\Omega_{\rm ext}\times(0,T),\\
 Z_0|_{\partial\Omega}=0,\qquad
 Z_0(\cdot,0)=(Z_0)_t(\cdot,0)=0.
 \end{cases}
\end{equation}
The forcing \(E_g-(E_g)_{tt}\) belongs to
\(L^2(0,T;L^2(\Omega_{\rm ext}))\) by
\eqref{eq:exterior-lift-time}. The standard energy estimate for
\eqref{eq:exterior-homogeneous-boundary}, first for smooth data and
then by approximation, gives
\[
\begin{aligned}
 &\sup_{0\le t\le T}\Bigl(
 \norm{(Z_0)_t(t)}_{L^2(\Omega_{\rm ext})}
 +\norm{\nabla Z_0(t)}_{L^2(\Omega_{\rm ext})}\Bigr)\\
 &\qquad\le C
 \norm{E_g-(E_g)_{tt}}_{L^1(0,T;L^2(\Omega_{\rm ext}))}
 \le C_T\norm{E_g}_{W^{2,2}(0,T;H^1(\Omega_{\rm ext}))}.
\end{aligned}
\]
The initial condition \(Z_0(0)=0\) also gives
\(\norm{Z_0}_{L^\infty(0,T;L^2(\Omega_{\rm ext}))}\) bounded by
the same right-hand side. The embedding
\(W^{1,2}(0,T;H^1)\hookrightarrow C([0,T];H^1)\), applied to
\(E_g\) and \((E_g)_t\), now yields
\begin{equation}\label{eq:exterior-dirichlet-bound}
 \norm{Z}_{L^2(0,T;H^1(\mathcal O))}
 +\norm{Z_t}_{L^2(\mathcal O\times(0,T))}
 \le C_{\Omega,\mathcal O,T}
 \norm g_{W^{2,2}(0,T;H^{1/2}(\partial\Omega))}.
\end{equation}
Finally, the quotient-norm definition
\eqref{eq:causal-restriction-norm} implies
\[
 \norm g_{W^{2,2}(0,T;H^{1/2}(\partial\Omega))}
 \le \norm g_{W_+^2(0,T;H^{1/2}(\partial\Omega))}.
\]
Thus \eqref{eq:exterior-dirichlet-bound} proves
\eqref{eq:exterior-est}. The argument uses only the Dirichlet trace;
\eqref{eq:exterior-est} retains the full Cauchy pair to match the
measurement discrepancy in \eqref{eq:data-discrepancy}.
\end{proof}

For \(s\ge1\), define the weighted norms of the cutoff-generated terms by
\begin{equation}\label{eq:cutoff-source-norms}
\begin{aligned}
\mathcal F_s^{\rm cut}
&=
\sum_{j=0}^1
\norm{e^{s\phi}\partial_t^j
(F^{\rm loc}-F^{\rm coef})}_{L^2(\Omega_2\times(0,T))}^2,\quad 
\mathcal G_s^{\rm cut}
=
\sum_{j=0}^1
\norm{e^{s\phi}\partial_t^j
G^{\rm loc}}_{L^2(\Omega_2\times(0,T))}^2,
\end{aligned}
\end{equation}
where \(\phi\) is defined in \eqref{eq:phi} and the source fields are defined in \eqref{eq:v-source} and \eqref{eq:localized-fields}.

\begin{proposition}\label{prop:localized-difference}
Let \(q_1,q_2\in\Qad\). The fields \(\widetilde v\), \(F^{\rm loc}\), and \(G^{\rm loc}\) defined in \eqref{eq:localized-fields} have compact spatial support in \(\Omega_2\) and satisfy the equation \eqref{eq:localized-equation} and the source compatibility conditions \eqref{eq:localized-source-initial}. Moreover,
\begin{equation}\label{eq:localized-initial-source}
\widetilde v(\cdot,0)
=\widetilde v_t(\cdot,0)=0,
\qquad
\widetilde v_{tt}(\cdot,0)=h
=-\beta\partial_\omega r.
\end{equation}
The filtered difference field \(W_\beta\), defined in \eqref{eq:filtered-difference-field}, satisfies
\begin{equation}\label{eq:filtered-field-data}
\norm{W_\beta}_{L^2(0,T;H^1(U_\theta))}
+
\norm{(W_\beta)_t}_{L^2(U_\theta\times(0,T))}
\le C\Dt.
\end{equation}
Furthermore, for every \(s\ge1\),
\begin{equation}\label{eq:cutoff-weighted}
\mathcal F_s^{\rm cut}
+\mathcal G_s^{\rm cut}
\le Ce^{C_bs}\Dt^2.
\end{equation}
The constants \(C\) and \(C_b>0\) depend only on the fixed a priori bounds, localization geometry, and Carleman weight, and are independent of \(q_1,q_2\), \(s\), and \(\Dt\).
\end{proposition}

\begin{proof}
The compact support properties, the localized equation \eqref{eq:localized-equation}, and the source compatibility conditions \eqref{eq:localized-source-initial} follow from the construction in \eqref{eq:localized-fields}. Since \(\theta=1\) in a neighborhood of \(\overline D\) and \(h=-\beta\partial_\omega r\) is supported in \(\overline D\) by \eqref{eq:h-direction}, the initial conditions \eqref{eq:v-initial} imply \eqref{eq:localized-initial-source}.

To prove \eqref{eq:filtered-field-data}, apply Lemma~\ref{lem:exterior} with \(Z=W_\beta\) and \(\mathcal O=U_\theta\), where \(U_\theta\) is specified in \eqref{eq:exterior-neighborhood}. By \eqref{eq:filtered-difference-field}, \eqref{eq:v-initial}, and \eqref{eq:h-direction},
\[
W_\beta(\cdot,0)=0,
\qquad
(W_\beta)_t(\cdot,0)=h.
\]
Both initial traces vanish in \(\R^3\setminus\overline\Omega\). Moreover, the lateral Cauchy traces of \(W_\beta\) are given by
\[
\Phi_T(q_1)-\Phi_T(q_2)
\]
according to \eqref{eq:filtered-difference-field} and \eqref{eq:observation-map}. Their initial compatibility conditions follow from \eqref{eq:trace-compat} applied to \(q_1\) and \(q_2\). Lemma~\ref{lem:exterior} and the definition of \(\Dt\) in \eqref{eq:data-discrepancy} therefore yield \eqref{eq:filtered-field-data}.

Finally, \eqref{eq:localization-error} and \eqref{eq:filtered-field-data} give
\[
\begin{aligned}
&\norm{F^{\rm loc}-F^{\rm coef}}_{H^1(0,T;L^2(\Omega_2))}
+\norm{G^{\rm loc}}_{H^1(0,T;L^2(\Omega_2))}
\le C\Dt.
\end{aligned}
\]
Since \(\phi\) is bounded above on \(\Omega_2\times(0,T)\), the definitions in \eqref{eq:cutoff-source-norms} imply \eqref{eq:cutoff-weighted}.
\end{proof}

\section{Acoustic Carleman estimate and recovery of the initial trace}
\label{sec:acoustic}

This section establishes weighted estimates for compactly supported
solutions of the third-order equation \eqref{eq:general-Lq}, where
\(q\in\Qad\) is an arbitrary admissible coefficient and \(F,G\) are
source fields satisfying \eqref{eq:acoustic-support}--
\eqref{eq:acoustic-regularity}. The principal difficulty is the divergence
forcing \(\diver F\): applying a scalar Carleman estimate directly would
require an additional spatial derivative of \(F\).
Proposition~\ref{prop:acoustic} avoids this derivative loss by estimating
the associated acoustic variables without spatially differentiating \(F\).
Proposition~\ref{prop:initial-trace} then establishes a weighted estimate
for the initial value \(h=v_{tt}(\cdot, 0)\).

Fix a compact set \(\mathcal K\Subset\Omega_2\), and assume throughout this
section that
\begin{equation}\label{eq:acoustic-support}
 q\in\Qad,\qquad
 \supp v(\cdot,t)\cup\supp F(\cdot,t)\cup\supp G(\cdot,t)
 \subset\mathcal K,\qquad 0\le t\le T.
\end{equation}
The fields satisfy
\begin{equation}\label{eq:general-Lq}
 L_qv=\diver F+G,\qquad
 v(\cdot,0)=v_t(\cdot,0)=0,\qquad
 v_{tt}(\cdot,0)=h,\qquad
 F(\cdot,0)=G(\cdot,0)=0.
\end{equation}
When applying the estimates of this section to the localized difference
equation \eqref{eq:localized-equation}, we take \(q=q_1\) in
\eqref{eq:general-Lq} and identify the fields by
\begin{equation}\label{eq:localized-acoustic-identification}
 v=\widetilde v,\qquad
 F=F^{\rm loc},\qquad
 G=G^{\rm loc}.
\end{equation}
By Proposition~\ref{prop:localized-difference}, these fields satisfy the
equation and initial conditions in \eqref{eq:general-Lq}. The support
condition \eqref{eq:acoustic-support} holds with a fixed compact set
\(\mathcal K\Subset\Omega_2\) containing
\(\overline D\cup\supp\theta\cup\supp\rho\), by
\eqref{eq:v-source} and \eqref{eq:localized-fields}. In particular,
\eqref{eq:localized-initial-source} gives
\[
 h=\widetilde v_{tt}(\cdot,0)=-\beta\partial_\omega r.
\]
Moreover, the acoustic variable \(Y=v_t+\beta v\), defined in
\eqref{eq:acoustic-vars}, satisfies \(Y=\theta W_\beta\) by
\eqref{eq:localized-fields} and \eqref{eq:filtered-difference-field}.

Unless stated otherwise, all space-time norms in this section are taken over
\begin{equation}\label{eq:space-time-cylinder}
 Q_T=\R^3\times(0,T).
\end{equation}
For the Carleman parameter \(s\ge1\), use the weight \(\phi\) defined in
\eqref{eq:phi} and set
\begin{equation}\label{eq:weighted-norms}
\begin{aligned}
 \Lambda_s&=(s^2-\Delta_x)^{1/2},\quad  \norm U_s=\norm{e^{s\phi}U}_{L^2(Q_T)},\quad \norm U_{-1,s}=\norm{\Lambda_s^{-1}(e^{s\phi}U)}_{L^2(Q_T)},\\
 \mathcal F_s&=\norm{F_t}_s^2+\norm{F}_s^2,
 \qquad
 \mathcal G_s=\norm{G_t}_s^2+\norm{G}_s^2.
\end{aligned}
\end{equation}
The notation in \eqref{eq:weighted-norms} will be used throughout the
section.

Assume in addition that
\begin{equation}\label{eq:acoustic-regularity}
\begin{aligned}
 &v\in L^\infty(0,T;H^5(\R^3)),\quad v_t,v_{tt}\in L^\infty(0,T;H^3(\R^3)),\quad v_{ttt},v^{(4)}\in L^\infty(0,T;H^1(\R^3)),\\
 &F\in H^1(0,T;L^2(\R^3;\R^3)),
 \qquad
 G\in H^1(0,T;L^2(\R^3)).
\end{aligned}
\end{equation}
The initial conditions in \eqref{eq:general-Lq} are understood in the
corresponding trace spaces. Under the identification \eqref{eq:localized-acoustic-identification},
the regularity in \eqref{eq:acoustic-regularity} follows from
Proposition~\ref{prop:direct}, the identity
\(\widetilde v=\theta w_t\), and the cutoff estimate
\eqref{eq:localization-error}.

\subsection{First-order acoustic reduction}

The third-order equation \eqref{eq:general-Lq} is first rewritten as a
first-order acoustic system. Define
\begin{equation}\label{eq:acoustic-vars}
 Y=v_t+\beta v,\qquad
 a=Y_t,\qquad
 p=\nabla v,\qquad
 \pi=\nabla Y+F-\beta qp.
\end{equation}
With these definitions, \(a\) is the wave component, whereas \(p\) and
\(\pi\) satisfy first-order evolution equations. The vector source \(F\)
remains undifferentiated in the resulting lower-order terms.

\begin{lemma}\label{lem:acoustic-id}
Assume \eqref{eq:acoustic-support}, \eqref{eq:general-Lq}, and
\eqref{eq:acoustic-regularity}. Then the variables defined in
\eqref{eq:acoustic-vars} satisfy
\begin{equation}\label{eq:acoustic-system}
\begin{cases}
 a_t-\diver\pi=G,\\
 \pi_t+\beta q\pi-\nabla a=R,\\
 p_t+\beta(1-q)p=\pi-F,
\end{cases}
\qquad
R=\beta^2q(1-q)p+F_t+\beta qF.
\end{equation}
\end{lemma}

\begin{proof}
The first equation in \eqref{eq:acoustic-system} follows from
\eqref{eq:general-Lq} and the definition of \(\pi\) in
\eqref{eq:acoustic-vars}. Moreover,
\[
 p_t
 =\nabla Y-\beta p
 =\pi-F-\beta(1-q)p,
\]
which gives the third equation. Differentiating the definition of \(\pi\)
in \eqref{eq:acoustic-vars} and substituting the preceding identity for
\(p_t\) gives the second equation.
\end{proof}

The initial conditions in \eqref{eq:general-Lq} and the definitions
\eqref{eq:acoustic-vars} give
\[
 a(0)=h,\qquad \pi(0)=p(0)=0.
\]
It remains to determine \(a_t(0)\). The regularity
\eqref{eq:acoustic-regularity} implies
\[
 v,v_t\in C([0,T];H^3(\R^3)),
 \qquad
 v_{tt},v_{ttt}\in C([0,T];H^1(\R^3)),
\]
and
\[
 F,G\in C([0,T];L^2(\R^3)).
\]
Hence every term in \eqref{eq:general-Lq} is continuous with values in
\(H^{-1}(\R^3)\), and \eqref{eq:general-Lq} may be evaluated at \(t=0\).
Using the zero initial traces of \(v,v_t,F\), and \(G\) gives
\[
 v_{ttt}(0)+\beta v_{tt}(0)=0.
\]
Since \(a_t=v_{ttt}+\beta v_{tt}\), it follows that
\begin{equation}\label{eq:acoustic-initial}
 a(\cdot,0)=h,\qquad
 \pi(\cdot,0)=p(\cdot,0)=0,\qquad
 a_t(\cdot,0)=0.
\end{equation}

The regularity in \eqref{eq:acoustic-regularity} further gives
\begin{equation}\label{eq:acoustic-field-regularity}
 a\in L^\infty(0,T;H^3),\qquad
 a_t,a_{tt}\in L^\infty(0,T;H^1),\qquad
 p\in L^\infty(0,T;H^4),\qquad
 \pi,R\in L^2(Q_T).
\end{equation}
Since \(\nabla a\in L^2(Q_T)\), the last two equations in
\eqref{eq:acoustic-system} imply
\[
 p,\pi\in W^{1,2}(0,T;L^2(\R^3)).
\]
Thus the initial values in \eqref{eq:acoustic-initial} are strong
\(L^2\)-traces.

\subsection{Scalar Carleman estimate and temporal localization}

The wave component \(a\) will be estimated by a scalar Carleman inequality.
The estimate is first stated on the symmetric interval \((-T,T)\), after
which the acoustic fields are extended across \(t=0\) and localized away
from \(t=\pm T\).

\begin{lemma}
\label{lem:scalar-Carleman}
Suppose that \(x_0\notin\overline{\Omega_2}\), \(0<\gamma<1\), and
\(\phi\) is the weight defined in \eqref{eq:phi}. Define
\begin{equation}\label{eq:conjugated-wave-operator}
 L_{\phi,s}Z
 =
 e^{s\phi}(\partial_t^2-\Delta)(e^{-s\phi}Z).
\end{equation}
There exist \(s_0\ge1\) and \(C>0\), depending only on
\(T,\gamma\) and the fixed geometry of \(\Omega_2\) and \(x_0\), such that
\begin{equation}\label{eq:scalar-Carleman}
 s\norm{\partial_tZ-s\phi_tZ}_{L^2}^2
 +s\norm{\nabla Z-s\nabla\psi Z}_{L^2}^2
 +s^3\norm Z_{L^2}^2
 \le C\norm{L_{\phi,s}Z}_{L^2}^2
\end{equation}
for every \(s\ge s_0\) and
\(Z\in C_c^\infty(\Omega_2\times(-T,T))\). The spatial weight \(\psi\) is defined in \eqref{eq:psi}. The norms in
\eqref{eq:scalar-Carleman} are taken over
\(\Omega_2\times(-T,T)\).
\end{lemma}

\begin{proof}
A direct proof is given for the quadratic weight \(\phi\). General
Carleman estimates of this type are discussed in
\cite[Chapter~28]{HormanderIV}. Set
\[
 y=x-x_0,\qquad
 d=\dist(x_0,\overline{\Omega_2})>0,
\]
and define
\begin{equation}\label{eq:Carleman-splitting}
\begin{aligned}
 A_s&=\partial_t^2-\Delta
       +4s^2(\gamma^2t^2-|y|^2),\\
 B_s&=4s(\gamma t\partial_t+y\cdot\nabla)
       +2s(3+\gamma).
\end{aligned}
\end{equation}
By \eqref{eq:conjugated-wave-operator},
\[
 L_{\phi,s}=A_s+B_s.
\]
On compactly supported functions, \(A_s\) is formally self-adjoint and
\(B_s\) is formally skew-adjoint. Their commutator is
\[
 [A_s,B_s]
 =8s\gamma\partial_t^2-8s\Delta
   +32s^3(|y|^2-\gamma^3t^2).
\]
Consequently, \begin{equation}\label{eq:Carleman-basic-identity}
\begin{aligned}
\norm{L_{\phi,s}Z}_{L^2}^2
={}&\norm{A_sZ}_{L^2}^2+\norm{B_sZ}_{L^2}^2\\
&+8s\bigl(\norm{\nabla Z}_{L^2}^2
-\gamma\norm{Z_t}_{L^2}^2\bigr)\\
&+32s^3\int_{-T}^{T}\int_{\Omega_2}
\bigl(|y|^2-\gamma^3t^2\bigr)|Z(x,t)|^2
\dd x\dd t.
\end{aligned}
\end{equation}

The coefficient of \(|Z|^2\) in
\eqref{eq:Carleman-basic-identity} need not remain positive on a long time
interval. Set
\[
 b_\gamma=4(1+\gamma),\qquad
 \varepsilon_\gamma=\frac{1-\gamma}{1+3\gamma},\qquad
 \varrho=\frac{\gamma^2t^2}{|y|^2}.
\]
Choose a smooth real-valued function \(\lambda=\lambda(\varrho)\) such that
\[
 |\lambda|\le b_\gamma,\qquad
 \lambda=-b_\gamma\quad(\varrho\le1),\qquad
 \lambda=b_\gamma\quad(\varrho\ge1+\varepsilon_\gamma).
\]
The function \(\lambda\) and its first two derivatives are bounded on
\(\overline{\Omega_2}\times[-T,T]\), independently of \(s\).

Define
\begin{equation}\label{eq:Carleman-forms}
\begin{aligned}
 \mathfrak q_\lambda(\xi,\tau)
   &=(8+\lambda)|\xi|^2-(8\gamma+\lambda)|\tau|^2,\quad  V_\lambda
   =4\bigl((8-\lambda)|y|^2
             +\gamma^2(\lambda-8\gamma)t^2\bigr),
\end{aligned}
\end{equation}
and
\begin{equation}\label{eq:Carleman-remainder-definition}
\begin{aligned}
\mathcal E_s
={}&\frac{s}{2}\int_{-T}^{T}\int_{\Omega_2}
(\lambda_{tt}-\Delta\lambda)|Z(x,t)|^2
\dd x\dd t\\
&-2s^2\int_{-T}^{T}\int_{\Omega_2}
(\gamma t\lambda_t+y\cdot\nabla\lambda)
|Z(x,t)|^2\dd x\dd t.
\end{aligned}
\end{equation}
The boundedness of the derivatives of \(\lambda\) on the fixed cylinder
implies
\begin{equation}\label{eq:Carleman-error-bound}
 |\mathcal E_s|
 \le C(s+s^2)\norm Z_{L^2}^2.
\end{equation}

Adding
\(s\Re\ip{L_{\phi,s}Z}{\lambda Z}\) to
\eqref{eq:Carleman-basic-identity} and integrating by parts over
\(\Omega_2\times(-T,T)\), we obtain
\begin{equation}\label{eq:Carleman-modified-identity}
\begin{aligned}
&\norm{L_{\phi,s}Z}_{L^2}^2
+s\Re\ip{L_{\phi,s}Z}{\lambda Z}\\
&\quad=
\norm{A_sZ}_{L^2}^2+\norm{B_sZ}_{L^2}^2\\
&\qquad
+s\int_{-T}^{T}\int_{\Omega_2}
\mathfrak q_\lambda(\nabla Z,Z_t)\dd x\dd t\\
&\qquad
+s^3\int_{-T}^{T}\int_{\Omega_2}
V_\lambda(x,t)|Z(x,t)|^2\dd x\dd t
+\mathcal E_s.
\end{aligned}
\end{equation}
Since \(Z\in C_c^\infty(\Omega_2\times(-T,T))\), no boundary terms arise in the integrations by parts.

For \(V_\lambda\) in \eqref{eq:Carleman-forms}, if
\(\varrho\le1\), then
\[
 \frac{V_\lambda}{4|y|^2}
 =(8+b_\gamma)-(8\gamma+b_\gamma)\varrho
 \ge8(1-\gamma).
\]
For \(1\le\varrho\le1+\varepsilon_\gamma\), observe that
\[
 \frac{V_\lambda}{4|y|^2}
 =8-8\gamma\varrho+\lambda(\varrho-1).
\]
Since \(\varrho-1\ge0\) and \(\lambda\ge-b_\gamma\), replacing
\(\lambda\) by \(-b_\gamma\) gives a lower bound. This lower bound
is decreasing in \(\varrho\) and equals \(4(1-\gamma)\) at
\(\varrho=1+\varepsilon_\gamma\). If
\(\varrho\ge1+\varepsilon_\gamma\), then
\[
 \frac{V_\lambda}{4|y|^2}
 =(8-b_\gamma)+(b_\gamma-8\gamma)\varrho
 \ge4(1-\gamma).
\]
Since \(|y|\ge d\) on \(\Omega_2\),
\begin{equation}\label{eq:Vlambda-positive}
 V_\lambda\ge16(1-\gamma)d^2.
\end{equation}

Next consider \(\mathfrak q_\lambda\) in
\eqref{eq:Carleman-forms}. For \(\varrho\le1\),
\[
 \mathfrak q_\lambda(\xi,\tau)
 =
 4(1-\gamma)(|\xi|^2+|\tau|^2).
\]
For \(\varrho\ge1\), restrict first to the kernel of
\[
 (\xi,\tau)\longmapsto y\cdot\xi+\gamma t\tau.
\]
On this kernel,
\[
 |\tau|^2\le\frac{|\xi|^2}{\varrho}\le|\xi|^2.
\]
If \(8\gamma+\lambda\ge0\), then
\[
 \mathfrak q_\lambda(\xi,\tau)
 \ge8(1-\gamma)|\xi|^2,
\]
whereas if \(8\gamma+\lambda<0\),
\[
 \mathfrak q_\lambda(\xi,\tau)
 \ge(8-b_\gamma)|\xi|^2
 =4(1-\gamma)|\xi|^2.
\]
Because \(|\tau|^2\le|\xi|^2\) on this kernel, both cases give
\(\mathfrak q_\lambda(\xi,\tau)\ge
2(1-\gamma)(|\xi|^2+|\tau|^2)\).
To extend this lower bound to arbitrary \((\xi,\tau)\), put
\(n=(y,\gamma t)\) and write
\((\xi,\tau)=v_\perp+\alpha n\), where
\(n\cdot v_\perp=0\) and
\(\alpha=n\cdot(\xi,\tau)/|n|^2\). The coefficients of
\(\mathfrak q_\lambda\) are uniformly bounded on the fixed cylinder,
while \(|n|\ge d\). Applying Young's inequality to the mixed term in the expansion of
\(\mathfrak q_\lambda(v_\perp+\alpha n)\) gives constants
\(c,C>0\), independent of \(x,t\), such that
\begin{equation}\label{eq:qlambda-pointwise}
 \mathfrak q_\lambda(\xi,\tau)
 \ge
 c(|\xi|^2+|\tau|^2)
 -C|y\cdot\xi+\gamma t\tau|^2.
\end{equation}
Here the negative term controls only the component of
\((\xi,\tau)\) parallel to \(n\), since
\(|\alpha n|^2=|n\cdot(\xi,\tau)|^2/|n|^2\).

Furthermore,
\begin{equation}\label{eq:Bs-directional-identity}
 y\cdot\nabla Z+\gamma tZ_t
 =
 \frac{1}{4s}B_sZ-\frac{3+\gamma}{2}Z.
\end{equation}
Applying \eqref{eq:qlambda-pointwise} to the real and imaginary parts of
\((\nabla Z,Z_t)\), using
\eqref{eq:Bs-directional-identity}, and integrating gives
\begin{equation}\label{eq:qlambda-integrated}
\begin{aligned}
&s\int_{-T}^{T}\int_{\Omega_2}
\mathfrak q_\lambda(\nabla Z,Z_t)\dd x\dd t\\
&\qquad\ge
cs\bigl(\norm{\nabla Z}_{L^2}^2+\norm{Z_t}_{L^2}^2\bigr)
-Cs^{-1}\norm{B_sZ}_{L^2}^2
-Cs\norm Z_{L^2}^2.
\end{aligned}
\end{equation}

Substituting
\eqref{eq:Carleman-error-bound},
\eqref{eq:Vlambda-positive}, and
\eqref{eq:qlambda-integrated} into
\eqref{eq:Carleman-modified-identity} gives a positive contribution
\(c s^3\|Z\|_{L^2}^2\) from \eqref{eq:Vlambda-positive} and a term
\(c s(\|\nabla Z\|_{L^2}^2+\|Z_t\|_{L^2}^2)\) from
\eqref{eq:qlambda-integrated}. The negative term
\(Cs^{-1}\|B_sZ\|_{L^2}^2\) is absorbed by
\(\|B_sZ\|_{L^2}^2\) for sufficiently large \(s\), while the
\(O(s^2)\|Z\|_{L^2}^2\) terms are absorbed by the positive
\(s^3\|Z\|_{L^2}^2\) contribution. Thus, after increasing \(s_0\),
\begin{equation}\label{eq:Carleman-absorbed-identity}
 s\bigl(\norm{\nabla Z}_{L^2}^2+\norm{Z_t}_{L^2}^2\bigr)
 +s^3\norm Z_{L^2}^2
 \le C\norm{L_{\phi,s}Z}_{L^2}^2
 +Cs\bigl|\Re\ip{L_{\phi,s}Z}{\lambda Z}\bigr|.
\end{equation}
The last term satisfies
\begin{equation}\label{eq:Carleman-cross-term}
 s\bigl|\Re\ip{L_{\phi,s}Z}{\lambda Z}\bigr|
 \le
 \frac12\norm{L_{\phi,s}Z}_{L^2}^2
 +\frac12b_\gamma^2s^2\norm Z_{L^2}^2.
\end{equation}
After increasing \(s_0\) once more, the last term on the right can be
absorbed, and
\begin{equation}\label{eq:unshifted-Carleman}
 s\bigl(\norm{\nabla Z}_{L^2}^2+\norm{Z_t}_{L^2}^2\bigr)
 +s^3\norm Z_{L^2}^2
 \le C\norm{L_{\phi,s}Z}_{L^2}^2.
\end{equation}
Since \(\phi_t\) and \(\nabla\psi\) are bounded on the fixed cylinder,
\eqref{eq:unshifted-Carleman} implies
\eqref{eq:scalar-Carleman}.
\end{proof}

In the subsequent estimates, a threshold denoted by \(s_0\) may be
enlarged; each such threshold depends only on the fixed a priori
parameters and is independent of the particular \(q\in\Qad\).

The scalar Carleman estimate \eqref{eq:scalar-Carleman} is stated on
\((-T,T)\), whereas the acoustic system \eqref{eq:acoustic-system}
is given on \((0,T)\). For a space--time field \(b=b(x,t)\) defined
on \(\R^3\times(0,T)\), write
\[
 b^e(x,t)=b(x,|t|),\qquad
 b^o(x,t)=\operatorname{sgn}(t)b(x,|t|),
 \quad -T<t<T.
\]
Set \(\sigma(t)=\operatorname{sgn}(t)\). Fix the parameter \(\delta\)
from Section~\ref{sec:setting}, and choose an even function
\(\eta\in C_c^\infty((-T,T))\), with \(0\le\eta\le1\), such that
\begin{equation}\label{eq:temporal-cutoff}
 \eta=1\quad\text{for }|t|\le T-2\delta,
 \qquad
 \eta=0\quad\text{for }|t|\ge T-\delta.
\end{equation}
Define the reflected and temporally localized fields by
\begin{equation}\label{eq:eta-fields}
 a_\eta=\eta a^e,\qquad
 \pi_\eta=\eta\pi^o,\qquad
 G_\eta=\eta G^o+\eta'a^e,\qquad
 R_\eta=\eta R^e+\eta'\pi^o.
\end{equation}
For \(0<t<T\), also set
\begin{equation}\label{eq:localized-p-F}
 p_\eta=\eta p,\qquad F_\eta=\eta F.
\end{equation}
The terms generated by the derivatives of \(\eta\) are measured by
\begin{equation}\label{eq:terminal-remainder}
\begin{aligned}
 \mathcal L_s={}&
 \norm{\eta'p}_s^2+\norm{\eta'\pi}_s^2
 +\norm{\eta'F}_s^2+\norm{\eta'a}_s^2\\
 &+\norm{\eta''a+\eta'a_t+\eta'G}_s^2.
\end{aligned}
\end{equation}
The norms in \eqref{eq:terminal-remainder} are taken over \(Q_T\)
according to \eqref{eq:weighted-norms}. Every term in
\eqref{eq:terminal-remainder} is supported in
\(T-2\delta\le t\le T-\delta\).

The following lemma establishes the equations and initial conditions satisfied by the reflected and temporally localized fields, together with estimates for the localized source terms in terms of \(\mathcal F_s\), \(\mathcal G_s\), and \(\mathcal L_s\).

\begin{lemma}
\label{lem:temporal-localization}
Assume \eqref{eq:acoustic-support}, \eqref{eq:general-Lq}, and
\eqref{eq:acoustic-regularity}, with arbitrary \(q\in\Qad\).
Let \(a^e,R^e\) be the even time extensions of \(a,R\), and let
\(\pi^o,G^o\) be the odd time extensions of \(\pi,G\), as defined
before \eqref{eq:temporal-cutoff}. For the cutoff \(\eta\) specified
in \eqref{eq:temporal-cutoff}, let
\(a_\eta,\pi_\eta,G_\eta,R_\eta\) be defined by
\eqref{eq:eta-fields} on \((-T,T)\), and let
\(p_\eta,F_\eta\) be defined by \eqref{eq:localized-p-F}
on \((0,T)\).

The extensions \(a^e,\pi^o,G^o,R^e\) satisfy, in the sense of
distributions on \(\R^3\times(-T,T)\),
\begin{equation}\label{eq:extended-system}
(a^e)_t-\diver\pi^o=G^o,
\qquad
(\pi^o)_t+\beta q\sigma\pi^o-\nabla a^e=R^e.
\end{equation}
The temporally localized fields satisfy
\begin{equation}\label{eq:temporally-localized-system}
 (a_\eta)_t-\diver\pi_\eta=G_\eta,
 \qquad
 (\pi_\eta)_t+\beta q\sigma\pi_\eta-\nabla a_\eta=R_\eta,
\end{equation}
with initial traces
\begin{equation}\label{eq:temporally-localized-traces}
 a_\eta(\cdot,0)=h,\qquad
 (a_\eta)_t(\cdot,0)=0,\qquad
 \pi_\eta(\cdot,0)=G_\eta(\cdot,0)=0.
\end{equation}
In addition, the localized wave component satisfies
\begin{equation}\label{eq:localized-scalar-a}
 (a_\eta)_{tt}-\Delta a_\eta+\beta q\sigma(a_\eta)_t
 =
 \diver R_\eta
 -\beta\sigma\nabla q\cdot\pi_\eta
 +(G_\eta)_t+\beta q\sigma G_\eta.
\end{equation}
On \((0,T)\), the localized vector fields satisfy
\begin{equation}\label{eq:p-eta}
 (p_\eta)_t+\beta(1-q)p_\eta
 =\pi_\eta-F_\eta+\eta'p,
\end{equation}
and
\begin{equation}\label{eq:R-eta}
 R_\eta
 =\beta^2q(1-q)p_\eta
 +(F_\eta)_t+\beta qF_\eta
 +\eta'(\pi-F).
\end{equation}
For every \(s\ge1\), the temporally localized sources satisfy
\begin{equation}\label{eq:localized-G-bound}
 \norm{e^{s\phi}G_\eta}_{L^2(Q_T)}^2
 +\norm{e^{s\phi}(G_\eta)_t}_{L^2(Q_T)}^2
 \le C(\mathcal G_s+\mathcal L_s),
\end{equation}
and
\begin{equation}\label{eq:localized-F-bound}
 \norm{F_\eta}_s^2+\norm{(F_\eta)_t}_s^2
 \le C(\mathcal F_s+\mathcal L_s),
\end{equation}
where \(\mathcal F_s\) and \(\mathcal G_s\) are formed from the
original fields \(F,G\) in \eqref{eq:weighted-norms}, and
\(\mathcal L_s\) is defined in \eqref{eq:terminal-remainder}.
The constant \(C\) depends on the fixed cutoff \(\eta\) and is
independent of \(q\in\Qad\) and \(s\).
\end{lemma}

\begin{proof}
For \(t<0\), the first equality in \eqref{eq:extended-system} is
obtained by taking the negative of the first equation in
\eqref{eq:acoustic-system} at \(-t\); the second equality is the
second equation at \(-t\). The regularity
\eqref{eq:acoustic-field-regularity} gives
\(a\in W^{2,2}(0,T;L^2)\) and
\(\pi\in W^{1,2}(0,T;L^2)\), while
\(G\in W^{1,2}(0,T;L^2)\) by
\eqref{eq:acoustic-regularity}. Because \(a_t(\cdot,0)=0\) and
\(\pi(\cdot,0)=0\) by \eqref{eq:acoustic-initial}, and
\(G(\cdot,0)=0\) by \eqref{eq:general-Lq}, the reflections satisfy
\begin{equation}\label{eq:even-reflection}
 (a^e)_t=(a_t)^o,
 \qquad
 (a^e)_{tt}=(a_{tt})^e,
\end{equation}
and
\begin{equation}\label{eq:odd-reflection}
 (\pi^o)_t=(\pi_t)^e,
 \qquad
 (G^o)_t=(G_t)^e,
\end{equation}
without distributions supported at \(t=0\). The even extension
\(R^e\in L^2(\R^3\times(-T,T))\) requires no initial trace
condition. These identities justify \eqref{eq:extended-system} across
\(t=0\); in particular, no derivative of \(\sigma\) is taken.

Since \(\eta\) depends only on time, applying the product rule to
\eqref{eq:extended-system} and using \eqref{eq:eta-fields} yields
\eqref{eq:temporally-localized-system}. The cutoff \(\eta\) is
constant and equal to one near \(t=0\), so
\eqref{eq:acoustic-initial} and \eqref{eq:general-Lq} give
\eqref{eq:temporally-localized-traces}. Differentiating the first
equation in \eqref{eq:temporally-localized-system} in time and taking
the divergence of the second gives \eqref{eq:localized-scalar-a}; here
we use \(\diver\pi_\eta=(a_\eta)_t-G_\eta\).

On \((0,T)\), we have \(\sigma=1\) and \(\pi_\eta=\eta\pi\).
Multiplying the third equation in \eqref{eq:acoustic-system} by
\(\eta\) and differentiating \(p_\eta=\eta p\) proves
\eqref{eq:p-eta}. Similarly, \(R_\eta=\eta R+\eta'\pi\), and the
formula for \(R\) in \eqref{eq:acoustic-system}, together with
\((F_\eta)_t=\eta F_t+\eta'F\), gives \eqref{eq:R-eta}.

For \(0<t<T\), the definitions \eqref{eq:eta-fields} imply
\[
 G_\eta=\eta G+\eta'a,\qquad
 (G_\eta)_t=\eta G_t+\eta'G+\eta''a+\eta'a_t.
\]
The bounds \(|\eta|\le1\), \eqref{eq:weighted-norms}, and
\eqref{eq:terminal-remainder} imply
\eqref{eq:localized-G-bound}. Likewise,
\(F_\eta=\eta F\) and \((F_\eta)_t=\eta F_t+\eta'F\)
give \eqref{eq:localized-F-bound}.
\end{proof}

In the remainder of this section, \(a,\pi,p,F,G,R\) always denote
the original acoustic fields in \eqref{eq:acoustic-vars} and
\eqref{eq:acoustic-system}. The temporally localized fields retain
the subscript \(\eta\), as specified in \eqref{eq:eta-fields} and
\eqref{eq:localized-p-F}. In particular, \(\mathcal F_s,\mathcal G_s\)
are defined using the original sources. Since \(\phi(x,t)\) and
\(\eta(t)\) are even in time, squared full-time \(L^2\)-norms of the
reflected and localized fields equal twice their positive-time
counterparts. This fixed factor is absorbed into the constants below.

\subsection{Bessel regularization and the negative-order estimate}

The temporally localized scalar equation \eqref{eq:localized-scalar-a},
established in Lemma~\ref{lem:temporal-localization}, contains the
divergence forcing \(\diver R_\eta\). Estimating this term directly in \(L^2\) would require
one additional spatial derivative of \(R_\eta\), which is unavailable in the
application to \eqref{eq:v-source}. A spatial Bessel potential of order
\(-1\) removes this derivative loss.

Set
\begin{equation}\label{eq:conjugated-acoustic-variables}
 A=e^{s\phi}a_\eta,\qquad
 P=e^{s\phi}\pi_\eta,\qquad
 \mathcal R=e^{s\phi}R_\eta,\qquad
 H=e^{s\phi}(a_\eta)_t=A_t-s\phi_tA.
\end{equation}
The first equation in \eqref{eq:temporally-localized-system},
conjugated by \(e^{s\phi}\), gives
\begin{equation}\label{eq:H-divP}
 H
 =
 \diver P-s\nabla\psi\cdot P+e^{s\phi}G_\eta.
\end{equation}

A Bessel potential does not preserve compact spatial support, whereas
Lemma~\ref{lem:scalar-Carleman} requires compact support in
\(\Omega_2\). The following lemma inserts fixed spatial cutoffs and
estimates both the associated off-diagonal errors and the coefficient
commutators.

\begin{lemma}
\label{lem:Bessel}
Let \(q\in\Qad\), where \(\Qad\) is defined in \eqref{eq:Qad}.
Let \(A,P,H\) be the weighted fields defined in
\eqref{eq:conjugated-acoustic-variables}, with \(H\) satisfying
\eqref{eq:H-divP}. Assume that, for every relevant time,
\[
 \supp_x A(\cdot,t)\cup\supp_x P(\cdot,t)
 \cup\supp_x G_\eta(\cdot,t)
 \subset\mathcal K,
\]
where \(\mathcal K\Subset\Omega_2\) is fixed.

Choose
\(\widetilde\chi,\chi\in C_c^\infty(\Omega_2)\) such that
\(\widetilde\chi=1\) in a neighborhood of \(\mathcal K\) and
\(\chi=1\) in a neighborhood of \(\supp\widetilde\chi\). In particular,
\[
 \dist\bigl(\supp\widetilde\chi,\supp(1-\chi)\bigr)>0,
\]
and this separation is fixed independently of the Carleman parameter \(s\).  Define
\begin{equation}\label{eq:Bessel-operators}
 J_s^0=\Lambda_s^{-1}=(s^2-\Delta_x)^{-1/2},
 \qquad
 J_s=\chi J_s^0\widetilde\chi,
\end{equation}
where \(\Lambda_s\) is defined in \eqref{eq:weighted-norms}.

For every fixed integer \(\ell\ge1\), there exists \(C_\ell>0\) such
that
\begin{equation}\label{eq:parametrix}
 \norm{\Lambda_sJ_sA-A}_{L^2}
 \le
 C_\ell s^{-\ell}\norm A_{L^2},
 \qquad s\ge1.
\end{equation}
Moreover,
\begin{equation}\label{eq:J-basic}
 \norm{[L_{\phi,s},J_s]A}_{L^2}
 \le C\norm A_{L^2},
\end{equation}
where $L_{\phi,s}$ is defined by \eqref{eq:conjugated-wave-operator}. 
For every \(L^2\)-field \(U\) whose spatial support is contained in
\(\mathcal K\),
\begin{equation}\label{eq:J-order-minus-one}
 \norm{sJ_sU}_{L^2}
 \le C\norm U_{L^2},
\end{equation}
and, if \(U\) is vector-valued,
\begin{equation}\label{eq:J-divergence}
 \norm{J_s\diver U}_{L^2}
 \le C\norm U_{L^2}.
\end{equation}
The coefficient commutators satisfy
\begin{equation}\label{eq:Jq-comm}
 \norm{[J_s,q]U}_{L^2}
 \le CMs^{-2}\norm U_{L^2},
\end{equation}
for scalar- or vector-valued \(U\), and
\begin{equation}\label{eq:Jq-divP}
 \norm{[J_s,q]\diver P}_{L^2}
 \le CMs^{-1}\norm P_{L^2}.
\end{equation}

Finally, the identity \eqref{eq:H-divP} and the fixed separation of
\(\mathcal K\) from \(\supp(1-\chi)\) imply that, for every fixed
integer \(\ell\ge1\),
\begin{equation}\label{eq:J0-vs-J}
 \norm{J_s^0H}_{L^2}
 \le
 C\norm{J_sH}_{L^2}
 +C_\ell s^{-\ell}
 \left(
 \norm P_{L^2}
 +\norm{e^{s\phi}G_\eta}_{L^2}
 \right),
 \qquad s\ge1.
\end{equation}
The constant \(C\) is independent of \(s\ge1\) and of the particular
coefficient \(q\in\Qad\).
The constants \(C\) and \(C_\ell\) may depend on the fixed weight,
\(\mathcal K\), and the cutoff functions
\(\widetilde\chi,\chi\), while \(C_\ell\) may additionally depend on
the decay order \(\ell\).
\end{lemma}


\begin{proof}
The kernel representation and multiplier properties of the Bessel
potential used below are standard; see
\cite{AdamsFournier2003,AronszajnSmith1961}.
All estimates in the proof are first understood as spatial estimates for a
fixed time. Since their constants are independent of time, the corresponding
space-time estimates follow by integration over the relevant time interval.

\medskip
\noindent\emph{Step 1: Kernel estimates and the localized parametrix.}
Let \(K_s\) denote the convolution kernel of
\(J_s^0=(s^2-\Delta_x)^{-1/2}\). The Laplace representation of its
Fourier multiplier gives
\begin{equation}\label{eq:Bessel-kernel}
\begin{aligned}
 (s^2+|\xi|^2)^{-1/2}
 &=
 \frac1{\Gamma(1/2)}
 \int_0^\infty
 \tau^{-1/2}e^{-\tau(s^2+|\xi|^2)}\dd\tau,\\
 K_s(x)
 &=
 c_3\int_0^\infty
 \tau^{-2}e^{-s^2\tau-|x|^2/(4\tau)}\dd\tau,
\end{aligned}
\end{equation}
where \(c_3>0\) is the heat-kernel normalization constant.
Rescaling in \eqref{eq:Bessel-kernel} gives
\[
 K_s(x)=s^2K_1(sx),
 \qquad
 K_s(x)\le C|x|^{-2},
 \qquad x\ne0.
\]
Gaussian integration also gives
\begin{equation}\label{eq:Bessel-kernel-moments}
 \int_{\R^3}K_s(x)\dd x=s^{-1},\qquad
 \int_{\R^3}|x|K_s(x)\dd x\le Cs^{-2},
 \qquad
 \int_{\R^3}|x|\,|\nabla K_s(x)|\dd x\le Cs^{-1}.
\end{equation}
For the last estimate, differentiating the Gaussian in
\eqref{eq:Bessel-kernel} and integrating its second spatial moment reduces
the remaining integral to
\[
 C\int_0^\infty\tau^{-1/2}e^{-s^2\tau}\dd\tau.
\]

The inequality
\[
 s^2\tau+\frac{|x|^2}{4\tau}\ge s|x|
\]
implies rapid decay of \(K_s\) away from the origin. Indeed,
\(K_s(x)=s^2K_1(sx)\), and differentiation under the integral in
\eqref{eq:Bessel-kernel} shows that for every multi-index \(\alpha\)
and every integer \(N\ge1\),
\[
 |\partial^\alpha K_1(z)|
 \le C_{\alpha,N}(1+|z|)^{-N},
 \qquad |z|\ge1.
\]
After the change of variables \(z=sx\), this bound gives the following
uniform off-diagonal estimate: for fixed multi-indices
\(\alpha,\mu\), every \(d>0\), and every integer \(\ell\ge1\),
\begin{equation}\label{eq:Bessel-offdiagonal}
 \int_{|x|\ge d}
 |x|^{|\mu|}
 |\partial_x^\alpha K_s(x)|\dd x
 \le
 C_{\ell,\alpha,\mu,d}s^{-\ell},
 \qquad s\ge1.
\end{equation}
Only the cases \(|\alpha|\le1\) are used below.

Since \(\supp A\subset\mathcal K\) and
\(\chi=1\) in a neighborhood of \(\supp\widetilde\chi\supset\mathcal K\),
the sets \(\supp A\) and \(\supp(1-\chi)\) are separated by a fixed
positive distance. Set
\[
 g=(1-\chi)J_s^0A.
\]
Because \(\widetilde\chi A=A\) and
\(\Lambda_sJ_s^0A=A\),
\[
 \Lambda_sJ_sA-A
 =
 -\Lambda_sg.
\]
Using \eqref{eq:Bessel-offdiagonal} with a sufficiently high decay order
to absorb the factor \(s\), we obtain, for every fixed \(\ell\ge1\),
\begin{equation}\label{eq:parametrix-local-bound}
 s\norm g_{L^2}
 +\norm{\nabla g}_{L^2}
 \le
 C_\ell s^{-\ell}\norm A_{L^2}.
\end{equation}
Since
\[
 \norm{\Lambda_sg}_{L^2}^2
 =
 s^2\norm g_{L^2}^2+\norm{\nabla g}_{L^2}^2,
\]
\eqref{eq:parametrix-local-bound} gives
\[
 \norm{\Lambda_sJ_sA-A}_{L^2}
 \le
 C_\ell s^{-\ell}\norm A_{L^2},
\]
which proves \eqref{eq:parametrix}.

\medskip
\noindent\emph{Step 2: Multiplier and coefficient commutators.}
The estimates \eqref{eq:J-order-minus-one} and
\eqref{eq:J-divergence} follow directly from the Fourier multipliers
\[
 s(s^2+|\xi|^2)^{-1/2},
 \qquad
 i\xi(s^2+|\xi|^2)^{-1/2},
\]
whose \(L^\infty_\xi\)-norms are bounded uniformly for \(s\ge1\).
Indeed, if \(U\) is supported in \(\mathcal K\), then
\(\widetilde\chi U=U\). If \(U\) is vector-valued, the distribution
\(\diver U\) is also supported in \(\mathcal K\), and hence
\[
 \widetilde\chi\diver U=\diver U.
\]
Multiplication by the fixed cutoff \(\chi\) therefore preserves the
corresponding \(L^2\)-bounds.

We next prove \eqref{eq:J-basic}. Since \(J_s\) acts only in the spatial
variables, it commutes with \(\partial_t\) and with multiplication by
functions of \(t\) alone. Set
\[
 j_s(\xi)=(s^2+|\xi|^2)^{-1/2}.
\]
The conjugated wave operator introduced in
\eqref{eq:conjugated-wave-operator} is
\[
\begin{aligned}
 L_{\phi,s}
 ={}&
 \partial_t^2-\Delta
 -2s\phi_t\partial_t
 +2s\nabla\psi\cdot\nabla+s^2(\phi_t^2-|\nabla\psi|^2)
 -s(\phi_{tt}-\Delta\psi).
\end{aligned}
\]
Write \(y=x-x_0\), so that \(\nabla\psi=2y\) and
\(|\nabla\psi|^2=4|y|^2\). The time-dependent coefficients in
\(L_{\phi,s}\) commute with \(J_s^0\), as do \(-\Delta_x\) and the
constant \(\Delta\psi\). Consequently,
\begin{equation}\label{eq:Bessel-unlocalized-commutator}
 [L_{\phi,s},J_s^0]A
 =4s\sum_{j=1}^3[y_j,J_s^0]\partial_{x_j}A
 -4s^2[|y|^2,J_s^0]A.
\end{equation}
The Fourier transform in the spatial variables gives
\begin{equation}\label{eq:Bessel-coordinate-commutators}
\begin{aligned}
 \widehat{[y_j,J_s^0]A}
 &=i(\partial_{\xi_j}j_s)\widehat A,\\
 \widehat{[|y|^2,J_s^0]A}
 &=-(\Delta_\xi j_s)\widehat A
   +2i\nabla_\xi j_s\cdot\widehat{yA}.
\end{aligned}
\end{equation}
Here \(yA\in L^2\) because \(\supp A\subset\mathcal K\), and
\(\|yA\|_{L^2}\le C_{\mathcal K,x_0}\|A\|_{L^2}\).
The multiplier derivatives are
\begin{equation}\label{eq:Bessel-symbol-derivatives}
 \nabla_\xi j_s
 =-\xi(s^2+|\xi|^2)^{-3/2},
 \qquad
 \Delta_\xi j_s
 =-3s^2(s^2+|\xi|^2)^{-5/2}.
\end{equation}
Rescaling \(\xi=s\zeta\) shows that
\begin{equation}\label{eq:Bessel-symbol-bound}
 \sup_{\xi\in\R^3}
 \left[
 (s|\xi|+s^2)|\nabla_\xi j_s|
 +s^2|\Delta_\xi j_s|
 \right]
 \le C,
 \qquad s\ge1.
\end{equation}
Applying Plancherel's theorem to
\eqref{eq:Bessel-unlocalized-commutator} and using
\eqref{eq:Bessel-coordinate-commutators}--
\eqref{eq:Bessel-symbol-bound}, we obtain
\[
\begin{aligned}
 \norm{[L_{\phi,s},J_s^0]A}_{L^2}
 &\le C\left(
 s\sup_\xi |\xi|\,|\nabla_\xi j_s|
 +s^2\sup_\xi |\Delta_\xi j_s|
 \right)\norm A_{L^2}\\
 &\quad +Cs^2\sup_\xi|\nabla_\xi j_s|\,
 \norm{yA}_{L^2}
 \le C_{\mathcal K,x_0}\norm A_{L^2},
\end{aligned}
\]
with a constant independent of \(s\ge1\).

It remains to account for the spatial cutoffs in \(J_s\).
Since \(\widetilde\chi=1\) near \(\supp A\), the term
\(\chi J_s^0[L_{\phi,s},\widetilde\chi]A\) vanishes.
The commutator identity therefore reads
\[
 [L_{\phi,s},J_s]A
 =\chi[L_{\phi,s},J_s^0]A
  +[L_{\phi,s},\chi]J_s^0A,
\]
where the cutoff commutator is the first-order operator
\[
 [L_{\phi,s},\chi]
 =-2\nabla\chi\cdot\nabla-\Delta\chi
   +2s\nabla\psi\cdot\nabla\chi.
\]
The cutoff functions and \(\nabla\psi\) are bounded on their fixed
supports, while
\(\|\nabla J_s^0U\|_{L^2}\le\|U\|_{L^2}\) and
\(\|sJ_s^0U\|_{L^2}\le\|U\|_{L^2}\) by Plancherel's theorem.
Together with the unlocalized bound above, these estimates prove
\eqref{eq:J-basic}.

We now prove the coefficient commutator estimates. Since
\(J_s^0\) is convolution with \(K_s\),
\[
 [J_s^0,q]U(x)
 =
 \int_{\R^3}
 K_s(x-y)\bigl(q(y)-q(x)\bigr)U(y)\dd y.
\]
The Lipschitz estimate
\[
 |q(y)-q(x)|
 \le
 \norm{\nabla q}_{L^\infty}|x-y|,
\]
together with the second estimate in
\eqref{eq:Bessel-kernel-moments} and the
\(L^1*L^2\to L^2\) convolution inequality
\cite[Theorem~1.2.10]{Grafakos2014}, gives
\[
 \norm{[J_s^0,q]U}_{L^2}
 \le
 Cs^{-2}\norm{\nabla q}_{L^\infty}\norm U_{L^2}.
\]
For \(U\) supported in \(\mathcal K\),
\(\widetilde\chi U=U\), and multiplication by \(q\) commutes with the
cutoff functions. Hence
\[
 [J_s,q]U
 =
 \chi[J_s^0,q]U.
\]
Using
\(\norm{\nabla q}_{L^\infty}\le M\) proves
\eqref{eq:Jq-comm}.

For the divergence commutator in \eqref{eq:Jq-divP}, first take
\(P\) smooth and compactly supported. Integrating by parts in \(y\) gives
\begin{equation}\label{eq:Jq-divergence-formula}
\begin{aligned}
 [J_s^0,q]\diver P(x)
 ={}&
 -\int_{\R^3}
 \nabla_yK_s(x-y)
 \bigl(q(y)-q(x)\bigr)\cdot P(y)\dd y\\
 &-\int_{\R^3}
 K_s(x-y)\nabla q(y)\cdot P(y)\dd y.
\end{aligned}
\end{equation}
The integration may first be carried out on
\(\{|x-y|>\varepsilon\}\). The boundary term on
\(|x-y|=\varepsilon\) tends to zero as \(\varepsilon\to0\), because
\[
 K_s(x-y)\le C|x-y|^{-2}
\]
and
\[
 |q(y)-q(x)|
 \le C|x-y|.
\]
The first and third estimates in
\eqref{eq:Bessel-kernel-moments} and Young's inequality then give
\[
 \norm{[J_s^0,q]\diver P}_{L^2}
 \le
 Cs^{-1}\norm{\nabla q}_{L^\infty}\norm P_{L^2}.
\]
For \(P\in L^2(\R^3;\R^3)\), the commutator
\([J_s^0,q]\diver P\) is initially understood as a distribution:
\(\diver P\in H^{-1}\), and multiplication by the Lipschitz function
\(q\) is bounded on \(H^{-1}\). The integral formula above defines
an \(L^2\) representative with the stated bound. Approximating
compactly supported \(P\) by smooth vector fields and passing to the
limit in distributions extends both the identity and the estimate.

Since \(\widetilde\chi=1\) near \(\supp P\),
\[
 \widetilde\chi\diver P=\diver P
\]
in the sense of distributions. Therefore
\[
 [J_s,q]\diver P
 =
 \chi[J_s^0,q]\diver P,
\]
and the bound
\(\norm{\nabla q}_{L^\infty}\le M\) proves
\eqref{eq:Jq-divP}.

\medskip
\noindent\emph{Step 3: Comparison of \(J_s^0H\) and \(J_sH\).}
By \eqref{eq:H-divP} and the support assumptions on \(P\) and \(G_\eta\),
the field \(H\) is supported in \(\mathcal K\). Hence
\(\widetilde\chi H=H\), and
\[
 J_s^0H-J_sH
 =
 (1-\chi)J_s^0H.
\]
Applying \((1-\chi)J_s^0\) to \eqref{eq:H-divP} gives
\begin{equation}\label{eq:J0-J-decomposition}
\begin{aligned}
 J_s^0H-J_sH
 ={}&
 (1-\chi)J_s^0\diver P-s(1-\chi)J_s^0(\nabla\psi\cdot P)
 +(1-\chi)J_s^0(e^{s\phi}G_\eta).
\end{aligned}
\end{equation}
The supports of \(P\) and \(G_\eta\) are contained in
\(\mathcal K\), which is separated by a fixed positive distance from
\(\supp(1-\chi)\). For the first term in
\eqref{eq:J0-J-decomposition}, integration by parts transfers the
divergence from \(P\) to the kernel \(K_s\), so
\eqref{eq:Bessel-offdiagonal} with \(|\alpha|=1\) applies. For the
remaining terms, \eqref{eq:Bessel-offdiagonal} with
\(|\alpha|=0\) applies; in the second term, the decay order is increased
by one to absorb the explicit factor \(s\). Since \(\nabla\psi\) is
bounded on the fixed set \(\mathcal K\), for every fixed
\(\ell\ge1\),
\begin{equation}\label{eq:J0-J-offdiagonal}
 \norm{J_s^0H-J_sH}_{L^2}
 \le
 C_\ell s^{-\ell}
 \left(
 \norm P_{L^2}
 +\norm{e^{s\phi}G_\eta}_{L^2}
 \right).
\end{equation}
Finally, the triangle inequality gives
\[
 \norm{J_s^0H}_{L^2}
 \le
 \norm{J_sH}_{L^2}
 +\norm{J_s^0H-J_sH}_{L^2},
\]
and \eqref{eq:J0-J-offdiagonal} proves
\eqref{eq:J0-vs-J}.
\end{proof}


Lemma~\ref{lem:Bessel} removes the spatial derivative from the divergence
term \(\diver R_\eta\) in the localized scalar equation
\eqref{eq:localized-scalar-a}. We now combine this regularization with the
scalar Carleman estimate \eqref{eq:scalar-Carleman}. The resulting estimate
controls the conjugated wave variable \(A\) and the negative-order quantity
\(J_s^0H\) in terms of \(\mathcal R\), \(P\), and the source quantities
\(\mathcal G_s\) and \(\mathcal L_s\). The term involving \(\mathcal R\) is subsequently estimated using
Lemma~\ref{lem:Volterra}, which yields \eqref{eq:p-est}--\eqref{eq:R-est}.
The remaining term involving the full vector field \(P\) is treated by the
frequency decomposition in Lemma~\ref{lem:Helmholtz} and the component
estimates in Lemmas~\ref{lem:low}--\ref{lem:transverse}, leading to the
closed acoustic estimate in Proposition~\ref{prop:acoustic}.

\begin{proposition}\label{prop:negative-wave}
Assume \eqref{eq:acoustic-support},
\eqref{eq:general-Lq}, and
\eqref{eq:acoustic-regularity}, and let the acoustic fields be extended and
temporally localized as in Lemma~\ref{lem:temporal-localization}.
Let
\(A,P,\mathcal R,H\) be the conjugated fields defined in
\eqref{eq:conjugated-acoustic-variables}, and let
\(J_s^0,J_s\) be the Bessel operators defined in
\eqref{eq:Bessel-operators}.

There exist \(s_{\rm B}\ge1\) and \(C>0\), depending only on the fixed
a priori parameters, the compact set \(\mathcal K\), and the cutoff functions
chosen in  \eqref{eq:temporal-cutoff} and \eqref{eq:Bessel-operators}, such
that
\begin{equation}\label{eq:negative-wave}
 s\norm A_{L^2}^2
 +s\norm{J_s^0H}_{L^2}^2
 \le
 C\norm{\mathcal R}_{L^2}^2
 +Cs^{-2}\norm P_{L^2}^2
 +Cs^{-2}\mathcal G_s
 +C\mathcal L_s
\end{equation}
for every \(s\ge s_{\rm B}\). Here
\(\mathcal G_s\) and \(\mathcal L_s\) are defined in
\eqref{eq:weighted-norms} and \eqref{eq:terminal-remainder},
respectively.
\end{proposition}

\begin{proof}
We apply the scalar Carleman estimate
\eqref{eq:scalar-Carleman} to
\[
 Z=J_sA,
\]
where \(A=e^{s\phi}a_\eta\) is defined in
\eqref{eq:conjugated-acoustic-variables}.

We first justify this application at the regularity level
\eqref{eq:acoustic-regularity}. The condition \(a_t(0)=0\) in
\eqref{eq:acoustic-initial}, together with the even-reflection identity
\eqref{eq:even-reflection} and the regularity
\eqref{eq:acoustic-field-regularity}, implies that the even extension of
\(a\) has two weak time derivatives in the required \(L^2\)-spaces. After
the temporal localization \eqref{eq:eta-fields} and zero extension outside
\((-T,T)\), multiplication by \(e^{s\phi}\) preserves this regularity for
each fixed \(s\). Since \(J_s\) acts only in the spatial variables and is
of order \(-1\),
\[
 J_sA\in H^2(\R^3\times\R).
\]
Moreover, the outer spatial cutoff \(\chi\) in
\eqref{eq:Bessel-operators} and the temporal cutoff \(\eta\) in
\eqref{eq:temporal-cutoff} give
\[
 \supp(J_sA)\Subset\Omega_2\times(-T,T).
\]
Hence \(J_sA\) can be approximated in \(H^2\) by functions in
\(C_c^\infty(\Omega_2\times(-T,T))\). The estimate
\eqref{eq:scalar-Carleman} therefore extends to \(J_sA\) by density. The
approximation is performed for each fixed \(s\), while the constant in
\eqref{eq:scalar-Carleman} remains independent of \(s\).

On the fixed spatial support of \(J_sA\), the boundedness of
\(\nabla\psi\) gives
\begin{equation}\label{eq:Lambda-JA-control}
 s\norm{\Lambda_sJ_sA}_{L^2}^2
 \le
 C\left(
 s\norm{\nabla J_sA-s\nabla\psi\,J_sA}_{L^2}^2
 +s^3\norm{J_sA}_{L^2}^2
 \right).
\end{equation}
Since \(J_s\) is independent of \(t\) and
\(\phi_t=-2\gamma t\) depends only on \(t\),
\[
 \partial_t(J_sA)-s\phi_tJ_sA
 =
 J_s(A_t-s\phi_tA)
 =
 J_sH,
\]
where \(H\) is defined in
\eqref{eq:conjugated-acoustic-variables}. Thus
\eqref{eq:scalar-Carleman} and \eqref{eq:Lambda-JA-control} imply
\begin{equation}\label{eq:regularized-Carleman-left}
 s\norm{\Lambda_sJ_sA}_{L^2}^2
 +s\norm{J_sH}_{L^2}^2
 \le
 C\norm{L_{\phi,s}(J_sA)}_{L^2}^2.
\end{equation}

Conjugating the localized scalar equation
\eqref{eq:localized-scalar-a} by \(e^{s\phi}\) gives
\begin{equation}\label{eq:conjugated-scalar-equation}
 L_{\phi,s}A
 =
 \diver\mathcal R
 -s\nabla\psi\cdot\mathcal R
 -\beta\sigma\nabla q\cdot P
 +e^{s\phi}(G_\eta)_t
 +\beta q\sigma e^{s\phi}G_\eta
 -\beta q\sigma H,
\end{equation}
where \(A,P,\mathcal R,H\) are defined from the localized fields
in \eqref{eq:conjugated-acoustic-variables}. Therefore,
\[
 L_{\phi,s}(J_sA)
 =
 J_s(L_{\phi,s}A)
 +[L_{\phi,s},J_s]A.
\]

The estimates \eqref{eq:J-order-minus-one} and
\eqref{eq:J-divergence}, together with the bounds on \(q\),
\(\nabla q\), and \(\nabla\psi\), yield
\begin{equation}\label{eq:regularized-forcing-bounds}
\begin{aligned}
 \norm{
 J_s(\diver\mathcal R-s\nabla\psi\cdot\mathcal R)
 }_{L^2}
 &\le C\norm{\mathcal R}_{L^2},\\
 \norm{J_s(\sigma\nabla q\cdot P)}_{L^2}
 &\le Cs^{-1}\norm P_{L^2},\\
 \norm{
 J_s(e^{s\phi}(G_\eta)_t+\beta q\sigma e^{s\phi}G_\eta)
 }_{L^2}
 &\le
 Cs^{-1}
 \bigl(
 \norm{e^{s\phi}(G_\eta)_t}_{L^2}
 +\norm{e^{s\phi}G_\eta}_{L^2}\bigr).
\end{aligned}
\end{equation}
Here the coefficient bounds follow from \eqref{eq:Qad}, while
\(\nabla\psi\) is bounded on the fixed spatial support. The bound \eqref{eq:localized-G-bound} controls these two
\(G_\eta\)-norms by \(\mathcal G_s\) from the original source
\(G\) and the terminal remainder \(\mathcal L_s\).

It remains to estimate the damping term \(-\beta q\sigma H\) in
\eqref{eq:conjugated-scalar-equation}. Since \(J_s\) acts only in the
spatial variables, it commutes with the time-dependent factor \(\sigma\),
and hence
\[
 J_s(q\sigma H)
 =
 q\sigma J_sH+[J_s,q]\sigma H.
\]
Using \eqref{eq:H-divP},
\[
 H=\diver P-s\nabla\psi\cdot P+e^{s\phi}G_\eta,
\]
together with \eqref{eq:Jq-comm} and
\eqref{eq:Jq-divP}, gives
\begin{equation}\label{eq:regularized-damping-bound}
 \norm{[J_s,q]H}_{L^2}
 \le
 CMs^{-1}\norm P_{L^2}
 +CMs^{-2}\norm{e^{s\phi}G_\eta}_{L^2}.
\end{equation}
Indeed, \eqref{eq:Jq-divP} controls the contribution of
\(\diver P\), whereas \eqref{eq:Jq-comm} controls the contributions of
\(s\nabla\psi\cdot P\) and \(e^{s\phi}G_\eta\).

Substituting \eqref{eq:conjugated-scalar-equation} into
\eqref{eq:regularized-Carleman-left} and using
\eqref{eq:J-basic},
\eqref{eq:regularized-forcing-bounds},
\eqref{eq:regularized-damping-bound}, and
\eqref{eq:localized-G-bound}, we obtain
\begin{equation}\label{eq:regularized-wave-preliminary}
\begin{aligned}
 &s\norm{\Lambda_sJ_sA}_{L^2}^2
 +s\norm{J_sH}_{L^2}^2\\
 &\qquad\le
 C\norm{\mathcal R}_{L^2}^2
 +Cs^{-2}\norm P_{L^2}^2
 +Cs^{-2}\mathcal G_s
 +C\norm A_{L^2}^2
 +C\norm{J_sH}_{L^2}^2
 +C\mathcal L_s.
\end{aligned}
\end{equation}

We now increase the lower threshold for \(s\), if necessary. Fix the
decay order \(\ell=2\) in the parametrix estimate
\eqref{eq:parametrix}. Then
\[
 \norm A_{L^2}
 \le
 \norm{\Lambda_sJ_sA}_{L^2}
 +C_2s^{-2}\norm A_{L^2}.
\]
For \(s\) sufficiently large,
\(C_2s^{-2}\le1/2\), and therefore
\begin{equation}\label{eq:A-by-parametrix}
 \norm A_{L^2}
 \le
 2\norm{\Lambda_sJ_sA}_{L^2}.
\end{equation}
Using \eqref{eq:A-by-parametrix}, the term
\(C\norm A_{L^2}^2\) on the right-hand side of
\eqref{eq:regularized-wave-preliminary} is absorbed by
\(s\norm{\Lambda_sJ_sA}_{L^2}^2\) after increasing \(s\) once more.
Similarly, \(C\norm{J_sH}_{L^2}^2\) is absorbed directly by
\(s\norm{J_sH}_{L^2}^2\). Consequently,
\begin{equation}\label{eq:regularized-wave-absorbed}
 s\norm A_{L^2}^2
 +s\norm{J_sH}_{L^2}^2
 \le
 C\norm{\mathcal R}_{L^2}^2
 +Cs^{-2}\norm P_{L^2}^2
 +Cs^{-2}\mathcal G_s
 +C\mathcal L_s.
\end{equation}

We next use \eqref{eq:J0-vs-J}, again with \(\ell=2\). Squaring that
estimate and multiplying by \(s\) gives
\[
 s\norm{J_s^0H}_{L^2}^2
 \le
 Cs\norm{J_sH}_{L^2}^2
 +Cs^{-3}
 \left(
 \norm P_{L^2}^2
 +\norm{e^{s\phi}G_\eta}_{L^2}^2
 \right).
\]
The localized source estimate \eqref{eq:localized-G-bound} and
\(s\ge1\) yield
\[
\begin{aligned}
 &Cs^{-3}\norm P_{L^2}^2
 +Cs^{-3}\norm{e^{s\phi}G_\eta}_{L^2}^2\\
 &\qquad\le
 Cs^{-2}\norm P_{L^2}^2
 +Cs^{-2}\mathcal G_s+C\mathcal L_s.
\end{aligned}
\]
Combining this inequality with
\eqref{eq:regularized-wave-absorbed} yields \eqref{eq:negative-wave}.
\end{proof}


\subsection{Weighted Volterra estimate}
\label{sec:volterra}

The negative-order estimate \eqref{eq:negative-wave} in
Proposition~\ref{prop:negative-wave} still contains the weighted source
\(\mathcal R=e^{s\phi}R_\eta\). By \eqref{eq:R-eta},
\(R_\eta\) contains the localized field \(p_\eta\), which satisfies
\eqref{eq:p-eta}. Since
\[
 e^{s\phi(x,t)}=e^{s\psi(x)}e^{-s\gamma t^2},
\]
the Gaussian temporal factor in the Carleman weight can be used to estimate
\(p_\eta\) with a factor \(s^{-1}\) in the squared norm. The weighted Volterra estimate \eqref{eq:Volterra} will be used again in
the next subsection, specifically in Lemmas~\ref{lem:low} and
\ref{lem:transverse}, to control the low-frequency and high-frequency
transverse components of \(P=e^{s\phi}\pi_\eta\).

\begin{lemma}
\label{lem:Volterra}
Let \(m\ge1\), and let \(Q_T\) be the space-time cylinder defined in
\eqref{eq:space-time-cylinder}. Suppose that
\[
 B\in L^\infty(Q_T;\R^{m\times m}),
 \qquad
 \norm B_{L^\infty(Q_T)}\le C_B,
\]
where the matrix norm is induced by the Euclidean norm. Let
\[
 U\in H^1(0,T;L^2(\R^3;\R^m)),
 \qquad
 R_U\in L^2(Q_T;\R^m),
\]
and assume that
\[
 U_t+BU=R_U
 \quad\text{in }Q_T,
 \qquad
 U(\cdot,0)=0.
\]
Then there exists \(C>0\), depending only on \(\gamma\) and \(C_B\), such
that
\begin{equation}\label{eq:Volterra}
 \int_0^T
 e^{-2s\gamma t^2}
 \norm{U(t)}_{L^2(\R^3)}^2\dd t
 \le
 \frac Cs
 \int_0^T
 e^{-2s\gamma t^2}
 \norm{R_U(t)}_{L^2(\R^3)}^2\dd t,
 \qquad s\ge1.
\end{equation}
Equivalently, if
\begin{equation}\label{eq:Volterra-conjugated-fields}
 \mathbf U=e^{-s\gamma t^2}U,
 \qquad
 \mathbf R_U=e^{-s\gamma t^2}R_U,
\end{equation}
then
\[
 \mathbf U_t+(2s\gamma t+B)\mathbf U=\mathbf R_U,
 \qquad
 \mathbf U(\cdot,0)=0,
\]
and
\[
 \norm{\mathbf U}_{L^2(Q_T)}^2
 \le
 Cs^{-1}\norm{\mathbf R_U}_{L^2(Q_T)}^2,
 \qquad s\ge1.
\]
\end{lemma}

\begin{proof}
For almost every \(x\in\R^3\), let \(E_B(t,\tau;x)\) denote the
fundamental matrix determined by
\[
 \partial_tE_B(t,\tau;x)
 +B(x,t)E_B(t,\tau;x)=0,
 \qquad
 E_B(\tau,\tau;x)=I,
 \qquad
 0\le\tau\le t\le T.
\]
Since \(\norm{B(x,t)}\le C_B\) for almost every \((x,t)\), the integral
equation for \(E_B\) gives
\[
 \norm{E_B(t,\tau;x)}
 \le
 1+C_B\int_\tau^t\norm{E_B(t',\tau;x)}\dd t'.
\]
Gronwall's inequality therefore yields
\[
 \norm{E_B(t,\tau;x)}
 \le
 e^{C_B(t-\tau)},
 \qquad
 0\le\tau\le t\le T,
\]
for almost every \(x\in\R^3\). 
Applying variation of constants to
\(\mathbf U_t+(2s\gamma t+B)\mathbf U=\mathbf R_U\), with
\(\mathbf U,\mathbf R_U\) defined in
\eqref{eq:Volterra-conjugated-fields}, gives
\[
 \mathbf U(x,t)
 =
 \int_0^t
 e^{-s\gamma(t^2-\tau^2)}
 E_B(t,\tau;x)
 \mathbf R_U(x,\tau)\dd\tau.
\]
Consequently,
\[
 |\mathbf U(x,t)|
 \le
 \int_0^T
 k_s(t,\tau)
 |\mathbf R_U(x,\tau)|\dd\tau,
\]
where
\[
 k_s(t,\tau)
 =
 \mathbf1_{\{0\le\tau<t\le T\}}
 e^{-s\gamma(t^2-\tau^2)+C_B(t-\tau)}.
\]

For \(0\le\tau\le t\),
\[
 t^2-\tau^2
 =(t-\tau)(t+\tau)
 \ge(t-\tau)^2.
\]
Hence
\[
\begin{aligned}
 \max\left\{
 \sup_{0\le t\le T}\int_0^T k_s(t,\tau)\dd\tau,\,
 \sup_{0\le\tau\le T}\int_0^T k_s(t,\tau)\dd t
 \right\}
 &\le
 \int_0^\infty
 e^{-s\gamma(t')^2+C_Bt'}\dd t'\\
 &\le
 e^{C_B^2/(2s\gamma)}
 \int_0^\infty
 e^{-s\gamma(t')^2/2}\dd t'\\
 &\le
 Cs^{-1/2},
 \qquad s\ge1.
\end{aligned}
\]
Schur's test
\cite[Appendix~A.1, p.~589]{GrafakosModern2014}
therefore gives
\[
 \norm{\mathbf U(x,\cdot)}_{L^2(0,T)}
 \le
 Cs^{-1/2}
 \norm{\mathbf R_U(x,\cdot)}_{L^2(0,T)}
\]
for almost every \(x\), with \(C\) independent of \(x\) and \(s\ge1\).
Squaring this inequality and integrating over \(x\in\R^3\) proves
\eqref{eq:Volterra}.
\end{proof}

The weighted Volterra estimate \eqref{eq:Volterra} can now be applied to
the first-order equation \eqref{eq:p-eta}. This controls the localized field \(p_\eta\)
appearing in \(R_\eta\) defined by \eqref{eq:R-eta}, and thereby
removes the term \(\mathcal R=e^{s\phi}R_\eta\) from the negative-order estimate
\eqref{eq:negative-wave}.

\begin{lemma} 
\label{lem:Volterra-reduction}
Assume the hypotheses of Proposition~\ref{prop:negative-wave}. Then, for
\(s\ge s_{\rm B}\),
\begin{equation}\label{eq:p-est}
 \norm{p_\eta}_s^2
 \le
 \frac Cs
 \bigl(
 \norm P_{L^2}^2+\norm F_s^2
 \bigr)
 +C\mathcal L_s,
\end{equation}
and
\begin{equation}\label{eq:R-est}
 \norm{\mathcal R}_{L^2}^2
 \le
 \frac Cs\norm P_{L^2}^2
 +C\mathcal F_s
 +C\mathcal L_s.
\end{equation}
Consequently,
\begin{equation}\label{eq:negative-wave2}
 s\norm A_{L^2}^2
 +s\norm{J_s^0H}_{L^2}^2
 \le
 \frac Cs\norm P_{L^2}^2
 +C\mathcal F_s
 +Cs^{-2}\mathcal G_s
 +C\mathcal L_s.
\end{equation}
Here \(C>0\) depends only on the fixed a priori parameters and the fixed
cutoff functions introduced above.
\end{lemma}

\begin{proof}
Apply Lemma~\ref{lem:Volterra} to the first-order equation
\eqref{eq:p-eta}. Set
\[
 U=e^{s\psi}p_\eta,
 \qquad
 B=\beta(1-q)I_3,
\]
and
\[
 R_U
 =
 e^{s\psi}
 \bigl(\pi_\eta-F_\eta+\eta'p\bigr).
\]
The coefficient bounds in \eqref{eq:Qad} imply a uniform
\(L^\infty\)-bound for \(B\), and
\[
 p_\eta(0)=p(0)=0
\]
by \eqref{eq:acoustic-initial}. Since
\[
 e^{s\phi}=e^{s\psi}e^{-s\gamma t^2},
\]
the Volterra estimate \eqref{eq:Volterra} gives
\[
 \norm{p_\eta}_s^2
 \le
 \frac Cs
 \left(
 \norm P_{L^2}^2
 +\norm{F_\eta}_s^2
 +\norm{\eta'p}_s^2
 \right).
\]
The term \(\norm{\eta'p}_s^2\) is included in
\(\mathcal L_s\) by \eqref{eq:terminal-remainder}, while
\(\norm{F_\eta}_s\le\norm F_s\) follows from
\eqref{eq:localized-p-F} and \(0\le\eta\le1\) in
\eqref{eq:temporal-cutoff}. Since \(s\ge1\), this proves
\eqref{eq:p-est}.

Next, the representation \eqref{eq:R-eta}, together with the
localized vector-source bound \eqref{eq:localized-F-bound}, gives
\[
 \norm{\mathcal R}_{L^2}^2
 \le
 C\norm{p_\eta}_s^2
 +C\bigl(
 \norm{F_t}_s^2+\norm F_s^2
 \bigr)
 +C\mathcal L_s.
\]
Substituting \eqref{eq:p-est} and using the definition of
\(\mathcal F_s\) in \eqref{eq:weighted-norms} yields
\eqref{eq:R-est}.

Finally, substituting \eqref{eq:R-est} into the negative-order estimate
\eqref{eq:negative-wave} of Proposition~\ref{prop:negative-wave} gives
\eqref{eq:negative-wave2}.
\end{proof}

The estimate \eqref{eq:negative-wave2} no longer contains the remainder
\(\mathcal R\). Its only uncontrolled acoustic quantity is the full vector
field \(P=e^{s\phi}\pi_\eta\). The frequency decomposition in
Lemma~\ref{lem:Helmholtz}, followed by
Lemmas~\ref{lem:low}--\ref{lem:transverse}, provides the remaining control
of \(P\).



\subsection{Fourier decomposition and closure}
\label{sec:fourier-closure}

Lemma~\ref{lem:Volterra-reduction} reduces the negative-order estimate
\eqref{eq:negative-wave} to \eqref{eq:negative-wave2}. The only
uncontrolled acoustic quantity on the right-hand side of
\eqref{eq:negative-wave2} is the full vector field
$
 P=e^{s\phi}\pi_\eta
$
defined in \eqref{eq:conjugated-acoustic-variables}. The relation
\eqref{eq:H-divP} controls a weighted divergence of \(P\), but does not by
itself determine all components of \(P\).

To recover the full vector field \(P=e^{s\phi}\pi_\eta\), we use the
decomposition
\[
 P=T_0P+T_LP+T_TP
\]
from \eqref{eq:decomp}, where \(T_0P\), \(T_LP\), and \(T_TP\) denote,
respectively, the low-frequency, high-frequency longitudinal, and
high-frequency transverse components of \(P\). The low-frequency component is estimated from the evolution equation
\eqref{eq:P-eq} by the weighted Volterra estimate
\eqref{eq:Volterra}. The high-frequency longitudinal component is controlled
directly by the divergence relation \eqref{eq:H-divP}, while the transverse
component is again estimated from \eqref{eq:P-eq}, using the fact that the
transverse projection annihilates \(\nabla A\). The three estimates are
combined in Proposition~\ref{prop:acoustic}.

\begin{lemma}
\label{lem:Helmholtz}
Let \(q\in\Qad\), where \(\Qad\) is defined in \eqref{eq:Qad},
and fix \(K\ge2\). Choose smooth radial functions
\(0\le\chi_0,\chi_1\le1\) satisfying
\begin{equation}\label{eq:frequency-cutoffs}
 \chi_0+\chi_1=1,\qquad
 \chi_0=1\quad\text{on }|\xi|\le K,\qquad
 \chi_0=0\quad\text{on }|\xi|\ge2K.
\end{equation}
For \(s\ge1\), define
\begin{equation}\label{eq:Helmholtz-multipliers}
 T_0=\chi_0(\mathsf D/s),\qquad
 T_L=\chi_1(\mathsf D/s)
 \frac{\mathsf D\otimes\mathsf D}{|\mathsf D|^2},
 \qquad
 T_T=\chi_1(\mathsf D/s)
 \left(
 I-\frac{\mathsf D\otimes\mathsf D}{|\mathsf D|^2}
 \right).
\end{equation}
These operators are bounded on \(L^2(\R^3;\R^3)\), uniformly in
\(s\ge1\), and satisfy
\begin{equation}\label{eq:decomp}
 I=T_0+T_L+T_T,
 \qquad
 T_T\nabla A=0
\end{equation}
for every scalar \(A\in H^1(\R^3)\). Moreover,
\begin{equation}\label{eq:T-comm}
 \norm{[T_0,q]}_{L^2\to L^2}
 +
 \norm{[T_T,q]}_{L^2\to L^2}
 \le
 C_Ks^{-1},
 \qquad s\ge1,
\end{equation}
where \(C_K>0\) may depend on \(K\), \(M\), and the fixed radial cutoffs
in \eqref{eq:frequency-cutoffs}, but is independent of \(s\) and of the
particular coefficient \(q\).
\end{lemma}

\begin{proof}
The identities in \eqref{eq:decomp} follow directly from
\eqref{eq:frequency-cutoffs}--\eqref{eq:Helmholtz-multipliers}. In
particular, the factor \(\chi_1\) vanishes near \(\xi=0\), so the
longitudinal and transverse projection symbols in
\eqref{eq:Helmholtz-multipliers} are smooth on their supports.

Let \(k_K=\mathcal F^{-1}\chi_0\), where \(\chi_0\) is defined in \eqref{eq:frequency-cutoffs}. Since \(\chi_0\in C_c^\infty(\R^3)\), its inverse Fourier transform \(k_K\) belongs to the Schwartz space \(\mathcal S(\R^3)\); see \cite[Section~2.2]{Grafakos2014}. By Fourier scaling, the operator \(T_0=\chi_0(\mathsf D/s)\) defined in \eqref{eq:Helmholtz-multipliers} is convolution with \(s^3k_K(s\cdot)\). The change of variables \(z=sx\) gives
\[
\begin{aligned}
\int_{\R^3}|x|s^3|k_K(sx)|\dd x
&=s^{-1}\int_{\R^3}|z||k_K(z)|\dd z \le C_Ks^{-1},
\end{aligned}
\]
where \(C_K\) is independent of \(s\).

For \(U\in L^2(\R^3;\R^3)\), the commutator \([T_0,q]U=T_0(qU)-qT_0U\) has the integral representation
\[
[T_0,q]U(x)
=
\int_{\R^3}s^3k_K(s(x-y))
\bigl(q(y)-q(x)\bigr)U(y)\dd y.
\]
Since \(q\in\Qad\), the Lipschitz bound
\[
|q(y)-q(x)|
\le \norm{\nabla q}_{L^\infty}|x-y|
\]
holds. Young's convolution inequality therefore yields
\[
\begin{aligned}
\norm{[T_0,q]U}_{L^2}
&\le \norm{\nabla q}_{L^\infty}
\left(\int_{\R^3}|x|s^3|k_K(sx)|\dd x\right)
\norm U_{L^2}.
\end{aligned}
\]
Using the uniform bound \(\norm{\nabla q}_{L^\infty}\le M\) from \eqref{eq:Qad}, we obtain
\begin{equation}\label{eq:T0-commutator}
\norm{[T_0,q]U}_{L^2}
\le C_Ks^{-1}\norm U_{L^2},
\end{equation}
where \(C_K\) is independent of \(s\) and of the particular coefficient \(q\in\Qad\).

It remains to establish the bound for \([T_T,q]\) in \eqref{eq:T-comm}. Choose a fixed smooth order-zero matrix-valued symbol \(\widetilde\Pi_T\) such that
\[
\widetilde\Pi_T(\xi)
=
I-\frac{\xi\otimes\xi}{|\xi|^2},
\qquad |\xi|\ge1.
\]
Then
\[
 T_T
 =
 \chi_1(\mathsf D/s)\widetilde\Pi_T(\mathsf D),
\]
and
\begin{equation}\label{eq:TT-commutator-decomposition}
\begin{aligned}
 [T_T,q]
 ={}&
 \chi_1(\mathsf D/s)
 [\widetilde\Pi_T(\mathsf D),q]\\
 &+
 [\chi_1(\mathsf D/s),q]
 \widetilde\Pi_T(\mathsf D).
\end{aligned}
\end{equation}
Since \(\chi_1=1-\chi_0\), the second term in
\eqref{eq:TT-commutator-decomposition} is controlled by
\eqref{eq:T0-commutator}.

For the first term in \eqref{eq:TT-commutator-decomposition}, the
fixed symbol \(\widetilde\Pi_T\) belongs to the order-zero class
\(S^0(\R^3)\). Since \(q\in\Qad\), the coefficient is Lipschitz,
with \(\norm{q}_{W^{1,\infty}}\le M\). Apply the order-zero case of
\cite[Corollary~A.2]{Smith2014}, with Sobolev exponent zero, to
\(q\) and each scalar entry of the fixed matrix-valued symbol
\(\widetilde\Pi_T\). Summing the resulting component estimates gives
\begin{equation}\label{eq:fixed-transverse-commutator}
 \norm{[\widetilde\Pi_T(\mathsf D),q]U}_{H^1(\R^3)}
 \le C\norm{q}_{W^{1,\infty}(\R^3)}\norm U_{L^2(\R^3)}.
\end{equation}
Here \(C\) depends only on the fixed symbol \(\widetilde\Pi_T\), not on
\(K\), \(s\), or the particular coefficient \(q\).

By \eqref{eq:frequency-cutoffs}, the symbol
\(\chi_1(\xi/s)\) vanishes for \(|\xi|\le Ks\), and
\(|\chi_1|\le1\). Hence Plancherel's theorem gives
\begin{equation}\label{eq:high-frequency-smoothing}
 \norm{\chi_1(\mathsf D/s)g}_{L^2(\R^3)}
 \le (Ks)^{-1}\norm{\nabla g}_{L^2(\R^3)}
 \le (Ks)^{-1}\norm g_{H^1(\R^3)},
 \qquad g\in H^1(\R^3).
\end{equation}
Combining \eqref{eq:fixed-transverse-commutator} and
\eqref{eq:high-frequency-smoothing}, and using \(K\ge2\), yields
\begin{equation}\label{eq:high-PiT-commutator}
 \norm{\chi_1(\mathsf D/s)
 [\widetilde\Pi_T(\mathsf D),q]U}_{L^2}
 \le C M(Ks)^{-1}\norm U_{L^2}
 \le C M s^{-1}\norm U_{L^2}.
\end{equation}
Finally, \eqref{eq:TT-commutator-decomposition},
\eqref{eq:T0-commutator}, and \eqref{eq:high-PiT-commutator}
prove \eqref{eq:T-comm}.
\end{proof}

On \((0,T)\), where \(\sigma=1\), multiplying the second equation
in the temporally localized system
\eqref{eq:temporally-localized-system} by \(e^{s\phi}\) gives
\begin{equation}\label{eq:P-eq}
 P_t+2s\gamma tP+\beta qP-\nabla A+sA\nabla\psi
 =
 \mathcal R.
\end{equation}
In Lemmas~\ref{lem:low}--\ref{lem:transverse},
\(\mathcal F_s\) and \(\mathcal G_s\) denote the weighted source
quantities defined in \eqref{eq:weighted-norms}, before temporal
localization, and \(\mathcal L_s\) denotes the terminal remainder defined
in \eqref{eq:terminal-remainder}.

For the low-frequency component \(T_0P\), the Fourier support condition \eqref{eq:frequency-cutoffs} gives
\[
\norm{T_0\nabla A}_{L^2}
\le 2Ks\norm A_{L^2}.
\]
We apply the conjugated Volterra estimate \eqref{eq:Volterra} to the evolution equation for \(T_0P\) obtained by applying \(T_0\) to \eqref{eq:P-eq}, as detailed in the proof of Lemma~\ref{lem:low}.

\begin{lemma}
\label{lem:low}
Assume the hypotheses of Lemma~\ref{lem:Volterra-reduction}. For every
fixed \(K\ge2\), there exists \(C_K>0\), independent of \(s\) and of the
particular coefficient \(q\in\Qad\), such that
\begin{equation}\label{eq:low}
 \norm{T_0P}_{L^2}^2
 \le
 \frac{C_K}{s}\norm P_{L^2}^2
 +C_K\mathcal F_s
 +C_Ks^{-2}\mathcal G_s
 +C_K\mathcal L_s
\end{equation}
for every \(s\ge s_{\rm B}\).
\end{lemma}

\begin{proof}
Applying \(T_0\) to \eqref{eq:P-eq} gives
\[
\begin{aligned}
 (T_0P)_t
 +(2s\gamma t+\beta q)T_0P
 ={}&
 T_0(\mathcal R+\nabla A-sA\nabla\psi)-\beta[T_0,q]P.
\end{aligned}
\]
By the localized initial conditions
\eqref{eq:temporally-localized-traces},
\[
 P(0)=e^{s\psi}\pi_\eta(0)=0,
\]
and hence \(T_0P(0)=0\). We may therefore apply the conjugated form of
Lemma~\ref{lem:Volterra} with
\[
 \mathbf U=T_0P,
 \qquad
 B=\beta qI_3.
\]

On the support of \(T_0\),
\[
 |\xi|\le2Ks,
\]
so
\[
 \norm{T_0\nabla A}_{L^2}
 \le
 2Ks\norm A_{L^2}.
\]
Since \(\nabla\psi\) is bounded on the fixed spatial support and
\(T_0\) is uniformly bounded on \(L^2\),
\[
 \norm{sT_0(A\nabla\psi)}_{L^2}
 \le
 Cs\norm A_{L^2}.
\]
Together with \eqref{eq:T-comm}, the Volterra estimate gives
\begin{equation}\label{eq:low-preliminary}
 \norm{T_0P}_{L^2}^2
 \le
 C_Ks\norm A_{L^2}^2
 +\frac Cs\norm{\mathcal R}_{L^2}^2
 +\frac{C_K}{s^3}\norm P_{L^2}^2.
\end{equation}
Substituting \eqref{eq:negative-wave2} and
\eqref{eq:R-est} into \eqref{eq:low-preliminary}, and using \(s\ge1\),
gives \eqref{eq:low}.
\end{proof}

The high-frequency longitudinal component \(T_LP\), with \(T_L\) defined in \eqref{eq:Helmholtz-multipliers}, is estimated using the divergence identity \eqref{eq:H-divP}. Since \(\chi_1(\xi/s)\) is supported where \(|\xi|\ge Ks\) by \eqref{eq:frequency-cutoffs}, the Fourier multiplier \(s\nabla\Delta^{-1}\chi_1(\mathsf D/s)\) has \(L^2\)-operator norm bounded by \(C/K\). Consequently, the term \(s\nabla\psi\cdot P\) contributes at most \(CK^{-2}\|P\|_{L^2}^2\) to the squared estimate for \(T_LP\), with \(C\) independent of \(K\). This \(K^{-2}\)-term is absorbed when passing from \eqref{eq:P-pre-absorption} to \eqref{eq:closed-P-estimate} in the proof of Proposition~\ref{prop:acoustic}.

\begin{lemma}
\label{lem:longitudinal}
Assume the hypotheses of Lemma~\ref{lem:Volterra-reduction}. There exists
\(C>0\), independent of \(K\ge2\), \(s\ge s_{\rm B}\), and the
particular coefficient \(q\in\Qad\), such that
\begin{equation}\label{eq:longitudinal}
 \norm{T_LP}_{L^2}^2
 \le
 \left(
 \frac C{K^2}+\frac C{s^2}
 \right)\norm P_{L^2}^2
 +\frac Cs\mathcal F_s
 +Cs^{-2}\mathcal G_s
 +C\mathcal L_s
\end{equation}
for every \(K\ge2\) and \(s\ge s_{\rm B}\).
\end{lemma}

\begin{proof}
By \eqref{eq:H-divP},
\[
 \diver P
 =
 H+s\nabla\psi\cdot P-e^{s\phi}G_\eta.
\]
Applying the high-frequency longitudinal projection in
\eqref{eq:Helmholtz-multipliers} gives
\[
\begin{aligned}
 T_LP
 ={}&
 \nabla\Delta^{-1}\chi_1(\mathsf D/s)H\\
 &+
 s\nabla\Delta^{-1}\chi_1(\mathsf D/s)
 (\nabla\psi\cdot P)\\
 &-
 \nabla\Delta^{-1}\chi_1(\mathsf D/s)
 (e^{s\phi}G_\eta).
\end{aligned}
\]
By \eqref{eq:frequency-cutoffs},
\[
 \chi_1(\xi/s)\ne0
 \quad\Longrightarrow\quad
 |\xi|\ge Ks.
\]
On the region \(\lvert\xi\rvert\ge Ks\), the multipliers associated
with the \(H\), \(s\nabla\psi\cdot P\), and \(e^{s\phi}G_\eta\) terms
in the preceding decomposition have the respective bounds
\[
 C(s^2+|\xi|^2)^{-1/2},
 \qquad
 \frac C K,
 \qquad
 \frac C{Ks}.
\]
For the first multiplier, indeed,
\[
 \frac{(s^2+|\xi|^2)^{1/2}}{|\xi|}
 \le
 (1+K^{-2})^{1/2}
 \le
 \frac{\sqrt5}{2},
 \qquad K\ge2.
\]
Hence
\begin{equation}\label{eq:longitudinal-preliminary}
 \norm{T_LP}_{L^2}^2
 \le
 C\norm{J_s^0H}_{L^2}^2
 +\frac C{K^2}\norm P_{L^2}^2
 +\frac C{K^2s^2}
 \norm{e^{s\phi}G_\eta}_{L^2}^2.
\end{equation}
Dividing \eqref{eq:negative-wave2} by \(s\) and using
\eqref{eq:localized-G-bound} for the \(G_\eta\) term in \eqref{eq:longitudinal-preliminary} gives
\eqref{eq:longitudinal}. The additional terminal contribution is
absorbed into \(C\mathcal L_s\), since \(K\ge2\) and \(s\ge1\).
\end{proof}

For the high-frequency transverse component,
\[
 T_T\nabla A=0
\]
by \eqref{eq:decomp}. Thus the derivative term \(\nabla A\) in
\eqref{eq:P-eq} disappears after applying \(T_T\), and the conjugated
Volterra estimate can again be used.

\begin{lemma}\label{lem:transverse}
Assume the hypotheses of Lemma~\ref{lem:Volterra-reduction}. For every
fixed \(K\ge2\), there exists \(C_K>0\), independent of \(s\) and of the
particular coefficient \(q\in\Qad\), such that
\begin{equation}\label{eq:transverse}
 \norm{T_TP}_{L^2}^2
 \le
 \frac{C_K}{s}\norm P_{L^2}^2
 +C_K\mathcal F_s
 +C_Ks^{-2}\mathcal G_s
 +C_K\mathcal L_s
\end{equation}
for every \(s\ge s_{\rm B}\).
\end{lemma}

\begin{proof}
Applying \(T_T\) to \eqref{eq:P-eq} and using
\(T_T\nabla A=0\) gives
\[
\begin{aligned}
 (T_TP)_t
 +(2s\gamma t+\beta q)T_TP
 ={}&
 T_T(\mathcal R-sA\nabla\psi)\\
 &-\beta[T_T,q]P.
\end{aligned}
\]
The initial condition \(P(0)=0\), obtained from
\eqref{eq:temporally-localized-traces}, gives \(T_TP(0)=0\). The conjugated form of
Lemma~\ref{lem:Volterra}, together with \eqref{eq:T-comm} and the
boundedness of \(\nabla\psi\), gives
\begin{equation}\label{eq:transverse-preliminary}
 \norm{T_TP}_{L^2}^2
 \le
 Cs\norm A_{L^2}^2
 +\frac Cs\norm{\mathcal R}_{L^2}^2
 +\frac{C_K}{s^3}\norm P_{L^2}^2.
\end{equation}
Substituting \eqref{eq:negative-wave2} and
\eqref{eq:R-est} into \eqref{eq:transverse-preliminary}, and using
\(s\ge1\), gives \eqref{eq:transverse}.
\end{proof}

The decomposition \eqref{eq:decomp} and
Lemmas~\ref{lem:low}--\ref{lem:transverse} now provide estimates for all
three components of \(P\). The coefficient of
\(\norm P_{L^2}^2\) consists of a \(K^{-2}\)-term whose constant is
independent of \(K\), together with terms of order \(C_Ks^{-1}\) and
\(Cs^{-2}\). This structure determines the order of the final parameter
choice: \(K\) is fixed first, and \(s\) is then chosen sufficiently large.

\begin{proposition}
\label{prop:acoustic}
Assume \eqref{eq:acoustic-support},
\eqref{eq:general-Lq}, and
\eqref{eq:acoustic-regularity}. Let \(a_\eta,\pi_\eta,p_\eta\) be the temporally localized
acoustic fields defined in \eqref{eq:eta-fields} and
\eqref{eq:localized-p-F}. Let
\(\mathcal F_s,\mathcal G_s\) be the weighted source quantities defined in
\eqref{eq:weighted-norms}, before temporal localization, and let
\(\mathcal L_s\) be the terminal remainder defined in
\eqref{eq:terminal-remainder}.

There exist \(s_0\ge s_{\rm B}\) and \(C>0\), depending only on
\[
 \Omega_2,\ T,\ \gamma,\ \delta,\ x_0,\ \beta,\ M,
\]
the fixed compact set \(\mathcal K\), and the cutoff functions
\(\widetilde\chi,\chi,\eta\) chosen in
\eqref{eq:Bessel-operators} and \eqref{eq:temporal-cutoff}, such that
\begin{equation}\label{eq:acoustic-Carleman}
 \norm{\pi_\eta}_s^2
 +s\norm{a_\eta}_s^2
 +s\norm{(a_\eta)_t}_{-1,s}^2
 +\norm{p_\eta}_s^2
 \le
 C\mathcal F_s
 +Cs^{-2}\mathcal G_s
 +C\mathcal L_s,
 \qquad
 s\ge s_0.
\end{equation}
The constants \(s_0\) and \(C\) are independent of the particular
coefficient \(q\in\Qad\) and of the source fields \(F,G\).
\end{proposition}

\begin{proof}
By \eqref{eq:decomp} and
Lemmas~\ref{lem:low}--\ref{lem:transverse},
\begin{equation}\label{eq:P-pre-absorption}
 \norm P_{L^2}^2
 \le
 3\left(
 \frac C{K^2}
 +\frac{C_K}s
 +\frac C{s^2}
 \right)
 \norm P_{L^2}^2
 +C_K\mathcal F_s
 +C_Ks^{-2}\mathcal G_s
 +C_K\mathcal L_s.
\end{equation}
First choose \(K\ge2\) sufficiently large that
\[
 \frac{3C}{K^2}\le\frac14.
\]
The constant \(C\) in this inequality is independent of \(K\), so the
selected value of \(K\) depends only on the fixed a priori parameters.
With this value of \(K\) fixed, choose \(s_0\ge s_{\rm B}\) sufficiently
large that
\[
 3\left(
 \frac{C_K}s+\frac C{s^2}
 \right)
 \le\frac14,
 \qquad s\ge s_0.
\]
The coefficient of \(\norm P_{L^2}^2\) on the right-hand side of
\eqref{eq:P-pre-absorption} is then at most \(1/2\), and therefore
\begin{equation}\label{eq:closed-P-estimate}
 \norm P_{L^2}^2
 \le
 C\mathcal F_s
 +Cs^{-2}\mathcal G_s
 +C\mathcal L_s,
 \qquad s\ge s_0.
\end{equation}
The order of the choices \(K\) first and \(s\) second makes the absorption
uniform over \(q\in\Qad\); no smallness assumption on \(q\) is used.

Using \eqref{eq:closed-P-estimate} in
\eqref{eq:negative-wave2} and \eqref{eq:p-est}, together with
\[
\begin{aligned}
 \norm A_{L^2(\R^3\times(-T,T))}^2
 &=2\norm{a_\eta}_s^2,\\
 \norm P_{L^2(\R^3\times(-T,T))}^2
 &=2\norm{\pi_\eta}_s^2,\\
 \norm{J_s^0H}_{L^2(\R^3\times(-T,T))}^2
 &=2\norm{(a_\eta)_t}_{-1,s}^2.
\end{aligned}
\]
gives \eqref{eq:acoustic-Carleman}.
\end{proof}

\begin{remark}\label{rem:divergence-forcing}
The right-hand side of \eqref{eq:acoustic-Carleman} contains the weighted
\(L^2\)-norms of \(F\) and \(F_t\), through
\(\mathcal F_s\) defined in \eqref{eq:weighted-norms}, but contains no
spatial derivative of \(F\). For the coefficient-dependent source \(F^{\rm coef}\) defined in
\eqref{eq:v-source}, the estimate controls the source contribution in
terms of the weighted \(L^2(D)\)-norm of \(r=q_1-q_2\), without
introducing spatial derivatives of \(r\).
The bound \eqref{eq:Fcoef} for \(F^{\rm coef}\) is proved in
Section~\ref{sec:nonlinear}.
\end{remark}

Although the inverse problem is formulated for \(q\in\Qad\), the
commutator estimates \eqref{eq:Jq-comm} and \eqref{eq:T-comm} use only the
\(W^{1,\infty}\)-bound on \(q\). The higher coefficient regularity in
\eqref{eq:Qad} is required by the forward estimates in
Proposition~\ref{prop:direct}, which provide the field regularity
\eqref{eq:acoustic-regularity} used throughout this section.


\subsection{Recovery of the initial trace}

Proposition~\ref{prop:acoustic} gives weighted space-time control of the
acoustic fields. By \eqref{eq:general-Lq}, \eqref{eq:acoustic-initial},
and \eqref{eq:temporally-localized-traces}, the initial trace \(h\) satisfies
\begin{equation}\label{eq:acoustic-initial-trace-identification}
 h=v_{tt}(\cdot,0)=a(\cdot,0)=a_\eta(\cdot,0).
\end{equation}
The following proposition obtains a weighted bound for the initial trace
in \eqref{eq:acoustic-initial-trace-identification} from
\eqref{eq:acoustic-Carleman} by an energy identity.

\begin{proposition}\label{prop:initial-trace}
Assume the hypotheses and notation of
Proposition~\ref{prop:acoustic}. Let \(h\) be the initial trace
identified in \eqref{eq:acoustic-initial-trace-identification}, and let
\(\psi\) be the spatial weight defined in \eqref{eq:psi}. After increasing
\(s_0\) if necessary, there exists \(C>0\), with the same dependence on the
fixed a priori parameters as in Proposition~\ref{prop:acoustic}, such that
\begin{equation}\label{eq:initial-trace}
 \norm{e^{s\psi}h}_{L^2(\R^3)}^2
 \le
 Cs\bigl(
 \mathcal F_s+\mathcal G_s+\mathcal L_s
 \bigr),
 \qquad s\ge s_0.
\end{equation}
\end{proposition}

\begin{proof}
For the fields \(a_\eta,\pi_\eta\) in \eqref{eq:eta-fields}, define
\begin{equation}\label{eq:initial-energy}
 E_s(t)
 =
 \int_{\R^3}
 e^{2s\phi(x,t)}
 \bigl(
 |a_\eta|^2+|\pi_\eta|^2
 \bigr)\dd x.
\end{equation}
The regularity \eqref{eq:acoustic-field-regularity} and the cutoff
definitions \eqref{eq:eta-fields} imply
\[
 a_\eta,\pi_\eta\in W^{1,2}(0,T;L^2(\R^3)).
\]
Their spatial supports are contained in the fixed compact set \(\mathcal K\Subset\Omega_2\) from \eqref{eq:acoustic-support}. For each fixed \(s\ge1\), the functions \(e^{s\phi}\) and \(\partial_t(e^{s\phi})\) are bounded on \(\mathcal K\times[0,T]\). Hence
\[
e^{s\phi}a_\eta,\;
e^{s\phi}\pi_\eta
\in H^1(0,T;L^2(\R^3)).
\]
The Hilbert-space product rule applied to \eqref{eq:initial-energy} gives \(E_s\in W^{1,1}(0,T)\), with
\[
\begin{aligned}
E_s'(t)
={}&2\int_{\R^3}e^{2s\phi}
\bigl(a_\eta(a_\eta)_t+
\pi_\eta\cdot(\pi_\eta)_t\bigr)\dd x+2s\int_{\R^3}\phi_t e^{2s\phi}
\bigl(|a_\eta|^2+|\pi_\eta|^2\bigr)\dd x
\end{aligned}
\]
for almost every \(t\in(0,T)\).

By \eqref{eq:acoustic-field-regularity} and the first equation in \eqref{eq:temporally-localized-system}, for almost every \(t\in(0,T)\),
\[
\begin{aligned}
&a_\eta(t)\in H^1(\R^3),\qquad
\pi_\eta(t)\in L^2(\R^3;\R^3),\\
&\diver\pi_\eta(t)
=(a_\eta)_t(t)-G_\eta(t)\in L^2(\R^3).
\end{aligned}
\]
Thus \(\pi_\eta(t)\in H(\diver;\R^3)\). Moreover, the compact spatial support condition \eqref{eq:acoustic-support} and the definition \eqref{eq:eta-fields} imply
\[
e^{2s\phi(\cdot,t)}a_\eta(t)\in H^1(\R^3).
\]
Applying the \(H^1\)--\(H(\diver)\) integration-by-parts identity on \(\R^3\), and using \(\nabla_x\phi=\nabla\psi\) from \eqref{eq:phi}, we obtain
\[
\begin{aligned}
\int_{\R^3}e^{2s\phi}a_\eta\diver\pi_\eta\dd x
={}&-\int_{\R^3}e^{2s\phi}
\pi_\eta\cdot\nabla a_\eta\dd x -2s\int_{\R^3}e^{2s\phi}
a_\eta\nabla\psi\cdot\pi_\eta\dd x.
\end{aligned}
\]
No spatial boundary term occurs because the integration-by-parts identity is applied on the whole space \(\R^3\).

Since the temporal cutoff \(\eta\) defined in \eqref{eq:temporal-cutoff} vanishes near \(t=T\), the fields \(a_\eta\) and \(\pi_\eta\) vanish there. Consequently,
\[
E_s(T)=0.
\]
At \(t=0\), the initial conditions \eqref{eq:temporally-localized-traces} give
\[
a_\eta(\cdot,0)=h,\qquad
\pi_\eta(\cdot,0)=0.
\]
Since \(\phi(x,0)=\psi(x)\) by \eqref{eq:phi}, the definition of \(E_s\) in \eqref{eq:initial-energy} yields
\[
E_s(0)=\norm{e^{s\psi}h}_{L^2(\R^3)}^2.
\]

Using the temporally localized system
\eqref{eq:temporally-localized-system} and differentiating
\eqref{eq:initial-energy} gives
\begin{equation}\label{eq:initial-energy-identity}
\begin{aligned}
 E_s'(t)
 ={}&
 -4s\gamma tE_s(t)
 -4s\int
 e^{2s\phi}
 a_\eta\nabla\psi\cdot\pi_\eta\dd x\\
 &-2\beta\int
 e^{2s\phi}q|\pi_\eta|^2\dd x
 +2\int
 e^{2s\phi}
 (a_\eta G_\eta+\pi_\eta\cdot R_\eta)\dd x.
\end{aligned}
\end{equation}
Integrating \eqref{eq:initial-energy-identity} from \(0\) to \(T\) and
using the endpoint values gives
\begin{equation}\label{eq:initial-energy-integrated}
\begin{aligned}
 E_s(0)
 ={}&
 4s\gamma\int_0^TtE_s(t)\dd t
 +4s\int_{Q_T}
 e^{2s\phi}
 a_\eta\nabla\psi\cdot\pi_\eta\dd x\dd t\\
 &+2\beta\int_{Q_T}
 e^{2s\phi}q|\pi_\eta|^2\dd x\dd t\\
 &-2\int_{Q_T}
 e^{2s\phi}
 (a_\eta G_\eta+\pi_\eta\cdot R_\eta)\dd x\dd t.
\end{aligned}
\end{equation}
Since \(\nabla\psi\) is bounded on the fixed support, Young's inequality
gives
\begin{equation}\label{eq:initial-energy-pointwise}
\begin{aligned}
 4s|a_\eta|\,|\nabla\psi|\,|\pi_\eta|
 &\le
 Cs^2|a_\eta|^2+C|\pi_\eta|^2,\\
 2|a_\eta G_\eta+\pi_\eta\cdot R_\eta|
 &\le
 |a_\eta|^2+|\pi_\eta|^2
 +C\bigl(|G_\eta|^2+|R_\eta|^2\bigr).
\end{aligned}
\end{equation}
Combining \eqref{eq:initial-energy-pointwise} and
\eqref{eq:initial-energy-integrated}, then applying the bound
\eqref{eq:localized-G-bound} for \(G_\eta\), yields
\begin{equation}\label{eq:initial-energy-bound}
 \norm{e^{s\psi}h}_{L^2}^2
 \le
 C\left(
 s^2\norm{a_\eta}_s^2
 +s\norm{\pi_\eta}_s^2
 +\norm{\mathcal R}_{L^2}^2
 +\mathcal G_s
 +\mathcal L_s
 \right).
\end{equation}
Multiplying \eqref{eq:acoustic-Carleman} by \(s\) controls the
\(s^2\norm{a_\eta}_s^2\) and \(s\norm{\pi_\eta}_s^2\) terms in
\eqref{eq:initial-energy-bound}, while \eqref{eq:R-est} controls
\(\mathcal R=e^{s\phi}R_\eta\). Since \(s\ge1\), substitution
of these two estimates into \eqref{eq:initial-energy-bound} gives
\eqref{eq:initial-trace}.
\end{proof}

When Proposition~\ref{prop:initial-trace} is applied through
\eqref{eq:localized-acoustic-identification}, the localized vector source
splits as
\[
 F^{\rm loc}
 =F^{\rm coef}+
 \bigl(F^{\rm loc}-F^{\rm coef}\bigr).
\]
The coefficient-dependent contribution \(F^{\rm coef}\) is estimated in
Section~\ref{sec:nonlinear}, while the cutoff-generated terms
\(F^{\rm loc}-F^{\rm coef}\) and \(G^{\rm loc}\) are controlled by
\eqref{eq:cutoff-weighted}. These bounds are combined in
Section~\ref{sec:main-proof}.



\section{The coefficient-dependent source and a transport estimate}
\label{sec:nonlinear}

Under the identification \eqref{eq:localized-acoustic-identification}, the initial-trace estimate \eqref{eq:initial-trace} applies to the initial source \(h\), whose relation to the coefficient difference \(r\) is given by \eqref{eq:h-direction} and \eqref{eq:localized-initial-source}. The source norm \(\mathcal F_s\) in \eqref{eq:initial-trace} is evaluated with \(F=F^{\rm loc}\), defined in \eqref{eq:localized-fields}, which contains the coefficient-dependent field \(F^{\rm coef}\) defined in \eqref{eq:v-source}. We first establish the weighted estimate \eqref{eq:Fcoef} for \(F^{\rm coef}\), using the vanishing of \(u_{q,t}(\cdot,0)\) and \(u_{q,tt}(\cdot,0)\) in \(D\), as follows from \eqref{eq:affine-f} and \eqref{eq:forward-initial-traces}. We then prove Lemma~\ref{lem:transport}, which estimates \(r\) in terms of its directional derivative \(\partial_\omega r\).

For the coefficient-dependent source \(F^{\rm coef}\) defined in \eqref{eq:v-source}, set
\begin{equation}\label{eq:coefficient-source-norm}
\mathcal F_s^{\rm coef}
=
\sum_{j=0}^1
\norm{e^{s\phi}\partial_t^jF^{\rm coef}}_{L^2(D\times(0,T))}^2.
\end{equation}

\begin{lemma}
\label{lem:coefficient-source}
Let \(q_1,q_2\in\Qad\), with the coefficient difference \(r\) defined in \eqref{eq:difference-variables}. There exists a constant \(C>0\), independent of \(q_1,q_2\) and \(s\), such that, for every \(s\ge1\),
\begin{equation}\label{eq:Fcoef}
\mathcal F_s^{\rm coef}
\le
C\Mdir^2s^{-3/2}
\norm{e^{s\psi}r}_{L^2(D)}^2.
\end{equation}
\end{lemma}

\begin{proof}
For every \(q\in\Qad\), the affine condition \eqref{eq:affine-f} and the forward initial-trace identities \eqref{eq:forward-initial-traces} give
\[
u_{q,t}(\cdot,0)=0,
\qquad
u_{q,tt}(\cdot,0)=\Delta f=0
\quad\text{in }D.
\]
Proposition~\ref{prop:direct} and the Sobolev embedding
\(H^3(\R^3)\hookrightarrow W^{1,\infty}(\R^3)\) imply
\[
\norm{u_{q,ttt}}_{L^\infty(0,T;W^{1,\infty}(\R^3))}
\le C\Mdir.
\]

The regularity in \eqref{eq:Xclass} permits integration of the time derivatives in \(H^3(D)\). Since both initial traces vanish in \(D\), we have
\[
\begin{aligned}
u_{q,tt}(x,t)
&=\int_0^t u_{q,ttt}(x,\tau)\dd\tau,\quad u_{q,t}(x,t)
=\int_0^t(t-\tau)u_{q,ttt}(x,\tau)\dd\tau,
\end{aligned}
\]
for almost every \(x\in D\) and \(0\le t\le T\). Consequently,
\begin{equation}\label{eq:time-vanish}
\begin{aligned}
\norm{u_{q,tt}(t)}_{W^{1,\infty}(D)}
&\le C\Mdir t,\quad \norm{u_{q,t}(t)}_{W^{1,\infty}(D)}
\le C\Mdir t^2.
\end{aligned}
\end{equation}

Applying \eqref{eq:time-vanish} with \(q=q_2\) to the source \(F^{\rm coef}\) defined in \eqref{eq:v-source}, we obtain
\begin{equation}\label{eq:source-pointwise}
\begin{aligned}
|F^{\rm coef}(x,t)|
&\le C\Mdir t^2|r(x)|,\quad 
|\partial_tF^{\rm coef}(x,t)|
\le C\Mdir t|r(x)|
\end{aligned}
\end{equation}
for almost every \(x\in D\) and \(0<t<T\).

For \(m>-1\), the change of variables
\(\vartheta=(2s\gamma)^{1/2}t\) gives
\begin{equation}\label{eq:Gaussian-moment}
\int_0^\infty t^m e^{-2s\gamma t^2}\dd t
=C_{m,\gamma}s^{-(m+1)/2}.
\end{equation}
Using the weight \(\phi\) defined in \eqref{eq:phi}, the pointwise bounds \eqref{eq:source-pointwise}, and the definition \eqref{eq:coefficient-source-norm}, we obtain
\[
\begin{aligned}
\mathcal F_s^{\rm coef}
&\le C\Mdir^2
\norm{e^{s\psi}r}_{L^2(D)}^2
\int_0^T(t^2+t^4)e^{-2s\gamma t^2}\dd t\\
&\le C\Mdir^2
\left(s^{-3/2}+s^{-5/2}\right)
\norm{e^{s\psi}r}_{L^2(D)}^2.
\end{aligned}
\]
Since \(s\ge1\), this proves \eqref{eq:Fcoef}.
\end{proof}

The estimate \eqref{eq:Fcoef} controls the coefficient-dependent source by the weighted \(L^2(D)\) norm of \(r\), without requiring a spatial derivative of \(r\). To estimate \(r\) in terms of its directional derivative, we next use the support condition in \eqref{eq:Qad} and the monotonicity assumption \eqref{eq:direction}.

\begin{lemma}\label{lem:transport}
Suppose that $r\in H_0^1(D)$ and $\omega\in\Sph^2$,
and that $\psi$ is the quadratic weight \eqref{eq:psi}.
Assume \eqref{eq:direction} with $c_*>0$.
Then
\begin{equation}\label{eq:transport}
 s^2\norm{e^{s\psi}r}_{L^2(D)}^2
 \le c_*^{-2}
 \norm{e^{s\psi}\partial_\omega r}_{L^2(D)}^2,
 \qquad s>0.
\end{equation}
\end{lemma}

\begin{proof}
Set $r_s=e^{s\psi}r\in H_0^1(D)$. The calculation below is first
performed for $r_s\in C_c^\infty(D)$ and then extended by density in
$H_0^1(D)$. Since
\[
 e^{s\psi}\partial_\omega r
 =\partial_\omega r_s-s(\partial_\omega\psi)r_s,
\]
squaring and integrating over $D$, then integrating the cross term
by parts in the constant direction $\omega$, gives
\[
\begin{aligned}
-2s\int_D(\partial_\omega\psi)r_s\,\partial_\omega r_s\dd x
&=-s\int_D(\partial_\omega\psi)
        \partial_\omega(|r_s|^2)\dd x\\
&=s\int_D(\partial_\omega^2\psi)|r_s|^2\dd x.
\end{aligned}
\]
The boundary term is zero because $r_s$ has zero trace on $\partial D$.
Hence
\begin{equation}\label{eq:transport-energy-identity}
\begin{aligned}
 \norm{e^{s\psi}\partial_\omega r}_{L^2(D)}^2
={}&\norm{\partial_\omega r_s}_{L^2(D)}^2
 +s^2\norm{(\partial_\omega\psi)r_s}_{L^2(D)}^2 \\
&+s\int_D(\partial_\omega^2\psi)|r_s|^2\dd x.
\end{aligned}
\end{equation}
For the weight \eqref{eq:psi},
$\partial_\omega^2\psi=2|\omega|^2=2$.
Using \eqref{eq:direction} in \eqref{eq:transport-energy-identity}
and discarding the nonnegative terms gives
\[
 \norm{e^{s\psi}\partial_\omega r}_{L^2(D)}^2
 \ge s^2c_*^2\norm{r_s}_{L^2(D)}^2.
\]
Since $r_s=e^{s\psi}r$, this proves \eqref{eq:transport}.
\end{proof}


\section{Proof of Theorem~\ref{thm:main}}\label{sec:main-proof}

Let \(q_1,q_2\in\Qad\), and let the coefficient difference \(r\), the initial source \(h\), and the measurement discrepancy \(\Dt\) be defined in \eqref{eq:difference-variables}, \eqref{eq:v-initial}, and \eqref{eq:data-discrepancy}, respectively. We apply the initial-trace estimate \eqref{eq:initial-trace} to the localized difference equation \eqref{eq:localized-equation}, with coefficient \(q_1\) and the fields specified in \eqref{eq:localized-acoustic-identification}. The coefficient-dependent source \(F^{\rm coef}\) is estimated by Lemma~\ref{lem:coefficient-source}, while the cutoff-generated terms \(F^{\rm loc}-F^{\rm coef}\) and \(G^{\rm loc}\) are controlled by the measurement discrepancy through \eqref{eq:cutoff-weighted}. The initial-source identity \eqref{eq:h-direction} and Lemma~\ref{lem:transport} then convert the weighted estimate for \(h\) into an estimate for \(r\).

The initial-trace estimate \eqref{eq:initial-trace} also contains the terminal remainder \(\mathcal L_s\), defined in \eqref{eq:terminal-remainder}. We first establish an exponential bound for \(\mathcal L_s\) corresponding to the localized difference fields, using the uniform forward estimate \eqref{eq:Mstar-bounds} and the spatial localization estimates of Section~\ref{sec:compact-localization}. After absorbing the coefficient-source contribution for sufficiently large \(s\), the positive weight gap \eqref{eq:weight-gap} yields the two-exponential inequality \eqref{eq:two-exp}. Optimizing this inequality in the Carleman parameter gives the stability estimate in Theorem~\ref{thm:main}.

\begin{lemma}\label{lem:localized-terminal}
Let \(q_1,q_2\in\Qad\). In the acoustic problem \eqref{eq:general-Lq}, take the coefficient \(q=q_1\) and the fields \(v,F,G\) specified in \eqref{eq:localized-acoustic-identification}. Let \(a,p,\pi\) be defined in \eqref{eq:acoustic-vars}, and let the temporal cutoff \(\eta\) satisfy \eqref{eq:temporal-cutoff}. Then the terminal remainder \(\mathcal L_s\) defined in \eqref{eq:terminal-remainder} satisfies
\begin{equation}\label{eq:Lterminal}
\mathcal L_s
\le
C\exp\bigl(2s\{M_\psi-\gamma(T-2\delta)^2\}\bigr)\Mdir^2,
\qquad s\ge1,
\end{equation}
where \(M_\psi\) is defined in \eqref{eq:weight-levels}. The constant \(C>0\) depends only on the fixed a priori parameters and cutoff functions, and is independent of \(q_1,q_2\) and \(s\).
\end{lemma}

\begin{proof}
By the definition of \(\Qad\) in \eqref{eq:Qad} and the coefficient difference in \eqref{eq:difference-variables},
\[
\norm{r}_{W^{5,\infty}(\R^3)}\le2M.
\]
The coefficient-dependent source \(F^{\rm coef}\) defined in \eqref{eq:v-source} and Proposition~\ref{prop:direct} therefore give
\[
\norm{F^{\rm coef}}_{H^1(0,T;L^2(\Omega_2))}
\le C\Mdir,
\]
where the relaxation rate \(\beta>0\) and the admissible bound \(M\) are fixed.

Moreover, \(W_\beta=D_\beta w\) by \eqref{eq:filtered-difference-field}, with \(w\) defined in \eqref{eq:difference-variables}. The forward regularity estimate \eqref{eq:Mstar-bounds} consequently implies
\[
\norm{W_\beta}_{H^1(0,T;L^2(U_\theta))}
\le C\Mdir.
\]
Combining this bound with the localization-error estimate \eqref{eq:localization-error} yields
\[
\begin{aligned}
&\norm{F^{\rm loc}-F^{\rm coef}}_{H^1(0,T;L^2(\Omega_2))}\\
&\qquad
+\norm{G^{\rm loc}}_{H^1(0,T;L^2(\Omega_2))}
\le C\Mdir.
\end{aligned}
\]
The definitions of the localized fields in \eqref{eq:localized-fields}, the acoustic variables in \eqref{eq:acoustic-vars}, and the forward regularity estimates then give
\begin{equation}\label{eq:terminal-field-bounds}
\begin{aligned}
&\norm{p}_{L^2(Q_T)}
+\norm{\pi}_{L^2(Q_T)}
+\norm{F}_{H^1(0,T;L^2(\Omega_2))}\\
&\qquad
+\norm{a}_{H^1(0,T;L^2(\Omega_2))}
+\norm{G}_{H^1(0,T;L^2(\Omega_2))}
\le C\Mdir.
\end{aligned}
\end{equation}
Here \(a,p,\pi,F,G\) are the acoustic fields and source terms before temporal localization, with the identification \eqref{eq:localized-acoustic-identification}.

By \eqref{eq:temporal-cutoff}, the derivatives \(\eta'\) and \(\eta''\) are bounded, and their supports on the positive time interval are contained in \([T-2\delta,T-\delta]\). The weight \(\phi\) defined in \eqref{eq:phi} consequently satisfies
\[
\phi(x,t)\le
M_\psi-\gamma(T-2\delta)^2
\]
for every \(x\in\Omega_2\) and \(t\in\supp\eta'\cap(0,T)\), with the same bound on the support of \(\eta''\).

Using the definition of \(\mathcal L_s\) in \eqref{eq:terminal-remainder}, the support properties of the cutoff derivatives, and the preceding weight bound, we obtain
\begin{equation}\label{eq:terminal-remainder-unweighted}
\begin{aligned}
\mathcal L_s
\le{}&
C_\eta
e^{2s\{M_\psi-\gamma(T-2\delta)^2\}}
\Bigl(
\norm{p}_{L^2(Q_T)}^2
+\norm{\pi}_{L^2(Q_T)}^2\\
&\qquad
+\norm{F}_{H^1(0,T;L^2(\Omega_2))}^2
+\norm{a}_{H^1(0,T;L^2(\Omega_2))}^2\\
&\qquad
+\norm{G}_{H^1(0,T;L^2(\Omega_2))}^2
\Bigr),
\qquad s\ge1.
\end{aligned}
\end{equation}
The constant \(C_\eta\) depends on the fixed temporal cutoff and the a priori parameters, but not on \(s,q_1,q_2\).

Combining \eqref{eq:terminal-remainder-unweighted} with \eqref{eq:terminal-field-bounds} proves \eqref{eq:Lterminal}.
\end{proof}

\begin{proof}[Proof of Theorem~\ref{thm:main}]
Fix a compact set \(\mathcal K\Subset\Omega_2\) containing
\[
\overline D\cup\supp\theta\cup\supp\rho,
\]
where the spatial cutoffs \(\theta,\rho\) are specified in \eqref{eq:spatial-cutoffs}. The set \(\mathcal K\) is independent of the particular coefficients \(q_1,q_2\in\Qad\).

We apply the general acoustic problem \eqref{eq:general-Lq} with coefficient \(q=q_1\) and the fields specified in \eqref{eq:localized-acoustic-identification}. Proposition~\ref{prop:localized-difference} establishes the localized equation \eqref{eq:localized-equation}, the initial conditions \eqref{eq:localized-initial-source}, and the source compatibility condition \eqref{eq:localized-source-initial}.

The coefficient-dependent source \(F^{\rm coef}\) is supported in \(D\) by \eqref{eq:v-source}. By \eqref{eq:localized-fields}, the remaining source terms are supported in \(\supp\theta\cup\supp\rho\). Hence the support condition \eqref{eq:acoustic-support} holds with the fixed compact set \(\mathcal K\). The regularity condition \eqref{eq:acoustic-regularity} follows from Proposition~\ref{prop:direct}, Lemma~\ref{lem:compact-localization}, and \eqref{eq:localization-error}. Thus the hypotheses of Proposition~\ref{prop:initial-trace} are satisfied.

We next compare the source norms \(\mathcal F_s,\mathcal G_s\) in \eqref{eq:weighted-norms} with the coefficient-dependent and cutoff-generated source norms. Under the identification \eqref{eq:localized-acoustic-identification}, the definitions \eqref{eq:localized-fields}, \eqref{eq:coefficient-source-norm}, and \eqref{eq:cutoff-source-norms} give
\begin{equation}\label{eq:source-norm-decomposition}
\mathcal F_s
\le
2\mathcal F_s^{\rm coef}
+2\mathcal F_s^{\rm cut},
\qquad
\mathcal G_s=\mathcal G_s^{\rm cut}.
\end{equation}
Indeed, the first inequality follows from
\[
F^{\rm loc}
=
F^{\rm coef}+(F^{\rm loc}-F^{\rm coef})
\]
and the elementary bound
$
|X+Y|^2\le2|X|^2+2|Y|^2.
$
The supports of the two summands permit the corresponding integrals to be taken over \(D\) and \(\Omega_2\), respectively, as prescribed in \eqref{eq:coefficient-source-norm} and \eqref{eq:cutoff-source-norms}. The equality for \(\mathcal G_s\) follows from \(G=G^{\rm loc}\) in \eqref{eq:localized-acoustic-identification}.

Applying Proposition~\ref{prop:initial-trace}, together with \eqref{eq:source-norm-decomposition} and the measurement estimate \eqref{eq:cutoff-weighted}, gives
\[
\norm{e^{s\psi}h}_{L^2(D)}^2
\le
Cs\mathcal F_s^{\rm coef}
+Cs e^{C_bs}\Dt^2
+Cs\mathcal L_s
\]
for all sufficiently large \(s\). Here the initial source \(h\) is supported in \(D\) by \eqref{eq:h-direction} and \eqref{eq:Qad}.

We retain the constant \(C_b\) appearing in \eqref{eq:cutoff-weighted}. Since
\[
s e^{C_bs}\le e^{(C_b+1)s},
\qquad s\ge1,
\]
the preceding inequality yields
\begin{equation}\label{eq:initial-trace-with-data}
\norm{e^{s\psi}h}_{L^2(D)}^2
\le
Cs\mathcal F_s^{\rm coef}
+Ce^{(C_b+1)s}\Dt^2
+Cs\mathcal L_s.
\end{equation}

By \eqref{eq:difference-variables} and the support condition in \eqref{eq:Qad}, the coefficient difference \(r\) belongs to \(H_0^1(D)\). Applying Lemma~\ref{lem:transport} and using the initial-source identity \eqref{eq:h-direction}, we obtain
\begin{equation}\label{eq:weighted-r-by-h}
s^2\norm{e^{s\psi}r}_{L^2(D)}^2
\le
C\norm{e^{s\psi}h}_{L^2(D)}^2,
\end{equation}
where \(C\) may depend on the fixed relaxation rate \(\beta>0\), but is independent of \(q_1,q_2\) and \(s\).

Combining \eqref{eq:weighted-r-by-h} with \eqref{eq:initial-trace-with-data} and the coefficient-source estimate \eqref{eq:Fcoef} established in Lemma~\ref{lem:coefficient-source}, we obtain
\begin{equation}\label{eq:main-step2}
\begin{aligned}
s^2\norm{e^{s\psi}r}_{L^2(D)}^2
\le{}&
C_1s^{-1/2}\Mdir^2
\norm{e^{s\psi}r}_{L^2(D)}^2\\
&+Ce^{(C_b+1)s}\Dt^2
+Cs\mathcal L_s.
\end{aligned}
\end{equation}
The factor \(s^{-1/2}\) comes from multiplying the factor \(s\) in \eqref{eq:initial-trace-with-data} by the factor \(s^{-3/2}\) in \eqref{eq:Fcoef}. The constant \(C_1>0\) depends only on the fixed a priori parameters.

To absorb the first term $
C_1s^{-1/2}\Mdir^2
\norm{e^{s\psi}r}_{L^2(D)}^2 
$ on the right-hand side of \eqref{eq:main-step2}, we compare its coefficient \(C_1s^{-1/2}\Mdir^2\) with the factor \(s^2\) on the left-hand side. Their ratio is
\[
C_1\Mdir^2s^{-5/2}.
\]
Since \(\Mdir\) depends only on the fixed a priori data by \eqref{eq:Mstar}, we choose \(s_0\ge1\), no smaller than the threshold in Proposition~\ref{prop:initial-trace}, such that
\begin{equation}\label{eq:nonlinear-absorption-threshold}
C_1\Mdir^2s_0^{-5/2}\le\frac12.
\end{equation}
For every \(s\ge s_0\),
\[
C_1s^{-1/2}\Mdir^2\le\frac12s^2.
\]
It follows from \eqref{eq:main-step2} that
\[
\frac12s^2\norm{e^{s\psi}r}_{L^2(D)}^2
\le
Ce^{(C_b+1)s}\Dt^2+Cs\mathcal L_s.
\]
After enlarging \(C\), we therefore have
\[
s^2\norm{e^{s\psi}r}_{L^2(D)}^2
\le
Ce^{(C_b+1)s}\Dt^2+Cs\mathcal L_s,
\qquad s\ge s_0.
\]
The threshold \(s_0\) depends only on the fixed a priori data and is independent of \(q_1,q_2\) and \(\Dt\).

We now use the terminal remainder estimate \eqref{eq:Lterminal} from Lemma~\ref{lem:localized-terminal}. By the definition of \(m_D\) in \eqref{eq:weight-levels}, we have \(\psi(x)\ge m_D\) for every \(x\in D\). Hence
\[
\norm{e^{s\psi}r}_{L^2(D)}^2
\ge
e^{2sm_D}\norm{r}_{L^2(D)}^2.
\]
Combining this inequality with the absorbed estimate and \eqref{eq:Lterminal} gives
\[
\begin{aligned}
s^2e^{2sm_D}\norm{r}_{L^2(D)}^2
\le{}&
Ce^{(C_b+1)s}\Dt^2\\
&+Cs e^{2s\{M_\psi-\gamma(T-2\delta)^2\}}
\Mdir^2.
\end{aligned}
\]
Dividing by \(s^2e^{2sm_D}\) and using the definition of \(d_*\) in \eqref{eq:weight-levels}, we obtain
\begin{equation}\label{eq:pre-two-exp}
\begin{aligned}
\norm{r}_{L^2(D)}^2
\le{}&
Cs^{-2}e^{(C_b+1-2m_D)s}\Dt^2\\
&+Cs^{-1}e^{-2d_*s}\Mdir^2,
\qquad s\ge s_0.
\end{aligned}
\end{equation}
The observation-time condition \eqref{eq:long-time} ensures that \(d_*>0\).

Choose a fixed constant
\[
A_0>\max\{C_b+1-2m_D,0\}.
\]
Since \(s^{-2}\le1\) and \(s^{-1}\le1\) for \(s\ge1\), inequality \eqref{eq:pre-two-exp} implies
\begin{equation}\label{eq:two-exp}
\norm{r}_{L^2(D)}^2
\le
Ce^{A_0s}\Dt^2
+Ce^{-2d_*s}\Mdir^2,
\qquad s\ge s_0.
\end{equation}
The constants \(A_0,d_*,s_0\), and \(C\) depend only on the fixed a priori data.

If \(\Dt=0\), letting \(s\to\infty\) in \eqref{eq:two-exp} gives
\[
\norm{r}_{L^2(D)}^2=0,
\]
because \(d_*>0\). Hence \(q_1=q_2\) almost everywhere in \(D\), proving uniqueness.

For the stability estimate, suppose first that
\[
0<\Dt<\Mdir.
\]
Define
\begin{equation}\label{eq:optimization-parameters}
s_*=
\frac{2}{A_0+2d_*}
\log\frac{\Mdir}{\Dt},
\qquad
\kappa=
\frac{2d_*}{A_0+2d_*}\in(0,1).
\end{equation}
We distinguish the cases \(s_*\ge s_0\) and \(s_*<s_0\).

If \(s_*\ge s_0\), substituting \(s=s_*\) into \eqref{eq:two-exp} gives two equal contributions on the right-hand side. Indeed, \eqref{eq:optimization-parameters} implies
\[
\begin{aligned}
e^{A_0s_*}\Dt^2
&=e^{-2d_*s_*}\Mdir^2\\
&=\Mdir^{2(1-\kappa)}\Dt^{2\kappa}.
\end{aligned}
\]
Consequently,
\[
\norm{r}_{L^2(D)}^2
\le
C\Mdir^{2(1-\kappa)}\Dt^{2\kappa},
\]
and taking square roots yields
\[
\norm{r}_{L^2(D)}
\le
C\Mdir^{1-\kappa}\Dt^\kappa.
\]

If \(s_*<s_0\), the definition of \(s_*\) gives
\[
\log\frac{\Mdir}{\Dt}
<
\frac{A_0+2d_*}{2}s_0.
\]
Multiplying by \(\kappa\) and using
\(\kappa(A_0+2d_*)=2d_*\), we obtain
\[
\left(\frac{\Dt}{\Mdir}\right)^\kappa
>
e^{-d_*s_0}.
\]
The a priori bound \eqref{eq:r-apriori-Mdir} therefore implies
\[
\begin{aligned}
\norm{r}_{L^2(D)}
&\le\Mdir\\
&\le e^{d_*s_0}
\Mdir^{1-\kappa}\Dt^\kappa.
\end{aligned}
\]

Finally, if \(\Dt\ge\Mdir\), the same a priori bound \eqref{eq:r-apriori-Mdir} gives
\[
\norm{r}_{L^2(D)}
\le\Mdir
\le\Mdir^{1-\kappa}\Dt^\kappa.
\]

Combining these cases with the uniqueness result for \(\Dt=0\), and recalling the definition \(r=q_1-q_2\) in \eqref{eq:difference-variables}, we conclude that
\[
\norm{q_1-q_2}_{L^2(D)}
\le
C\Mdir^{1-\kappa}\Dt^\kappa.
\]
The exponent \(\kappa\in(0,1)\) and the constant \(C>0\), enlarged to include the factor \(e^{d_*s_0}\), depend only on the fixed a priori data.
\end{proof}


\section*{Acknowledgments}

This work was supported by the National Natural Science Foundation of China
(Grant No. 12371422) and by the Fundamental Research Funds for the Central
Universities at Jilin University. The authors used OpenAI ChatGPT to assist with English editing, LaTeX organization, and suggestions concerning the presentation of mathematical arguments. The authors assume responsibility for all content.

\end{document}